%% file: ControlFromHypersurfaces.tex
\documentclass[10pt]{article}
\usepackage[latin1]{inputenc}
\usepackage[T1]{fontenc}
\usepackage{stmaryrd}
\usepackage{amsmath,amsfonts,amssymb,mathrsfs,amsthm}
\usepackage{xcolor}
\usepackage{dsfont}
\usepackage{graphicx}
\usepackage[mathscr]{eucal}
\usepackage{makeidx}
\usepackage{verbatim}
\usepackage{graphics,graphicx}
\usepackage{textcomp}
\usepackage{float}
\usepackage{comment}
\usepackage[colorlinks=true]{hyperref}
\usepackage{txfonts} %% Police qui gagne un peu de place
\usepackage{mathrsfs}

\newcommand{\asref}[2]{GC-$($\ref{#1},#2$)$}
\newcommand{\Id}{\operatorname{Id}}
\newcommand{\Fix}{\operatorname{Fix}}

\input mymacros.tex

\newcommand{\mc}[1]{\mathcal{#1}}
\newcommand{\la}{\left\langle}
\newcommand{\ra}{\right\rangle}
\newcommand\bna{\begin{eqnarray*}} 
\newcommand\ena{\end{eqnarray*}}

\newcommand\bnan{\begin{eqnarray}} 
\newcommand\enan{\end{eqnarray}}

\newcommand\bnp{\begin{proof}} 
\newcommand\enp{\end{proof}}

\newcommand\bneq{\begin{eqnarray*}\left\lbrace \begin{array}{rcl}}
\newcommand\eneq{\end{array} \right.\end{eqnarray*}}
\newcommand\bneqn{\begin{eqnarray}\left\lbrace \begin{array}{rcl}}
\newcommand\eneqn{\end{array} \right.\end{eqnarray}}
\newcommand{\GS}{\mc{G}^{\Sigma}}
\newcommand{\ES}{\bar{\mc{E}}^{\Sigma}}

\newtheorem*{assume*}{\assumenumber}

\providecommand{\assumenumber}{}
\makeatletter
\newenvironment{assume}[2]
 {%
  \renewcommand{\assumenumber}{Assumption GC-$($#1,#2$)$}%
  \begin{assume*}
  \protected@edef\@currentlabel{#1}%
 }
 {%
  \end{assume*}
 }
 \newtheorem{conjecture}{Conjecture}
\makeatother

\newcommand{\e}{\varepsilon}
\newcommand{\TM}{{^bT^*M}}
\newcommand{\TRM}{{^bT^*(\R\times M)}}
\newcommand{\TRMsig}{{^bT_{\R\times \Sigma_0}^*(\R\times M)}}

\newcommand{\sgn}{\operatorname{sgn}}
\newcommand{\Ell}{\operatorname{Ell}}

\numberwithin{equation}{section}
\title{Control from an Interior Hypersurface}
  
\author{Jeffrey Galkowski\footnote{Department of Mathematics, Stanford University, Stanford, CA, USA. Email: jeffrey.galkowski@stanford.edu}, Matthieu L\'eautaud\footnote{\'Ecole Polytechnique, Centre de Math\'ematiques Laurent Schwartz UMR7640,  91128 Palaiseau cedex France. 
Most of this research was done when the second author was in CRM, CNRS UMI 3457, Universit\'e de Montr\'eal, Case Postale 6128, Succursale Centre-Ville, Montr\'eal (QC) Canada H3C 3J7 and Universit\'e Paris Diderot, IMJ-PRG, UMR 7586, B\^atiment Sophie Germain, 75205 Paris Cedex 13 France. Email: matthieu.leautaud@polytechnique.edu.} 
}

\begin{document}

\maketitle
\begin{abstract}
We consider a compact Riemannian manifold $M$ (possibly with boundary) and $\Sigma \subset M\setminus \partial M$ an interior hypersurface (possibly with boundary). We study observation and control from $\Sigma$ for both the wave and heat equations. For the wave equation, we prove controllability from $\Sigma$ in time $T$ under the assumption $(\mc{T}GCC)$ that all generalized bicharacteristics intersect $\Sigma$ transversally in the time interval $(0,T)$. For the heat equation we prove unconditional controllability from $\Sigma$. As a result, we obtain uniform lower bounds for the Cauchy data of Laplace eigenfunctions on $\Sigma$ under $\mc{T}GCC$ and unconditional exponential lower bounds on such Cauchy data.
\end{abstract}

\setcounter{tocdepth}{1}
\tableofcontents

%%%%%%%%%%%%%%%%%%%%%%%%%%%%%%%%
% Section                      %
%%%%%%%%%%%%%%%%%%%%%%%%%%%%%%%%

\section{Introduction}

Let $(M,g)$ be a compact $n$ dimensional Riemannian manifold possibly with boundary $\d M$ and denote $\Delta_g$ the (non-positive) Laplace-Beltrami operator on $M$. We study the observability and controllability questions from interior hypersurfaces in $M$. 
 
{To motivate the more involved developments in control theory, let us start by stating (slightly informally) the counterpart of our observability/controllability results for lower bounds for eigenfunctions, i.e. solutions to 
 \begin{equation}
 \label{e:eigenfunction}
 (-\Delta_g -\lambda^2) \phi=0,\qquad \phi |_{\partial M}=0.
\end{equation}
 } For more precision, see Section~\ref{s:intro-eigfct}.
{
 \begin{theorem}
 \label{t:prelimEig}
Assume $M$ is connected and let $\Sigma$ be a nonempty interior hypersurface. Then there exists $c>0$ so that for all $\lambda \geq 0$ and $\phi \in L^2(M)$ solutions to~\eqref{e:eigenfunction}, we have
\begin{equation}
\label{e:genLow}
\|\phi|_{\Sigma}\|_{L^2(\Sigma)}+\|\partial_\nu \phi|_{\Sigma}\|_{L^2(\Sigma)} \geq c e^{-\lambda/c} \|\phi \|_{L^2(M)}.
 \end{equation}
Furthermore, if we assume that all generalized geodesics of some finite length cross $\Sigma$ transversally, then there is $c>0$ so that for all $\lambda \geq 0$ and $\phi \in L^2(M)$ solutions to~\eqref{e:eigenfunction}, we have
\begin{equation}
\label{e:dynLow}
\|\phi|_{\Sigma}\|_{L^2(\Sigma)}+\|\langle \lambda \rangle^{-1}\partial_\nu \phi|_{\Sigma}\|_{L^2(\Sigma)} \geq c  \|\phi\|_{L^2(M)} .
 \end{equation}
 \end{theorem}
 Here, we write $\langle \lambda\rangle: =(1+|\lambda|^2)^{1/2}$.
Generalized geodesics are usual geodesics of $g$ inside $\Int(M)$, and reflect on $\d M$ according to laws of geometric optics (see below).}
 As far as the authors are aware~\eqref{e:genLow} is the first general lower bound to appear for restrictions of Laplace eigenfunctions to hypersurfaces and~\eqref{e:dynLow} is the first uniform lower bound for such restrictions without either taking a full density subsequence of eigenfunctions or imposing restrictive assumptions on $M$. We will prove Theorem~\ref{t:prelimEig} in the process of studying controllability for the heat and wave equations from interior hypersurfaces. Because of this, we postpone further discussion of Theorem~\ref{t:prelimEig} {(including optimality of~\eqref{e:genLow} and~\eqref{e:dynLow})} to Section~\ref{s:intro-eigfct}.

We define interior hypersurfaces as follows:
\begin{definition}
\label{def:hypersurface}
We say that $\Sigma$ is a {\em hypersurface} of $M$ if there is $\Sigma_0$ a compact embedded submanifold of $M$ of dimension $n-1$, possibly with boundary, such that $\Sigma$ is the closure of an open subset of $\Sigma_0$. The manifold $\Sigma_0$ shall be referred to as an extension of $\Sigma$. 
\begin{itemize}
\item[-]We say that $\Sigma$ is an {\em interior hypersurface} if moreover $\Sigma \subset \Int(\Sigma_0) \subset \Int(M)$.
\item[-]We say that $\Sigma$ is a {\em compact interior hypersurface} if it is a compact embedded submanifold of $\Int(M)$ of dimension $n-1$, without boundary.
\item[-]We say that $\Sigma$ is cooriented if $\Sigma_0$ is (i.e. the normal bundle $T_{\Sigma_0} M/T\Sigma_0$ is an orientable vector bundle)\footnote{In case $M$ is oriented, note that $\Sigma_0$ is cooriented iff it is oriented. However, if $M$ is not orientable, $\Sigma_0$ might be orientable without being coorientable, and {\em vice versa}.}. If not mentioned, all hypersurfaces considered in this paper are assumed to be coorientable.
\end{itemize}
\end{definition}
\noindent 
In the particular case $\Sigma$ is a compact interior hypersurface, then it is an interior hypersurface with $\Sigma_0 = \Sigma$. 
Since  $M$ is endowed with a Riemannian structure, the coorientability assumption is equivalent to that of having a smooth global vector field $\d_\nu$ normal to $\Int(\Sigma_0)$. Note that the coorientability condition can be slightly relaxed, see the discussion in Section~\ref{s:finite-union} below.

\bigskip
Given an interior hypersurface $\Sigma$, {the main goal of this paper is to study} the controllability of some evolution equations with a control force of the form 
\begin{equation}
\label{e:Dist-cont}
f_0 \delta_\Sigma+f_1 \delta'_{\Sigma},
\end{equation} where the distributions $f_0\delta_\Sigma$ and $f_1\delta_{\Sigma}'$ are defined by
\begin{equation}
\label{e:defDist}
\langle f_0\delta_\Sigma,\varphi\rangle =\int_{\Sigma}f_0\varphi d\sigma,\qquad \langle f_1\delta'_\Sigma, \varphi\rangle =-\int_{\Sigma}f_1\partial_\nu \varphi  d\sigma .
\end{equation}
In this expression $\sigma$ denotes the Riemannian surface measure on $\Sigma$ induced by the metric $g$ on $M$. 
This contrasts with usual control problems for PDEs, for which the control function appears in the equation: 
\begin{itemize}
\item either as a localized right handside (distributed or internal control) $\mathds{1}_\omega f$, where $\omega$ is an open subset of $M$, and typically, the control function $f$ is in $L^2((0,T)\times \omega)$;
\item or, in case $\d M\neq \emptyset$, as a localized boundary term, e.g. under the form $u|_{\d M} = \mathds{1}_\Gamma f$, where $\Gamma$ is an open subset of $\d M$, and typically, the control function $f$ is in $L^2((0,T)\times \Gamma)$ (here, $u$ denotes the function to be controlled).
\end{itemize}
Concerning the wave equation, the main result is the Bardos-Lebeau-Rauch Theorem~\cite{BLR:92,BG:97} providing a necessary and sufficient condition for the exact controllability with such control forces (see also e.g.~\cite{DL:09,LL:16,LRLTT:16} for recent developments).
Concerning the heat equation, the question of null-controllability with internal or boundary control was solved independently by Lebeau-Robbiano~\cite{LR:95} and Fursikov-Imanuvilov~\cite{FI:96}.
The aim of the present paper is threefold:
\begin{itemize}
\item Formulating a well-posedness result as well as an analogue of the Bardos-Lebeau-Rauch Theorem, for the wave equation with control like~\eqref{e:Dist-cont} (see Section~\ref{s:intro-cont-wave}); 
\item Formulating an analogue of the Lebeau-Robbiano-Fursikov-Imanuvilov Theorem for the heat equation with control like~\eqref{e:Dist-cont} (see Section~\ref{s:intro-cont-heat});
\item Formulating general lower bounds for restrictions on $\Sigma$ of eigenfunctions on $M$ (see Theorem~\ref{t:prelimEig} above and Section~\ref{s:intro-eigfct}). These are analogues of the observability inequalities used to prove the above controllability statements and are of their own interest.
\end{itemize}

%%%%%%%%%%%%%%%%%%%%%%%%
\subsection{Controllability for the wave equation}
\label{s:intro-cont-wave}
In this section, we state our main result concerning the wave equation controlled by an interior hypersurface $\Sigma$, namely
\begin{equation}
\label{e:waveControl}
\begin{cases}\Box v=f_0 \delta_\Sigma+f_1 \delta'_{\Sigma}&\text{on }(0,T)\times \Int(M) , \\
v =0 &\text{on }(0,T)\times \d M, \\
(v,\partial_t v)|_{t=0}=(v_0,v_1)&\text{in } \Int(M) .
\end{cases}
\end{equation}
where $\Box$ denotes the D'Alembert operator on $\R\times M$, 
$$
\Box= \d_t^2 - \Delta_g .
$$
Before considering the control problem, we need to investigate conditions on $f_0,f_1$ under which the Cauchy problem of~\eqref{e:waveControl} is well-posed.
Both the well-posedness and the control statements require the introduction of some geometric/microlocal definitions.

For a pseudodifferential operator $P$ on $\R \times M$, we write
 $$
 \Char(P)=\{q\in T^*(\R\times M)\setminus 0\mid \sigma(P)(q)=0\}
 $$
 and $\sigma(P)$ denotes the principal symbol of $P$. In particular, writing $|\xi|_g = \sqrt{g(\xi,\xi)}$ the Riemannian norm of a cotangent vector, we are interested in
 $$
 \sigma(\Box)(t,x,\tau,\xi)  = -\tau^2 + |\xi|_g^2 ,\qquad  \Char(\Box) = \{(t,x,\tau, \xi)\in T^*(\R\times M)\setminus 0 \mid  |\xi|_g^2 = \tau^2\} .
 $$
 Next, we define the glancing and the elliptic sets for $\Box$ above $\Sigma$ as
\begin{equation}
\label{e:squid}
 \begin{gathered}
 \G = \Char(\Box) \cap \iota(T^*(\R\times \Int(\Sigma)),\qquad
 \G^\Sigma = \iota^{-1}( \G) , \\
 \mc{E}=\{q\in T^*(\R\times M)\setminus 0\mid \sigma(\Box)(q)>0\}\cap \iota(T^*(\R\times \Int(\Sigma))),\qquad \mc{E}^{\Sigma}=\iota^{-1}(\mc{E}),
 \end{gathered}
 \end{equation}
 where \begin{equation}
 \label{e:def-iota}
 \iota:T^*(\R\times \Int(\Sigma_0))\hookrightarrow T^*(\R\times M)
 \end{equation} 
 is the inclusion map. A more explicit expression of these sets in normal coordinates is given in Section~\ref{s:intro-glancing-sets} below.
 
 Roughly speaking, the elliptic set $\mc{E}$ (resp. $\mc{E}^\Sigma$) consists in points $(t,x, \tau,\xi)$ in the whole phase space (resp. in tangential phase space to $\Sigma$) such that $x \in \Int(\Sigma)$ in which no ``ray of optics'' for $\Box$ lives. The glancing set $\G$ (resp. $\G^\Sigma$) consists in points $(t,x, \tau,\xi)$ in the whole phase space (resp. in tangential phase space to $\Sigma$) such that $x \in \Int(\Sigma)$, through which ``rays of optics'' for $\Box$ may pass {\em tangentially}. The complement of $\G \cup \mc{E}$ in the characteristic set of $\Box$ above $\R\times \Int(\Sigma)$ is the set of point  through which ``rays of optics'' for $\Box$ may pass {\em transversally}.
 
 \begin{definition}
\label{d:TGCC}
We say that $(\Sigma,T)$ satisfies the transverse geometric control condition ($\mc{T}$GCC) if every generalized bicharacteristic of $\Box$ intersects $T^*_{(0,T) \times \Int(\Sigma)}(\R\times M) \setminus \mc{G}$. We say that $\Sigma$ satisfies $\mc{T}$GCC if $(\Sigma, T)$ does for some $T>0$.
\end{definition}
Definition~\ref{d:TGCC} roughly says that $\mc{T}$GCC is satisfied if every ray of geometric optics intersects $\Int(\Sigma)$ in the time interval $(0,T)$ at a {\em transversal point}, i.e. a non-tangential point. 
In case $\d M = \emptyset$, "generalized bicharacteristics" are only bicharacteristics of $\Box$ and project on geodesics on $M$ (see e.g. \cite[Section~2.2]{DLRL:14}).
For a precise definition of generalized bicharacteristics in case $\d M \neq \emptyset$ {(and the geodesics of $M$ have no contact of infinite order with $\d M$}), we refer to~\cite[Section~3]{MS:78}, \cite[Chapter 24]{Hoermander:V3}, or~\cite[Section~1.3.1]{LRLTT:16}.
For a simple example of a compact manifold $M$ and a compact interior hypersurface $\Sigma$ satisfying $\mc{T}$GCC, see Figure~\ref{f:1}.

\begin{figure}
\centering
\includegraphics{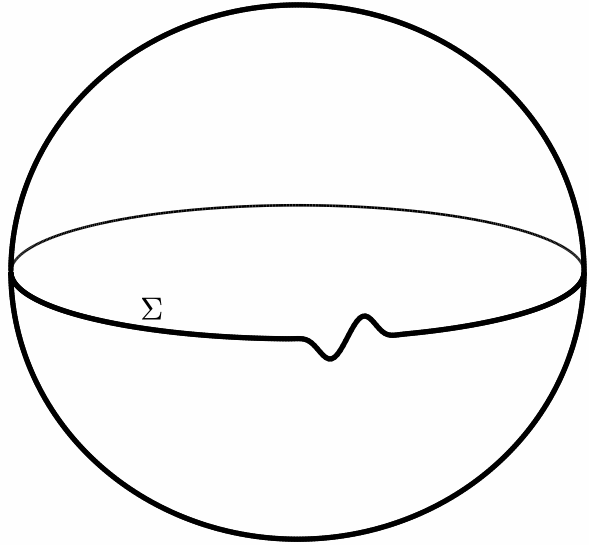}
\caption{\label{f:1} Here $M$ is the standard unit sphere $S^2$ in $\R^3$ and $\Sigma$ gives an example where $(\Sigma, 2\pi+\eps)$ satisfies $\mc{T}$GCC for $\eps>0$.}
\end{figure}

With these definitions in hand, our well-posedness result may be stated as follows.
 \begin{theorem}
\label{t:well-posedness-simple}
For all $(v_0,v_1)\in L^2(M) \times H^{-1}(M)$ and for all $f_0 \in H^{-1}_{\comp}(\R_+^* \times \Int(\Sigma))$ and $f_1 \in L^2_{\comp}( \R_+^* \times \Int(\Sigma))$ such that 
\bnan
\label{e:WF-condition}
\WF^{-\frac12}(f_0) , \WF^{\frac12}(f_1) \cap \G^\Sigma = \emptyset ,
\enan
there exists a unique  $v \in L^2_{\loc}(\R_+^*; L^2(M))$ solution of \eqref{e:waveControl}.
\end{theorem}
We refer e.g. to~\cite[Definition~1.2.21]{Lerner:10} for a definition of the $H^s$ wavefront set $\WF^s$ of a distribution. The wavefront condition states roughly that $(f_0, f_1)$ should have improved (namely $H^{-\frac12} \times H^{\frac12}$) microlocal regularity near the glancing set $\G^\Sigma$ (when compared to overall the $H^{-1}(\R\times \Sigma)\times L^2(\R\times \Sigma)$ regularity) for the Cauchy problem to be well-posed. A more precise version of this result is given in Theorem~\ref{t:well-posedness} below (where, in particular, the meaning of ``solution'' is made precise in the sense of transposition, see~\cite{Lio:88}).
This wavefront set condition on $f_0,\,f_1$ is far from sharp because we use a very rough analysis of solutions to the free wave equation near $\mc{G}$. A more detailed analysis near $\mc{G}$, similar to that in \cite{Ga16}, would yield sharper regularity requirements.

\bigskip
With this well-posedness result and the definition of $\mc{T}$GCC, we now give a sufficient condition for the null-controllability of~\eqref{e:waveControl} from $\Sigma$. 
\begin{theorem}
\label{thm:control}
Assume that the geodesics of $M$ have no contact of infinite order with $\d M$ and that $(\Sigma,T)$ satisfies $\mc{T}$GCC.
Then for any $(v_0,v_1)\in L^2(M)\times H^{-1}(M)$ there exist $(f_0,f_1)\in H^{-1}_{\comp}((0,T) \times \Int(\Sigma))\times L^2_{\comp}((0,T) \times \Int(\Sigma))$ with 
$$\WF(f_0),\WF(f_1)\cap  (\G^\Sigma\cup \mc{E}^\Sigma) = \emptyset ,
$$
 so that the solution to \eqref{e:waveControl} has $v\equiv 0$ for $t\geq T$. 
 \end{theorem}
 
Here, $\WF$ stands for the usual $C^\infty$ wavefront set.
Theorem \ref{thm:control} follows from an observability inequality given in Theorem \ref{thm:observe} below. 

Of course, it is classical to check that a necessary condition for controllability to hold is that all generalized bicharacteristics intersect $T^*_{(0,T) \times \Sigma}(\R\times M)$. As for the well-posedness problem, the issue of rays touching $\R\times \Sigma$ only at points of $\G^\Sigma$ is very subtle, and will be addressed in future work. See the discussion in Section~\ref{s:weak-TGCC} below.

 \subsection{Controllability from a hypersurface for the heat equation}
\label{s:intro-cont-heat}
 We next consider the controllability of the heat equation from a hypersurface, namely
 \begin{equation}
\begin{cases}
\label{e:heat-control}
(\d_t -\Delta)v = f_0 \delta_\Sigma +  f_1 \delta_\Sigma'   &\text{on }(0,T)\times \Int(M), \\
v =0 &\text{on }(0,T)\times \d M, \\
v|_{t=0}= v_0 &\text{in } \Int(M) .
\end{cases}
\end{equation}
Well-posedness in the sense of transposition follows from the standard parabolic estimates, and is proved in Section~\ref{s:WP:heat}. 
We only state a null-controllability result for~\eqref{e:heat-control}.
 \begin{theorem}
\label{t:control-heat}
Suppose $M$ is connected and $\Sigma$ is any nonempty interior hypersurface.
Then there exist $C,c>0$ such that for all $T>0$ and all $v_0 \in H^{-1}(M)$, there exist $f_0, f_1 \in L^2((0,T) \times \Sigma)$ with
$$
\| f_0 \|_{L^2((0,T)\times\Sigma)} + \| f_1 \|_{L^2((0,T)\times\Sigma)} \leq C e^{\frac{c}{T}}\|v_0\|_{H^{-1}(M)} ,
$$
such that the solution $v$ of~\eqref{e:heat-control} satisfies $v|_{t=T} = 0$.
\end{theorem}
Note that we also provide an estimate of the cost of the control as $T\to0^+$, similar to the one in case of internal/boundary control~\cite{FI:96,Miller:10}.

\subsection{Eigenfunction Restriction Bounds}
\label{s:intro-eigfct}
{As usual, the above two control results (or rather, the equivalent observability estimates) have related implication concerning eigenfunctions, stated in Theorem~\ref{t:prelimEig} above.
%Namely, denoting by $\phi_j$  the eigenfunctions associated to the eigenvalues $\lambda_j^2$,
%\bnan
%\label{e:eigenfcts}
%(-\Delta_g-\lambda_j^2) \phi_j =0, \quad \phi_j|_{\d M} = 0, \quad (\phi_j , \phi_k)_{L^2(M)} = \delta_{jk} , \quad \lambda_j \geq 0,
%\enan
%counterparts of Theorem~\ref{t:control-heat} and~\ref{thm:control} can be written respectively as 
%\bnan
% \|\phi_j|_{\Sigma}\|_{L^2(\Sigma)}+\|\lambda_j^{-1} \partial_\nu \phi_j|_{\Sigma} \|_{L^2(\Sigma)} \geq C e^{-c\lambda_j} \quad \text{ if }M \text{ connected, } \Sigma \neq \emptyset , 
% \label{e:eigenfct-estime-expo}  \\
%\|\phi_j|_{\Sigma}\|_{L^2(\Sigma)}+\|\lambda_j^{-1} \partial_\nu \phi_j|_{\Sigma} \|_{L^2(\Sigma)} \geq C  \quad \text{ if }\Sigma \text{ satisfies $\mc{T}$GCC}.
% \label{e:eigenfct-estime-TGCC}  
%\enan
We now formulate these results under the (stronger) form of resolvent estimates. Below, we write $\langle \lambda\rangle: =(1+|\lambda|^2)^{1/2}$.
}
\begin{theorem}[Universal lower bound for eigenfunctions]
\label{t:expo-bound-eigenfunctions}
Assume $M$ is connected and $\Sigma$ is a nonempty interior hypersurface.
Then there exist $C,c>0$ so that for all $\lambda \geq 0$ and all $u \in H^2(M)\cap H^1_0(M)$ we have
\begin{equation}
\label{e:unique}
\| u \|_{L^2(M)}\leq Ce^{c\lambda}\big( \|u|_{\Sigma}\|_{L^2(\Sigma)}+\|\la \lambda \ra^{-1} \partial_\nu u|_{\Sigma}\|_{L^2(\Sigma)} + \|(-\Delta_g-\lambda^2) u\|_{L^2(M)} \big).
\end{equation}
\end{theorem}

As far as the authors are aware, estimates \eqref{e:genLow}-\eqref{e:unique} are the first general lower bounds to appear for restrictions of eigenfunctions. Moreover, these estimates are sharp  in the sense that simultaneously neither the growth rate $e^{c\lambda}$ nor the presence of both $u$ and $\partial_\nu u$ can be improved in general. This is demonstrated by the following example.
\begin{proposition}
\label{prop:counterexample-revolution}
Consider the manifold 
$$
M = [-\pi, \pi ] \times \T^1, 
$$
with variables $(z, \theta)$, endowed with the warped product metric
$$
g(z, \theta) = dz^2 + R(z)^2 d\theta^2 .
$$
Assume that $R$ is smooth, that 
$$
R(z)=R(-z), \quad  R(z)\geq 1\text{ for all }z \in [0, \pi] , \quad  R(0)=1, \quad R(\frac{\pi}{2})=\sqrt{5} , \quad R(\pi)= \frac{1}{\sqrt{2}}. 
$$
Let $\Sigma=\{z=0\}\times \T^1 \subset M$.
Then, there exist $C,c >0$ and sequences $\lambda_j^{\eo}\to + \infty$ and $\phi_j^{\eo} \in L^2(M)$ such that 
 $$(-\Delta-(\lambda_j^{\eo})^2) \phi_j^{\eo}=0,\quad \|\phi_j^{\eo}\|_{L^2(M)}=1,\quad \phi_j^{\eo}(\pm \pi)=0,$$
with
$$
\d_\nu \phi_j^e|_{\Sigma} = 0 , \quad  \| \phi_j^e|_{\Sigma}\|_{L^2(\Sigma)} \leq C e^{-c \lambda_j^e}, 
\quad \text{ and } \quad 
\phi_j^o |_{\Sigma} = 0 , \quad  \| \d_\nu \phi_j^o|_{\Sigma}\|_{L^2(\Sigma)} \leq C e^{-c \lambda_j^o} .
$$
\end{proposition}
\noindent This result is proved in Appendix \ref{s:sharpUnique}.

We expect that the symmetry in this example is the obstruction for removing one of the traces in the right handside of \eqref{e:genLow}, and formulate the following conjecture.
\begin{conjecture}
Let $(M,g)$ be a Riemannian manifold and $\Sigma$ an interior hypersurface with {\em positive definite second fundamental form}. Then there exists $C,c,\lambda_0>0$ so that for all $(\lambda, \phi) \in [\lambda_0,\infty) \times L^2(M)$ satisfying~\eqref{e:eigenfunction}, we have
$$\|\phi\|_{L^2(M)}\leq Ce^{c\lambda}\|\phi |_{\Sigma}\|_{L^2(\Sigma)},\quad \text{ and } \quad \|\phi \|_{L^2(M)}\leq Ce^{c\lambda}\|\lambda^{-1}\partial_\nu \phi|_{\Sigma}\|_{L^2(\Sigma)}.$$ 
\end{conjecture}
Note that if $\Sigma$ has positive definite second fundamental form, then it is geodesically curved and in particular, not fixed by a nontrivial involution. This prevents the construction of counterexamples via the methods used to prove Lemma~\ref{prop:counterexample-revolution}.

Under the geometric control condition $\mc{T}$GCC the estimate \eqref{e:unique} can be improved. 
\begin{theorem}[Improved lower bound for eigenfunctions under $\mc{T}$GCC]
\label{t:uniqueControl}Assume that the geodesics of $M$ have no contact of infinite order with $\d M$ and that $\Sigma$ satisfies $\mc{T}$GCC. Then there exists $C>0$ so that for all $\lambda \geq 0$ and $u\in H^2(M) \cap H_0^1(M)$, we have
\begin{equation}
\label{e:uniqueControl}\|u\|_{L^2(M)}\leq C\big( \|u|_{\Sigma}\|_{L^2(\Sigma)}+\|\langle \lambda\rangle^{-1}\partial_\nu u|_{\Sigma}\|_{L^2(\Sigma)}+\langle \lambda\rangle^{-1}\|(-\Delta_g-\lambda^2)u\|_{L^2(M)} \big) .
\end{equation}
\end{theorem}

\begin{conjecture}
Let $(M,g)$ be a Riemannian manifold and $\Sigma$ an interior hypersurface with {\em positive definite second fundamental form} satisfying $\mathcal{T}$GCC. Then there exists $C,c,\lambda_0>0$ so that for all $(\lambda, \phi) \in [\lambda_0,\infty) \times L^2(M)$ satisfying~\eqref{e:eigenfunction}
%$$(-\Delta_g-\lambda^2)u=0, \quad u|_{\d M}=0 ,$$
we have
$$\|\phi\|_{L^2(M)}\leq C \|\phi|_{\Sigma}\|_{L^2(\Sigma)},\quad \text{ and } \quad  \|\phi\|_{L^2(M)}\leq C \|\lambda^{-1}\partial_\nu \phi|_{\Sigma}\|_{L^2(\Sigma)}.$$ 
\end{conjecture}

 Other known lower bounds come from the quantum ergodic restriction theorem and apply to a full density subsequence of eigenfunctions rather than to the whole sequence~\cite{TZ:12,TZ:13,DZ:13,TZ:17}.
 These hold under a ergodicity assumption on the geodesic (or the billiard) flow, together with a microlocal asymmetry condition for the surface $\Sigma$. This assumption states roughly that the measure of the set of geodesics through $\Sigma$ whose tangential momenta agree at adjacent intersections with $\Sigma$ is zero. 
 In another direction, the work of Bourgain--Rudnick \cite{BoRu12, BoRu11, BoRu09} shows that on the torus $\mathbb{T}^d$, $d=2,3$ for any hypersurface $\Sigma$ with positive definite curvature, \eqref{e:dynLow} holds with the normal derivative removed from the left hand side. While the results of Bourgain--Rudnick do not hold on a general Riemannian manifold, we expect that either of the terms in the left hand side of \eqref{e:genLow} can be removed whenever $\Sigma$ is not totally geodesic (which is even weaker that $\Sigma$ having positive definite second fundamental form). 

\subsection{Weakening Assumption $\mc{T}$GCC}
\label{s:weak-TGCC}
One might hope that Theorem~\ref{t:uniqueControl} and its analog for the wave equation (the control result of Theorem~\ref{thm:control} above and the observability inequality of Theorem~\ref{thm:observe} below) when the assumption $\mc{T}$GCC is replaced by the  (weaker) assumption that  
\bnan
\label{e:hopeAssume}
\text{every generalized bicharacteristics of $\Box$ intersects $T^*((0,T) \times \Int(\Sigma))$}
\enan
 (rather than $T^*((0,T) \times \Int(\Sigma)) \setminus \mc{G}^\Sigma$).
 The following example shows that this is more subtle (see Appendix~\ref{s:appendix-h1/4} for the proof).

\begin{proposition}
\label{e:h1/4}
Assume $M = S^2$ and $\Sigma$ is a great circle. Then there exists a sequence $(\lambda_j, \phi_j)$ satisfying $(-\Delta_g-\lambda_j^2) \phi_j = 0$ together with $\lambda_j \to + \infty$ and
$$
\phi_j|_{\Sigma} = 0 , \quad \|\lambda_j^{-1}\partial_\nu \phi_j|_{\Sigma}\|_{L^2(\Sigma)} \leq \lambda_j^{-1/4} \|\phi_j\|_{L^2(M)}. 
$$
\end{proposition}
In particular, this shows that Theorem \ref{t:uniqueControl} and associated observability inequality for the wave equation cannot hold under only \eqref{e:hopeAssume}.  Moreover, the proof shows that $\phi_j$ is microlocalized $\lambda_j^{-1}$ close to the glancing set on $\Sigma$, this calculation suggests that one must scale the normal derivative and restriction of an eigenfunction as in \cite{Ga16} to obtain an analog of Theorem \ref{t:uniqueControl} under \eqref{e:hopeAssume}. More precisely, 
\begin{conjecture}
Suppose that $\Sigma$ is a compact interior hypersurface. Then there exists $C>0$ so that if
 $(\lambda, \phi)$ satisfies~\eqref{e:eigenfunction}, then
$$\|(1+\lambda^{-2}\Delta_\Sigma)_+^{1/4}\phi |_{\Sigma}\|_{L^2}+\|[(1+\lambda^{-2}\Delta_\Sigma)_++\lambda^{-1}]^{-1/4}\lambda^{-1}\partial_\nu \phi |_{\Sigma}\|_{L^2(\Sigma)} 
\le \|\phi\|_{L^2(M)} .$$
Suppose moreover that $\Sigma$ satisfies \eqref{e:hopeAssume}. Then there exists $C>0$ so that if
 $(\lambda_j, \phi_j)$ satisfies~\eqref{e:eigenfunction}, then
$$\|\phi\|_{L^2(M)}\le C(\|(1+\lambda^{-2}\Delta_\Sigma)_+^{1/4}\phi |_{\Sigma}\|_{L^2}+\|[(1+\lambda^{-2}\Delta_\Sigma)_++\lambda^{-1}]^{-1/4}\lambda^{-1}\partial_\nu \phi |_{\Sigma}\|_{L^2(\Sigma)}.$$
where $\Delta_\Sigma$ is the Laplace-Beltrami operator on $\Sigma$ induced from $(M,g)$, and the operator $(1+\lambda^{-2}\Delta_\Sigma)_+$ is defined via the functional calculus, see also~\cite[Section 1]{Ga16}.
\end{conjecture}

\subsection{Finite unions of hypersurfaces}
\label{s:finite-union}
In all of our results, one may replace $\Sigma$ by any finite union of cooriented interior hypersurfaces $\cup_{i=1}^m \Sigma_i$ where we replace the distribution $f_0\delta_\Sigma + f_1\delta'_{\Sigma}$ by  
\bnan
\label{e:sums-dirac}
\sum_{i=1}^m \Big( f_0^i\delta_{\Sigma_i}+ f_1^i\delta'_{\Sigma_i} \Big).
\enan
Then, all above results generalize with the sole modification that generalized bicharacteristics need only intersect {\em one of the} $\Sigma_i$'s transversally. This furnishes several simple examples for which our controllability/observability results for waves holds. Take e.g. $\T^2 \simeq [-\pi, \pi]^2$ with $\Sigma_1 = \{0\} \times \T^1$ and $\Sigma_2 = \T^1 \times \{0\}$.

\bigskip
This remark can also be used to remove the coorientability assumption. 
If the interior hypersurface $\Sigma$ is not coorientable, we can cover it by a union of overlapping cooriented hypersurfaces $\Sigma = \cup_{i=1}^m \Sigma_i$ and control from $\Sigma$ by a sum like~\eqref{e:sums-dirac}. In this context, we still obtain controllability results with controls supported by the hypersurface $\Sigma$, but the form of the control is changed slightly.

%%%%%%%%%%%%%%%%%%%%%%%%%%%
\subsection{Sketch of the proofs and organization of the Paper}

We start in Section~\ref{s:prelim} with the introduction of coordinates, some geometric definitions and Sobolev spaces on $\Sigma$.

Section~\ref{s:wellposed} is devoted to the proof of (a slightly more precise version of) the well-posedness result of Theorem~\ref{t:well-posedness-simple}. The definition of solutions in the sense of transposition follows~\cite{Lio:88}. The well-posedness result relies on {\em a priori} estimates on an adjoint equation -- the free wave equation. The well-posedness statement then reduces to the proof of regularity bounds for restrictions on $\Sigma$. This is done in Section~\ref{s:regularity}. Namely, we show that if $u$ is an $H^1$ solution to $\Box u=0$, then the restriction $(u|_\Sigma , \d_\nu u|_\Sigma)$ belong to $H^\frac12 \times H^{-\frac12}$ overall $\R \times \Sigma$, and have the additional (microlocal) regularity $H^1 \times L^2$ everywhere except near Glancing points ($\G^\Sigma$). This fact is already known (see e.g.~\cite{Tataru:98}) but we rewrite a short proof for the convenience of the reader. Then, Section~\ref{s:spaces-comp-loc} is aimed at defining the appropriate spaces for the statement of the precise version of the well-posedness result. These are needed in particular to state the stability result associated to well-posedness, as well as to formulate the duality between the control problem and the observation problem.
They are (loc and comp) Sobolev spaces on $\R\times \Sigma$ that have different regularities near and away from the Glancing set $\G^\Sigma$.
With these spaces in hand, we define properly solutions of~\eqref{e:waveControl} and prove well-posedness in Section~\ref{s:def-sol-WP}.

\bigskip
Section \ref{s:obs-cont-wave} is devoted to the proof of the control result of Theorem~\ref{t:well-posedness-simple}. Before entering the proofs, we briefly explain how~Theorem~\ref{t:uniqueControl} is deduced from the observability inequality of Theorem \ref{thm:observe}. Firstly, we prove in Section~\ref{s:TGCC} that the condition $\mc{T}$GCC implies a stronger geometric statement. Namely, using the openness of the condition and a compactness argument, we prove that all rays intersect in $(\eps, T-\eps)$ an open set of $\Sigma$ ``$\eps$-transversally'' (i.e. $\eps$ far away from the glancing region) for some $\eps>0$. Secondly, this condition is used in Section~\ref{s:weakObs} to prove an observability inequality, stating roughly that the observation of both traces $(u|_\Sigma , \d_\nu u|_\Sigma)$ of microlocalized $\eps$ far away from the glancing region in the time interval $(\eps, T-\eps)$ determines the full energy of solutions of $\Box u =0$ (in appropriate spaces). The proof proceeds as in~\cite{Leb:96} by contradiction, using microlocal defect measures. It contains two steps: first, we prove that the strong convergence of a sequence $(u_k|_\Sigma , \d_\nu u_k|_\Sigma)\to 0$ near a transversal point of $\Sigma$ implies the strong convergence of the sequence $u_k$ in a microlocal neighborhood of the two rays passing through this point (using the hyperbolic Cauchy problem). Then, a classical propagation argument (borrowed from~\cite{Leb:96, BL:01} in case $\d M\neq \emptyset$) implies the strong convergence of $(u_k)$ everywhere, which yields a contradiction with the fact that the energy of the solution is normalized.
This observability inequality contains, as in the usual strategy of~\cite{BLR:92}, a lower order remainder term (in order to force the weak limit of the above sequence to be $0$). The latter is finally removed in Section~\ref{s:obs} by the traditional compactness uniqueness argument of~\cite{BLR:92}, concluding the proof of the observability inequality.
Finally, in Section~\ref{s:control}, we deduce the controllability statement Theorem \ref{thm:control} (or its refined version Theorem~\ref{thm:control2}) from the observability inequality (Theorem~\ref{thm:observe}) via a functional analysis argument. The latter is not completely standard, since we do not know whether the solution of the controlled wave equation~\ref{e:waveControl} has the usual $C^0(0,T; L^2(M)) \cap C^1(0,T; H^{-1}(M))$ regularity, but only prove $L^2(0,T; L^2(M))$. As a consequence, we cannot use data of the adjoint equation at time $t=T$ as test functions. The test functions we use are rather forcing terms $F$ in the right hand-side of the adjoint equation, that are supported in $t \in (T, T_1)$, that is, outside of the time interval $(0,T)$. Also, we construct control functions having $H^N$ regularity near $\G^\Sigma$ and prove that they do not depend on $N$, yielding the statement with the $C^\infty$ wavefront set.

 \bigskip
 Section~\ref{s:controlHeat} deals with the case of the heat equation and the universal lower bound of Theorem~\ref{t:expo-bound-eigenfunctions}, in the spirit of the seminal article~\cite{LR:95}.
 First, Section~\ref{s:WP:heat} states the well-posedness result, in the sense of transposition. Again, it relies on the regularity of restrictions to $\Sigma$ of solutions of the adjoint free heat equations. The latter are deduced from standard parabolic regularity combined with Sobolev trace estimates.
 Then, to prove observability/controllability, we proceed with the Lebeau-Robbiano method~\cite{LR:95}. The starting point is a local Carleman estimate near $\Sigma$, borrowed from~\cite{LR:97}, from which we deduce in Section~\ref{s:unique} a global interpolation inequality for the operator $-\d_s^2-\Delta_g$. Theorem~\ref{t:expo-bound-eigenfunctions} directly follows from this interpolation inequality.
To deduce the observability of the heat equation, we revisit slightly (in an abstract semigroup setting) the original Lebeau-Robbiano method (as opposed to the simplified one~\cite{LZ:98,Miller:06,LeLe:09}, relying on a stronger spectral inequality) in Section~\ref{s:LR95-revisited}. The interpolation inequality yields as usual an observability result for a finite dimensional elliptic evolution equation (i.e. cutoff in frequency), from which we deduce observability for the finite dimensional parabolic equation, with precise dependence of the constant with respect to the cutoff frequency and observation time. The latter argument simplifies the original one by using an idea of Ervedoza-Zuazua~\cite{EZ:11s,EZ:11}. The observability of the full parabolic equation is finally deduced using the iterative Lebeau-Robbiano argument combining high-frequency dissipation with low frequency control/observation. We in particular use the method as refined by Miller~\cite{Miller:10}. We explain in Section~\ref{s:heat-application} how the heat equation observed by/controlled from $\Sigma$ fits into the abstract setting.

\bigskip
Appendix~\ref{s:pseudo} contains some background information on pseudodifferential operators used in Sections~\ref{s:wellposed} and~\ref{s:obs-cont-wave}. for the wave equation. Appendix~\ref{s:sharpUnique} proves Proposition~\ref{prop:counterexample-revolution}, i.e. constructs an example showing that Theorem \ref{t:expo-bound-eigenfunctions} is sharp. Finally, Appendix~\ref{s:appendix-h1/4} gives a proof of Proposition~\ref{e:h1/4}.

\bigskip
\noindent
{\em Acknowledgements.}  The authors would like to thank John Toth for having encouraged them to write a proof of~\eqref{e:genLow}. The first author is grateful to the National Science Foundation for support under the Mathematical Sciences Postdoctoral Research Fellowship  DMS-1502661.
The second author is partially supported by the Agence Nationale de la Recherche under grants GERASIC ANR-13-BS01-0007-01 and ISDEEC ANR-16-CE40-0013.

\section{Preliminary definitions}
\label{s:prelim}
\subsection{Fermi normal coordinates in a neighborhood $\Sigma_0$}
\label{s:fermi}
Throughout the article we shall use Fermi normal coordinates in a (sufficiently small) neighborhood, say $V_\eps$, of $\Sigma_0$. Namely since $\Sigma_0$ is cooriented, for $\eps$ sufficiently small, there exists a diffeomorphism (see \cite[Appendix~C.5]{Hoermander:V3})
\begin{align*}
[- \eps, \eps]  \times\Int(\Sigma_0) &\to  V_\eps \\
  (x_1 , x' ) &\mapsto  x,
\end{align*}
so that the differential operator $- \Delta_g$ takes the form 
\begin{equation*}
- \d_{x_1}^2  + r(x_1 , x' ,D_{x'}) + c(x,D),
\end{equation*}
where $c(x,D)$ is a first order differential operator and $r(x_1 , x' ,D_{x'})$ is an $x_1$-family of second-order elliptic differential operators on $\Int(\Sigma_0)$, i.e. a tangential operator, with principal symbol $r (x_1 , x' , \xi' )$, $\xi' \in T_{x'}^*\Int(\Sigma_0)$, that satisfies
\begin{equation}
  \label{eq: ellipticity r}
 r (x_1 , x' , \xi' ) \in \R,\quad \text{and}
  \quad C_1|\xi'|^2 \leq r (x_1 , x' , \xi' ) \leq C_2   |\xi'|^2,
\end{equation}
for some $0<C_1\leq C_2<\infty$. 

In these coordinates, note that we have in particular $|x_1(p)|=d(p,\Sigma)$, $\d_\nu = \d_{x_1}$ (up to changing $x_1$ into $-x_1$), as well as 
$$\sigma(-\Delta_g)=\xi_1^2+r(x_1,x',\xi')$$
and 
$$
-\Delta_{\Sigma_0} = r(0,x',D_{x'}), \quad \text{ with } \quad \sigma(-\Delta_{\Sigma_0})(x',\xi')= r(0,x',\xi')=:r_0(x',\xi'),
$$
where $-\Delta_{\Sigma_0}$ is the Laplacian on $\Int(\Sigma_0)$ given by the induced metric on $\Sigma_0$. We also recall that 
$$
\sigma(\Box)=-\tau^2+\sigma(-\Delta_g)  = -\tau^2+|\xi|_g^2 = -\tau^2 + \xi_1^2 + r(x_1 , x' , \xi') .
$$
With a slight abuse of notation, we shall also denote by $(x_1,x') \in \R \times \R^{n-1}$ (and $(\xi_1, \xi') \in \R \times \R^{n-1}$ associated cotangent variables) local coordinates in a neighborhood of a point in $\Int(\Sigma_0)$.

In these coordinates, the Hamiltonian vector field of $\Box$ is given by 
\bnan
\label{e:Ham-vect-coord}
H_{\sigma(\Box)} = -2 \tau \d_t + 2 \xi_1 \d_{x_1} - \d_{x_1}r(x_1, x' , \xi') \d_{\xi_1} + \d_{\xi'}r(x_1, x' , \xi') \d_{x'} - \d_{x'}r(x_1, x' , \xi') \d_{\xi'}
\enan
and generates the Hamiltonian flow of $\Box$ (these coordinates beeing away from the boundary $\R \times \d M$).

\subsection{The compressed cotangent bundle over $M$}
This section is independent of the hypersurface $\Sigma$ and is only aimed at defining, in case $\d M \neq \emptyset$, the space $Z$ on which the Melrose-Sj\"ostrand bicharacteristic flow is defined, as well as some properties of the flow. In case $\d M = \emptyset$, this set is $\Char(\Box) \subset T^*(\R \times M)\setminus 0$, the flow is the usual bicharacteristic flow of $\Box$, and this section not needed and may be skipped.
We refer to~\cite{MS:78}, \cite[Appendix A2]{Leb:96} for more complete treatments.

\medskip
We first embed $M\hookrightarrow \tilde{M}$ into a manifold, $\tilde{M}$, without boundary and write
$$T^*(\R\times M):=T_{\R\times M}^*(\R\times \tilde{M}).$$

Let $\TRM \simeq \big( T^*(\R\times \Int(M))\setminus 0 \big) \sqcup \big( T^*(\R\times \d M)\setminus 0 \big)$ denote the compressed cotangent bundle of of $\R \times M$ and $$j:T^*(\R\times M)\to \TRM$$ be the natural ``compression'' map. In any coordinates $(x',x_n)$ on $M$ where $x_n$ defines $\partial M$ and $x_n>0$ on $M$, $j$ has the form
\begin{equation}
\label{e:jdef}
j(t,x,\tau,\xi)=(t,x,\tau,\xi',x_n\xi_n).
\end{equation}
{The map $j$ endows $\TRM$ with a structure of homogeneous topological space.
We then write
\begin{equation}
\label{e:defZ} Z=j ( \Char(\Box)) , \quad \hat{Z} = Z \cup j \big(T^*_{\R \times \d M}(\R \times M )\big) ,
\end{equation}
and 
\begin{equation}
\label{e:defSZ} 
S\hat Z = \big(\hat{Z}\setminus(\R\times M) \big) /\R_+^* ,
\end{equation}
 the associated sphere bundle, which, endowed with the induced topology, are locally compact metric spaces.

Away from the boundary, $j$ is a bijection and we shall systematically identify $^bT^*(\R\times \Int(M))$ with $T^*(\R \times \Int(M))$ and $Z \cap \ ^bT^*(\R\times \Int(M))$ with $\Char(\Box)\cap T^*(\R\times \Int(M))$. 
This will be the case in particular near the hypersurface $\R \times \Sigma$.

\bigskip
Under the assumption that the geodesics of $M$ have no contact of infinite order with $\d M$, and with $Z$ as in~\eqref{e:defZ}, the (compressed) generalized bicharacteristic flow for the symbol $\frac{1}{2}(-\tau^2+|\xi|_g^2)$ is a (global) map 
\bnan
\label{e:def-flow}
\varphi :\R\times Z\to Z , \quad  (s, p) \mapsto \varphi(s,p)
\enan
We refer to \cite[Section~3]{MS:78}, \cite[Chapter 24]{Hoermander:V3}, \cite[Section~3.1]{BL:01} or~\cite[Section~1.3.1]{LRLTT:16} for a definition. In particular, it has the following properties
\begin{itemize}
\item $\varphi$ coincides with the usual bicharacteristic flow of $\Box$ (i.e. the Hamiltonian flow of $\sigma(\Box)$) in the interior $\Char(\Box) \cap T^*(\R\times \Int(M))$;
\item $\varphi$ satisfies the flow property 
\bnan
\label{e:phi-flowprop}
\varphi(t, \varphi(s, p)) = \varphi(t + s, p) ,\quad \text{ for all } t, s \in \R , p \in Z; 
\enan
\item $\varphi$ is homogeneous in the fibers of $Z$, in the sense that 
\bnan
 \label{e:flow-hom}
 M_{\lambda} \circ \varphi(s\lambda,\cdot) =\varphi(s,M_\lambda \cdot),
 \enan
where $M_\lambda$ denotes multiplication in the fiber by $\lambda>0$;
Hence, it induces a flow on $S\hat Z$.
\item $\varphi :\R\times Z\to Z$ is continuous, see~\cite[Theorem~3.34]{MS:78}.
\end{itemize}
}

\subsection{Glancing Sets over $\Sigma$}
\label{s:intro-glancing-sets}
For the following definitions, we use the above identification $\TRMsig \cong T^*_{\R\times \Sigma_0}(\R\times  M)$ for the cotangent bundle of $\R\times M$ with foot points at $\R \times\Sigma_0$, since in this case, we may assume in Definition~\ref{def:hypersurface} that $\Sigma_0 \cap \d M = \emptyset$.
Using the coordinates of Section~\ref{s:fermi}, the map $\iota$ defined in~\eqref{e:def-iota} reads
$$
\iota(t, x', \tau, \xi') = (t, 0, x', \tau , 0 , \xi') .
$$
Still in coordinaltes, we define for $\eps\geq 0$, the sets
\begin{gather*}
\mc{G}_{\e}=\{(t,0, x',\tau,\xi_1, \xi') \in T_{\R \times \Int(\Sigma)}^*(\R\times M) \setminus 0\mid \xi_1^2+r(0,x',\xi') = \tau^2 ,\,  \xi_1^2\leq \eps \tau^2 \},\\
\mc{T}_{\e}:=\{(t,0, x',\tau,\xi_1, \xi') \in T_{\R \times \Int(\Sigma)}^*(\R\times M)\setminus 0\mid  \xi_1^2+r(0,x',\xi') = \tau^2 , \,  \xi_1^2> \eps \tau^2 \}.
\end{gather*}
Let also
\begin{equation}
\label{e:SigmaSets}
\begin{aligned}
\G^{\Sigma}_\e& :=  \left\{ (t,x', \tau , \xi') \in T^*(\R \times \Int(\Sigma_0))\setminus 0 \mid x' \in \Int(\Sigma) ,\, -\e\tau^2\leq  \tau^2-r_0(x',\xi')\leq \e\tau^2\right\}.\\
\mc{T}_\e^{\Sigma}&:= \left\{ (t,x', \tau , \xi') \in T^*(\R \times \Int(\Sigma_0))\setminus 0 \mid x' \in \Int(\Sigma) , \, \e\tau^2< \tau^2-r_0(x',\xi')\right\},\\
\mc{E}_\e^{\Sigma}&:=  \left\{ (t,x', \tau , \xi') \in T^*(\R \times \Int(\Sigma_0)) \setminus 0\mid x' \in \Int(\Sigma) ,\,  \tau^2-r_0(x',\xi')< -\e\tau^2\right\}.
\end{aligned}
\end{equation}
Observe that $\G^{\Sigma}_\e\neq \pi(\mc{G}_\e)$ (although $\G^\Sigma_0=\pi(\G_0)$) where 
\begin{equation}
\label{e:projection} \pi:T^*_{\Int(\Sigma_0)}(\R \times M)\to T^*(\R\times \Int(\Sigma_0))\text{ is projection along }N^*(\R\times \Int(\Sigma_0)).
\end{equation}
In the above coordinates, $\pi(t,0,x',\tau,\xi_1,\xi')=(t,x',\tau,\xi').$
 Observe also that $\G_0=\G$ and $\G_0^\Sigma=\G^\Sigma$ where $\G$ and $\G^{\Sigma}$ are defined in~\eqref{e:squid}.

\begin{remark}
\label{rem:transverse}
In these coordinates, $\Sigma_0 = \{x_1 = 0\}$ and, according to~\eqref{e:Ham-vect-coord}, we have $\langle dx_1 , H_{\sigma(\Box)} \rangle = \langle dx_1 , 2 \xi_1 \d_{x_1}\rangle = 2 \xi_1$. Since $\xi_1 \neq 0$ on 
$\mc{T}_0$ this implies in particular that the vector field $H_{\sigma(\Box)}$ is transverse to the Hypersurface $\R \times \Sigma$ on this set (which explains its name $\mc{T}_0$).
\end{remark}

With these definitions, $\mc{T}$GCC can be written as
\begin{assume}{$0$}{$T$}\label{GC}
For all $p\in Z$, $$\bigcup\limits_{s\in \R}\{\varphi(s,p)\}\cap \mc{T}_0\cap {T^*((0,T)\times M)}\neq \emptyset.$$
\end{assume}

\subsection{Spaces on interior hypersurfaces}
In case $\Sigma$ is a compact internal hypersurface, then the Sobolev spaces $H^s(\Sigma)$ have a natural definition. Here, we give a definition adapted to the case $\d\Sigma \neq \emptyset$.
\begin{definition}
\label{d:def-ext-Hs}
Let $S$ be an interior hypersurface of a $d$ dimensional manifold $X$, and let $S_0$ be an extension of $S$ (see Definition~\ref{def:hypersurface}). 
Given $s\in\R$, we say that $u \in \bar{H}^s(S)$ (extendable Sobolev space) if there exists $\underline{u} \in H^s_{\comp}(S_0)$ such that $\underline{u}|_{S} = u$.

To put a norm on $\bar{H}^s(S)$, let $\chi \in C^\infty_c( \Int (S_0))$ such that $\chi=1$ in a neighborhood of $S$. We denote by $(U_j , \psi_j)_{j \in J}$ an atlas of $S_0$ such that for all $j\in J$
$$U_j \cap \supp \chi=\emptyset \qquad \text{ or }\qquad U_j \cap \partial S_0= \emptyset,$$
 and write $J_S=\{j \in J, U_j \cap \supp \chi \neq \emptyset\}$ and $J_\d =\{j \in J, U_j \cap (\supp\chi\setminus \Int(S)) \neq \emptyset\} \subset J_S$ (possibly empty). 
Let $(\chi_j)_{j \in J}$ be a partition of unity of $S_0$ subordinated to $(U_j)_{j \in J}$.
Given , we define 
\begin{equation}
\label{e:def-overHs}
\begin{gathered}
\|u\|_{\bar{H}^s(S)}= \sum_{j \in J_S \setminus J_\d} \|(\chi_ju) \circ \psi_j^{-1}\|_{H^s(\R^{d-1})} 
+\inf_{\underline{u}\in E_u}  \sum_{j \in J_\d} \| (\chi_j  \chi \underline{u}) \circ \psi_j^{-1}\|_{H^s(\R^{d-1})},\\
E_u:=\{\underline{u} \in H^s_{\comp}(\Int(S_0)) ,\, \underline{u}|_S=u\}.
\end{gathered}
\end{equation}
\end{definition}
The definition of the norm $\bar{H}^s(S)$ depends on $S_0 , \chi$, the choice of charts $(U_j , \psi_j)$ and the partition of unity $(\chi_j)$. One can however prove that, once $S_0$ and $\chi$ are fixed, two such choices of charts $(U_j , \psi_j)$ and partition of unity $(\chi_j)$ lead to equivalent norms $\bar{H}^s(S)$. In what follows, $(U_j , \psi_j, \chi_j)$ shall be traces on $S_0$ of charts and partition of unity on $X$. In case $S$ is a compact interior hypersurface, then the spaces $\bar{H}^s(S)$,  $\|\cdot\|_{\bar{H}^{s}(S)}$ coincides with usual $H^s(S)$ space.

%%%%%%%%%%%%%%%%%%%%%%
\section{Regularity of traces and well-posedness for the wave equation}
\label{s:wellposed}

The ultimate goal of the present section is to prove the well-posedness result for~\eqref{e:waveControl}, see Theorem~\ref{t:well-posedness-simple}. Defining solutions by transposition as in~\cite{Lio:88}, this amounts to proving regularity of traces on $\Sigma$ of solutions to the free wave equation. 

%%%%%%%%%%%%%%%%%%%%%%
\subsection{Regularity of traces}
\label{s:regularity}

We start by giving estimates on the restriction to $\Sigma$ of a solution to
\begin{equation}
\label{e:waveObs}
\begin{cases}\Box u= F &\text{on }\R\times \Int(M),\\
u=0&\text{on }\R\times \partial M,\\
(u,\partial_tu)|_{t=0}=(u_0,u_1)&(u_0,u_1)\in H^1(M)\times L^2(M).
\end{cases}
\end{equation} These bounds, indeed stronger bounds, can be found in \cite{Tataru:98}, but we choose to give the proof of the simpler estimates here for the convenience of the reader. They are closely related to the semiclassical restriction bounds from \cite{BGT,Tac:10,Tac:14,ChHaTo,Ga16}.

\begin{proposition}
\label{prop: regularity-waves}
Fix $T>0$. Then for any $A \in \Psi^0_{\phg}(\R \times \Int(\Sigma))$, with principal symbol vanishing in a neighborhood of $\mc{G}_0^\Sigma$ and all $\varphi\in C_c^\infty(\R)$, there exists $C>0$ so that for any  
 $(u_0,u_1)\in H^1(M)\times L^2(M)$ and $F\in L^2(\R\times M)$ with $\supp F\subset [0,T]\times M$ the solution $u$ to \eqref{e:waveObs} satisfies
\begin{multline} 
\label{e:cookie}
\|\varphi(t)u|_{\Sigma}\|_{\bar{H}^{1/2}(\R\times \Sigma)}^2+ \|\varphi(t)\partial_\nu u|_{\Sigma}\|_{\bar{H}^{-1/2}(\R\times \Sigma)}^2  +\|A \varphi(t)(u|_{\Sigma})\|_{\bar{H}^1(\R\times \Sigma)}^2 +\|A \varphi(t)(\partial_\nu u|_{\Sigma})\|_{L^2(\R \times \Sigma)}^2  \\
 \leq C(\|(u_0,u_1)\|_{H^1(M)\times L^2(M)}^2+\|F\|_{L^2}^2)
 \end{multline}
\end{proposition}

To prove Proposition \ref{prop: regularity-waves} we need the following elementary lemma.
\begin{lemma}
\label{l:restrict}
Suppose that $S$ is an interior hypersurface of the $d$ dimensional manifold $X$ (in the sense of Definition~\ref{def:hypersurface}) and $P\in \Psi_{\phg}^m(\Int(X))$ is elliptic on the conormal bundle to $\Int(S_0)$, 
$N^*\Int (S_0).$ 
Then for any $s\in \R$, $k\geq 0$ and $\epsilon>0$, there exists $C=C(\epsilon,k,s)>0$ so that for all $u \in C^\infty(M)$,
$$\|\partial_\nu^ku|_{S}\|_{\bar{H}^s(S)}\leq C(\|u\|_{H^{s+k+1/2}(X)}+\|Pu\|_{H^{1/2+k+\epsilon-m}(X)}).$$
\end{lemma}
\bnp
We start by proving the case $k=0$. 
In case $s>0$, the stronger inequality $\| u|_{S}\|_{\bar{H}^s(S)}\leq C\|u\|_{H^{s+k+1/2}(X)}$ holds as a consequence of standard trace estimates \cite[Theorem B.2.7]{Hoermander:V3} (that the $\bar{H}^s(S)$ norm is the appropriate one in case $S$ is not compact is made clear below).

We now assume that $s\leq 0$, and estimate each term in the definition~\eqref{e:def-overHs} of $\|u|_{S}\|_{\bar{H}^s(S)}$ in local charts. For this, we use charts $(\Omega_i, \kappa_i)_{i \in I}$ of $\Int(X)$ such that $S_0 \subset \bigcup_{i\in I}\Omega_i$ and such that $(\Omega_i \cap S_0, \kappa_i|_{S_0})_{i \in I}$ satisfy the assumptions of Definition~\ref{d:def-ext-Hs}. In a neighborhood of $S_0$, we have $u = \sum_i \tilde{\chi}_i u$ (where $(\tilde{\chi}_i)$ is now a partition of unity of $S_0$ associated to $\Omega_i$, and hence, $(\tilde{\chi}_i|_{S_0})$ satisfies the assumptions of Definition~\ref{d:def-ext-Hs}), and estimating $\|u|_{S}\|_{\bar{H}^s(S)}$ amounts to estimating each 
$$\|((\tilde{\chi}_ju)|_{S_0}) \circ \kappa_j^{-1}\|_{H^s(\R^{d-1})} =  \| (\check{\chi} w)|_{x_1=0} \|_{H^s(\R^{d-1})}$$ 
with $\check{\chi}=\tilde{\chi}_j \circ \kappa_j^{-1}$ and $w =u \circ \kappa_j^{-1}$.
We may now work locally, where $S$ is a subset of $\{x_1=0\}$, and estimate the trace of $z=\check{\chi} w$.

Let $\chi\in C_c^\infty(\R)$ have $\chi \equiv 1$ on $[-1,1]$ with $\supp \chi \subset [-2,2]$, $0\leq \chi \leq 1$ and fix $\delta>0$ small enough so that $P$ (which, by abuse of notation, we use for the operator in local coordinates) is elliptic on a neighborhood of 
$$\{|\xi_1|\geq\delta^{-1}|\xi'|\}\supset N^*(\{0\} \times \R^{d-1}) = \{\xi'=0\}$$
 and let $\chi_\delta(\xi_1,\xi')=\chi\left(\frac{2|\xi'|}{\delta \xi_1}\right)$ for $\xi_1\neq 0$ and $\chi_\delta(0,\xi')=0$. Then, we have
\begin{align*}
\|z|_{x_1=0}\|_{H^s(\R^{d-1})}^2&=\int_{\R^{d-1}} \langle \xi'\rangle^{2s} \left|\int_\R \hat{z}(\xi_1,\xi')d\xi_1\right|^2d\xi' \leq 2(A+B) ,
\end{align*}
with
\begin{align*}
A = \int_{\R^{d-1}}\langle \xi'\rangle^{2s}\left|\int_\R  (1-\chi_\delta)  \hat{z}(\xi_1,\xi')d\xi_1\right|^2d\xi' , \quad \text{and} \quad  B = \int_{\R^{d-1}}\langle \xi'\rangle^{2s}\left|\int_\R  \chi_\delta  \hat{z}(\xi_1,\xi')d\xi_1\right|^2d\xi'  .
\end{align*}
We now estimate each term. With the Cauchy Schwarz inequality, the first term is estimated by
\begin{align*}
A &= \int\langle \xi'\rangle^{2s}\left|\int \frac{(1-\chi_\delta)}{\langle \xi\rangle^{s+1/2} } \langle \xi\rangle^{s+1/2} \hat{z}(\xi_1,\xi')d\xi_1\right|^2d\xi'\\
&\leq  \int_{\R^{d-1}} \left(\int_\R \frac{\langle \xi'\rangle^{2s}(1-\chi_\delta)^2}{\langle \xi\rangle^{2s+1}}d\xi_1 \right) \left(\int_\R \langle \xi\rangle^{2s+1}|\hat{z}|^2(\xi_1,\xi')d\xi_1 \right)d\xi'\\ 
&\leq C_{s,\delta}\|z\|_{H^{s+1/2}(\R^d)}^2 ,
\end{align*}
since 
\begin{align*} 
\int_\R \frac{\langle \xi'\rangle^{2s}(1-\chi_\delta)^2}{\langle \xi\rangle^{2s+1}}d\xi_1 &=\int_\R \frac{1}{(1+t^2)^{s+1/2}} \left(1-\chi\left( \frac{2|\xi'|}{\delta \langle \xi' \rangle t}\right) \right)^2 dt\\
& \leq \int_{|t|\leq 2/\delta} \frac{1}{(1+t^2)^{s+1/2}} dt = : C_{s,\delta}
\end{align*}
 (which is large since $s\leq 0$).

Again with the Cauchy Schwarz inequality, the second term is estimated by
\begin{align*}
B &= \int_{\R^{d-1}} \langle\xi'\rangle^{2s}\left|\int_\R \frac{\langle \xi\rangle^{1/2+\epsilon}\chi_\delta}{\langle \xi\rangle^{1/2+\epsilon} }\hat{z}(\xi_1,\xi')d\xi_1\right|^2d\xi' \\
&\leq \int_{\R^{d-1}} \left(\int_\R \frac{\langle \xi'\rangle^{2s}}{\langle \xi\rangle^{1+2\epsilon}}d\xi_1 \right) \left(\int_\R \langle \xi\rangle^{1+2\epsilon}\chi_\delta^2|\hat{z}|^2(\xi_1,\xi')d\xi_1 \right)d\xi'\\
&\leq  C_\epsilon\|\chi_\delta(D) z\|_{H^{1/2+\epsilon}(\R^d)}^2 ,
\end{align*}
since 
$$\int_\R \frac{\langle \xi'\rangle^{2s}}{\langle \xi\rangle^{1+2\epsilon}}d\xi_1 = \langle \xi'\rangle^{2s-2\epsilon} \int_\R \frac{1}{(1+t^2)^{1/2+\epsilon}}  dt =  \langle \xi'\rangle^{2s-2\epsilon}  C_\epsilon,$$ 
with $C_\epsilon$ finite as soon as $\epsilon>0$, and $\langle \xi'\rangle^{2s-2\epsilon}\leq 1$ since $s\leq 0$. Combining the last three estimates and recalling that $z=\check{\chi} w$ yields
\begin{align}
\label{e:estim-loc-1}
\|\check{\chi} w|_{x_1=0}\|_{H^s(\R^{d-1})}^2&\leq C_{s,\delta}\| \check{\chi} w\|_{H^{s+1/2}(\R^d)}^2 + C_\epsilon\|\chi_\delta(D) \check{\chi} w\|_{H^{1/2+\epsilon}(\R^d)}^2,
\end{align}

Now, according to the definition of $\chi_\delta$, the operator $P$ is elliptic on a neighborhood of $\supp(\check\chi) \times \supp(\chi_\delta)$, a classical parametrix construction (see for instance \cite[Theorem~18.1.9]{Hoermander:V3}) implies, for any $N\in \N$, 
\begin{align}
\label{e:estim-loc-2}
\|\chi_\delta(D)\check{\chi} w\|_{H^{1/2+\epsilon}(\R^{d})}\leq C_N(\| \check{\check{\chi}} Pw\|_{H^{1/2+\epsilon-m}(\R^d)}+\| \check{\check{\chi}} w\|_{H^{-N}(\R^d)}),
\end{align}
where $\check{\check{\chi}}$ is supported in the local chart and equal to one in a neighborhood of $\supp(\check{\chi})$. Recalling that $w$ is the localization of $u$, and summing up the estimates~\eqref{e:estim-loc-1}-\eqref{e:estim-loc-2} in all charts yields the sought result for $k=0$.

We now show that the $k=0$ case implies the $k>0$ case. Let $\tilde{P}\in \Psi^m_{\phg}(\Int (X))$ be elliptic on $N^*(\Int(S_0))$ with $\WF(\tilde{P})\subset \{\sigma(P)\neq 0\}$ (see e.g. Appendix~\ref{s:pseudo} for a definition of $\WF(A)$ for a pseudodifferential operator $A$). Then, applying the case $k=0$ to the operator $\tilde{P}$, we obtain
\begin{align}
\label{e:estim-loc-3}
\|\partial_\nu ^k u|_{\Sigma}\|_{\bar{H}^s(\Sigma)}&\leq C \left( \|\partial_\nu ^k u\|_{H^{s+\frac{1}{2}}}+\|\tilde{P}\partial_\nu ^k u\|_{H^{\frac{1}{2}+\epsilon-m}} \right)
\leq C \left(\|u\|_{H^{s+k+\frac{1}{2}}}+\|\tilde{P}\partial_\nu ^k u\|_{H^{\frac{1}{2}+\epsilon-m}} \right) .
\end{align}
Now, we write
$$ \tilde{P}\partial_\nu ^ku=\partial_\nu^k \tilde{P}u+[\tilde{P},\partial_\nu^k]u.$$
Since $P$ is elliptic on $\WF(\tilde{P})$, by the elliptic parametrix construction, we can find $E_1\in \Psi^{k}_{\phg}(\Int(X))$, and $E_2\in \Psi^{k-1}_{\phg}(\Int(X))$ so that 
$$ \partial_\nu^k \tilde{P}=E_1P+R_1,\qquad [\tilde{P},\partial_\nu^k]=E_2P+R_2$$
with $R_i\in \Psi^{-\infty}_{\phg}(\Int(X))$. Hence, we obtain
\begin{align*}
\|\tilde{P}\partial_\nu ^k u\|_{H^{\frac{1}{2}+\epsilon-m}}
& \leq C\left(\|E_1Pu\|_{H^{\frac{1}{2}+\epsilon-m}}+\|E_2Pu\|_{H^{\frac{1}{2}+\epsilon-m}}+\|u\|_{H^{s+k+\frac{1}{2}}}  \right)\\
& \leq 
C\left(\|Pu\|_{H^{\frac{1}{2}+k+\epsilon-m}}+\|u\|_{H^{s+k+\frac{1}{2}}}\right) ,
\end{align*}
which, combined with~\eqref{e:estim-loc-3}, yields the result.
\enp

We now proceed with the proof of Proposition \ref{prop: regularity-waves}.
\bnp[Proof of Proposition \ref{prop: regularity-waves}]
First, observe that standard estimates for the Cauchy problem imply that for any $\tilde{\varphi}\in C_c^\infty(\R)$, 
$$\|\tilde{\varphi} u\|_{H^1}\leq C_{T,\tilde{\varphi}}(\|(u_0,u_1)\|_{H^1(M)\times L^2(M)}+\|F\|_{L^2(0,T;L^2(M))}),$$
so we may estimate by terms of the form $\|\tilde{\varphi} u\|_{H^1}$. 

Second, notice that $N^*(\R\times \Int(\Sigma_0))\subset \{\tau=0\}$, so $\Box$ is elliptic on  $N^*(\R\times \Int(\Sigma_0))$ and hence Lemma~\ref{l:restrict} implies 
$$\|\varphi(t)u|_{\Sigma}\|_{\bar{H}^{1/2}(\R\times\Sigma)}+\|\varphi(t)\partial_\nu u|_{\Sigma}\|_{\bar{H}^{-1/2}(\R\times \Sigma)}\leq C(\|\tilde{\varphi}(t)u\|_{H^1}+\|\tilde{\varphi}F\|_{L^2})$$
where $\tilde{\varphi}\in C_c^\infty(\R)$ with $\tilde{\varphi}\equiv 1$ on $\supp \varphi$.

Now, observe that since $\Sigma$ is an interior hypersurface, we may work in a fixed compact subset, $K$ of $\Int(M)$. Note also that there exists $\tilde{\Sigma}$ an interior hypersurface with $\tilde{\Sigma}\subset \Int(\Sigma)$ so that $A=\mathds{1}_{\tilde{\Sigma}}A \mathds{1}_{\tilde{\Sigma}}.$

\medskip
We proceed by making a microlocal partition of unity on a neighborhood  $T^*(\R \times K).$ It suffices to obtain the estimate
\begin{equation}
\label{e:microEst} \|A (\Op(\chi)\varphi(t)u|_{\Sigma})\|_{\bar{H}^1(\R\times \Sigma)}^2 +\|A (\partial_\nu \Op(\chi)\varphi(t)u)|_{\Sigma})\|_{L^2(\R \times \Sigma)}^2  
 \leq C(\|(u_0,u_1)\|_{H^1(M)\times L^2(M)}^2+\|F\|_{L^2}^2),
 \end{equation}
 for $\chi$ supported in a conic neighborhood of an arbitrary point, $q_0 =(t_0,\tau_0,x_0,\xi_0)$ in $T^*(\R\times K).$  
  We will focus on four regions: $q_0\notin \Char(\Box)$ (an elliptic point); $q_0\in \Char(\Box)$ but away from $\Sigma$; $q_0\in T^*_{\R\times \tilde{\Sigma}}(\R\times M) \cap \mc{T}_0$ (a transversal point); and $q_0\in T^*_{\R\times \tilde\Sigma}(\R\times M) \cap \mc{G}_0$ (a glancing point). 

In all regions, we shall use that given $\chi\in S_{\phg}^0$ a cutoff to a conic neighborhood, $U$ of $q_0$, we have
\begin{equation}
\label{e:commute1}\|\Box \Op(\chi)\varphi u\|_{L^2} \leq  \|[\Box ,\Op(\chi)\varphi] u\|_{L^2} +\| \Op(\chi)\varphi \Box u\|_{L^2} \leq C\left(\| \tilde{\varphi} u\|_{H^1}+\|\tilde{\varphi}F\|_{L^2}\right).
\end{equation} 

First start with $q_0$ in the elliptic region: $q_0\notin \Char(\Box)$. Shrinking the neighborhood if necessary, the microlocal ellipticity of $\Box$ near $q_0$ with~\eqref{e:commute1} yields
$$\|\Op(\chi)\varphi  u\|_{H^2}\leq  C\left(\|\tilde{\varphi}u\|_{H^1}+\|\tilde{\varphi}F\|_{L^2}\right) .
$$ 
Hence, rough trace estimates imply
$$\|(\partial_\nu \Op(\chi)\varphi u)|_{\Sigma}\|_{L^2(\R\times \Sigma)}+\|( \Op(\chi)\varphi u)|_{\Sigma}\|_{\bar{H}^1(\R\times \Sigma)}\leq  C\left(\|\tilde{\varphi}u\|_{H^1}+\|\tilde{\varphi}F\|_{L^2}\right) ,
$$
and boundedness of $A$ proves~\eqref{e:microEst} in this case.

Second, suppose that $q_0\in \Char(\Box)$ but $x_0\notin \Sigma$, then clearly there is a neighborhood $U$ of $q_0$ and $\chi$ elliptic at $q_0$ with $\supp \chi \subset U$ so that
$$\|(\partial_\nu \Op(\chi) \varphi u)|_{\Sigma}\|_{L^2(\R\times \Sigma)}+\|(\Op(\chi) \varphi u)|_{\Sigma}\|_{\bar{H}^1(\R\times \Sigma)}\leq C\|\tilde{\varphi}u\|_{L^2}$$
and again boundedness of $A$ proves~\eqref{e:microEst}.

Third, suppose $q_0\in T^*_{\R\times\Sigma}(\R\times M)$ is a transversal point. In that case, we use local Fermi normal coordinates (see Section~\ref{s:fermi}) near $\Sigma_0$ so that $x_0\mapsto (0,0)$. Note that since $q_0\in \Char(\Box)$, we have $\sigma(\Box) (q_0) = - \tau_0^2 + (\xi_0)_1^2 + r(x_0,\xi_0')=0$. Since $q_0 \in \mc{T}_0$, we have moreover $r_0(0,\xi_0')<\tau_0^2$ and hence $\partial_{\xi_1}\sigma(\Box) (q_0)=2(\xi_0)_1\neq 0$. Therefore, by the implicit function theorem, there exist a neighborhood $U$ of $q_0$ and real valued symbols $b(\tau,x,\xi')\in C^\infty((-\e,\e);S^1_{\phg}(T^*(\R\times \{x_1=0\})))$ and $e(\tau,x,\xi)\in S_{\phg}^1(T^*\R\times \R^n)$ elliptic near $q_0$ so that in $U$ we have
$$\sigma(\Box)=e(\tau,x,\xi)(\xi_1-b(\tau,x,\xi')).$$
Thus, letting $\tilde{\chi}\in S^0_{\phg}(\R \times \R^n)$ with $\tilde{\chi}\equiv 1$ on $\supp \chi$ and $\supp \tilde\chi \cap N^*(\{x_1=0\})=\emptyset$ (this is possible since we have $\Char(\Box) \cap N^*(\{x_1=0\})= \emptyset$ and $q_0 \in  \Char(\Box)$, so that we may assume $\supp \chi \cap N^*(\{x_1=0\})=\emptyset$), we have $b \tilde{\chi}\in S^1_{\phg}(T^*\R\times \R^n)$ (see~\cite[Theorem 18.1.35]{Hoermander:V3}) and in particular $\Op(b)\Op(\tilde{\chi})\in \Psi^{1}_{\phg}(\R\times \R^n).$ Therefore,
$$
\Box \Op(\chi)=\Op(e)(D_{x_1}-\Op(b)\Op(\tilde{\chi}))\Op(\chi)+R,
$$
where $R \in \Psi^{1}_{\phg}(\R\times \R^n)$ and hence, using a microlocal parametrix for $\Op(e)$ on $\supp \chi$, we have, using~\eqref{e:commute1}
$$\|(D_{x_1}-\Op(b))\Op(\chi) \varphi u\|_{H^1}\leq C(\|\tilde{\varphi}u\|_{H^1}+\|\tilde{\varphi}F\|_{L^2})$$
and also
\begin{gather*} 
\|(D_{x_1}-\Op(b))\partial_{x_1} (\Op(\chi) \varphi u)\|_{L^2}\leq C(\|\tilde{\varphi}u\|_{H^1}+\|\tilde{\varphi}F\|_{L^2}).
 \end{gather*}
So, by Lemma~\ref{l:energy}, we obtain
\begin{equation}
\label{e:energy1}
\begin{aligned}
\|(\Op(\chi) \varphi u)|_{x_1=0}\|_{H^1}&\leq C(\|\tilde{\varphi}u\|_{H^1}+\|\tilde{\varphi}F\|_{L^2}),\\
\|(\partial_{x_1}\Op(\chi) \varphi u)|_{x_1=0}\|_{L^2}&\leq C(\|\tilde{\varphi}u\|_{H^1}+\|\tilde{\varphi}F\|_{L^2}).
\end{aligned}
\end{equation}
Boundedness of $A$ and \eqref{e:energy1} implies~\eqref{e:microEst}.

Finally, it remains to show that for $q_0\in T^*_{\R\times \Sigma}(\R\times M)$ a glancing point (i.e. with $\tau_0^2-r_0(0,\xi_0')=0$) and $\chi$ supported sufficiently close to $q_0$, we have
$$\| A(\Op\chi \varphi u)|_{x_1=0}\|_{\bar{H}^1(\R\times \Sigma)}+\| A(\partial_{x_1}\Op\chi \varphi u)|_{x_1=0}\|_{L^2(\R\times \Sigma)}\leq \|\tilde{\varphi} u\|_{H^1(\R\times M)}.$$

Let $\psi\in C_c^\infty(\R)$ with $\psi \equiv 1$ near $0$ and define $\psi_\e(x_1)=\psi(\e^{-1}x_1).$ Then define 
$$\tilde{A}_\e u(x_1,x')=[\psi_\e(x_1)Au(x_1,\cdot)](x'), $$
so that $\tilde{A}_\e u|_{x_1=0} =A (u|_{x_1=0}).$ 
Then by \cite[Theorem 18.1.35]{Hoermander:V3}, 
$\tilde{A}_\e \Op(\chi)\varphi(t)\in \Psi_{\phg}^0(\R\times M)$ and for $\e>0$ small enough and $\chi$ supported sufficiently close to $q_0$, $\sigma(\tilde{A}_\e\Op(\chi))=0$. In particular, 
$\tilde{A}_\e \Op(\chi)\varphi(t)\in \Psi_{\phg}^{-1}(\R\times M)$. Similarly, $\tilde{A}_\e \partial_{x_1}\Op(\chi)\varphi(t)\in \Psi_{\phg}^0(\R\times M).$ Rough Sobolev trace estimates thus yield 
\begin{gather*}
\|\tilde{A}_\e \Op(\chi)\varphi(t)u|_{x_1=0}\|_{\bar{H}^1(\R\times \Sigma)}\leq C\|\tilde{\varphi}u\|_{H^{1}(\R\times M)},\\
\|\tilde{A}_\e \partial_{x_1}\Op(\chi)\varphi(t)u|_{x_1=0}\|_{L^2(\R\times \Sigma)}\leq C\|\tilde{\varphi}u\|_{H^{1}(\R\times M)}, 
\end{gather*}
and the proof is finished.
 
\enp

%%%%%%%%%%%%%%%%%%%%%%
\subsection{Microlocal spaces on the hypersurface}
\label{s:spaces-comp-loc}

This section is aimed at defining the appropriate spaces for the statement of the well-posedness and control results in the present context. 
{All along the section, a sequence $\mc{S} = (\eps_j)_{j \in \N}$, $\eps_j \to 0$ is fixed and $\eps , \eps' \in \mc{S}$. This precision is sometimes omitted for concision.}
Fix a family of interior hypersurfaces $\Sigma_\e$ with 
\begin{equation}
\label{e:sigs}
\Sigma_{\e'}\subset \Int(\Sigma_\e)\subset \Sigma_\e\subset \Int(\Sigma),\quad \e<\e',\qquad \bigcup_{\e>0}\Sigma_\e=\Int(\Sigma).
\end{equation}
Let 
\begin{equation}
\label{e:spaceSet} 
\Gamma\subset T^*(\R\times \Int(\Sigma))\setminus 0\text{ be a closed and conic set}.
\end{equation}
We define spaces adapted to $\Gamma$, i.e. measuring different regularities near and away from $\Gamma$. 
In the applications below, we shall take $\Gamma = \mc{G}^\Sigma$ for the study of the Cauchy problem and $\Gamma = \E^\Sigma \cup \mc{G}^\Sigma = \overline{\E}^\Sigma$ for the study of the control problem.

To this end, let $\e\mapsto \Gamma_\e$, $\e \in \mc{S}$, be a family of closed conic subset of $T^*\R\times \Int(\Sigma)\setminus0$ such that
\begin{equation}
\begin{gathered}
\Gamma_\e \text{ is closed and conic for any }\e, \qquad \Gamma_\e\subset \Int(\Gamma_{\e'}),\quad \e<\e',\qquad \Gamma=\bigcap_{\e>0}\Gamma_\e.
\end{gathered}
\end{equation}
Next, fix a family of cutoff functions
\begin{equation}
\label{e:cutoff1}
\varphi_\eps \in C^\infty_c((0,T)\times \Int \Sigma), \qquad \varphi_\e\equiv 1\text{ on }[\e,T-\e]\times \Sigma_\e.
\end{equation}
and a family of cutoff operators
\begin{equation}
\label{e:cutoff}
\begin{gathered}
B^{\Gamma}_\eps \in \Psi^0_{\phg} ((0,T)\times \Int(\Sigma)), \quad B^{\Gamma}_\eps \text{ selfadjoint on } L^2(\R\times \Sigma),\\
 \WF(B^{\Gamma}_\eps)\cap \Gamma_\eps=\emptyset , \quad \WF(\varphi_\e(1-B^{\Gamma}_\eps))\cap {T}_{[\e,T-\e]\times \Sigma_\e}^*(\R\times \Int(\Sigma)) \setminus \Gamma_{2\eps}=\emptyset,\\
 \WF(B^{\Gamma}_{\eps'})\subset \Ell(B^{\Gamma}_\eps),\quad \e<\e' \in \mc{S},\qquad \qquad B^{\Gamma}_\eps\varphi_{\eps}=B^{\Gamma}_{\eps}=\varphi_\e B^\Gamma_\e.
\end{gathered}
\end{equation}
Note that once $\Gamma$ will be fixed, (see Sections~\ref{s:def-sol-WP} and~\ref{s:obs-cont-wave}), a more explicit expression for the symbol of the operators $B^{\Gamma}_\eps$ will be given.

Next, we define for $k \geq s$, the Banach space 
\bna
H^{s,k}_{\comp, \Gamma,\e} (\Sigma_T) = \left\{f \in H_{\comp}^s((0,T) \times \Int(\Sigma)) , 
\supp (f) \subset [\eps,T-\eps] \times \Sigma_\e , (1-B^{\Gamma}_\eps) f \in H_{\comp}^k((0,T) \times\Int( \Sigma))\right\}, 
\ena
normed by 
$$
\|f\|_{H^{s,k}_{\comp, \Gamma,\e} (\Sigma_T)}^2 := \|f \|_{\bar{H}^s([0,T] \times \Sigma)}^2 + \| (1-B^{\Gamma}_\eps) f \|_{\bar{H}^k([0,T] \times \Sigma)}^2
$$
Notice that $(1-B^{\Gamma}_\e)$ measures regularity in $\Gamma_\e$ and therefore, for $f \in H^{s,k}_{\comp, \Gamma,\e}$, we have $f=\varphi_\eps f$ and $\WF^{k}(f) \subset {T}^*(\R\times \Int(\Sigma)) \setminus \Gamma_{\eps}$.  We define the Fr\'echet space
 \begin{gather*}
 H^{s,k}_{\comp,\Gamma} (\Sigma_T)= \bigcup_{\eps >0} H^{s,k}_{\comp, \Gamma,\e} (\Sigma_T)
  =  \left\{f \in H_{\comp}^s((0,T) \times \Int(\Sigma)), \WF^{k}(f) \cap \Gamma = \emptyset\right\}
  \end{gather*}
with topology given by the seminorms $\|\cdot\|_{H^{s,k}_{\comp, \Gamma,\e} (\Sigma_T)}$ (taken for a sequence of $\eps$ going to zero).
Functions/distributions in the space $H^{s,k}_{\comp,\Gamma} (\Sigma_T)$ are $H^s$ overall and microlocally $H^k$ ($k\geq s$) on $\Gamma$. In case $k=s$, we simply have $H^{s,k}_{\comp,\Gamma} (\Sigma_T) = H^{k}_{\comp} ((0,T)\times \Int(\Sigma))$.

\medskip
Similarly, we define for $k \leq s$, the vector space 
\bna
H^{s,k}_{\loc, \Gamma,\e} (\Sigma_T) = \Big\{u \in \D' ((0,T) \times \Int(\Sigma)) , \varphi_\e u \in H_{\comp}^k((0,T) \times \Int(\Sigma)) , B^{\Gamma}_\eps u \in H_{\comp}^s((0,T) \times \Int(\Sigma)) \Big\} ,
\ena
endowed with the seminorm
$$
\|u\|_{H^{s,k}_{\loc,\Gamma, \eps} (\Sigma_T)}^2 := \| \varphi_\e u \|_{\bar{H}^k([0,T]\times \Sigma) }^2 + \| B^{\Gamma}_\eps u \|_{\bar{H}^s([0,T] \times \Sigma)}^2.
$$

We define as well the Fr\'echet space
 \begin{align*}
 H^{s,k}_{\loc,\Gamma} (\Sigma_T)&= \bigcap_{\eps >0} H^{s,k}_{\loc, \Gamma,\e} (\Sigma_T)  \\
 &=  \Big\{f \in \D' ((0,T) \times \Int( \Sigma)) , f \in H_{\loc}^k((0,T) \times \Int( \Sigma)) , B f \in H_{\comp}^s((0,T) \times \Int(\Sigma))\\
&\qquad\text{ for all } B \in \Psi^0_{\phg}((0,T) \times \Int(\Sigma)),\text{ s.t. } \WF(B)\cap \Gamma=\emptyset \Big\} ,
\end{align*}
 with topology given by the seminorms $\|\cdot\|_{H^{s,k}_{\loc, \Gamma,\e} (\Sigma_T)}$.
 Functions/distributions in the space $H^{s,k}_{\loc,\Gamma} (\Sigma_T)$ are locally $H^k$ overall and microlocally $H^s$ ($s\geq k$) outside of  $\Gamma$. Remark again that in case $k=s$, we simply have $H^{s,k}_{\loc,\Gamma} (\Sigma_T) = H^{k}_{\loc} ((0,T)\times \Int(\Sigma))$.

 \begin{lemma}
 \label{l:H-loc-comp}
For $s, k \in \R, s \geq k$ the sesquilinear map 
\bna
C^\infty_c((0,T)\times \Int( \Sigma)) \times C^\infty((0,T)\times \Int(\Sigma))  \to \mathbb{C} , \qquad (f, u) \mapsto \int_{(0,T)\times \Sigma}  f(t,x) \overline{u} (t,x)dtd\sigma(x),
\ena
extends uniquely as a continuous sesquilinear map
$$
H^{-s,-k}_{\comp,\Gamma} \times H^{s,k}_{\loc,\Gamma} \to \mathbb{C} ,
$$
which we shall denote $\langle f, u \rangle_{H^{-s,-k}_{\comp} \times H^{s,k}_{\loc}}$. Moreover, for $(f,u)\in H^{-s,-k}_{\comp,\Gamma,\e} \times H^{s,k}_{\loc,\Gamma,\e}$, we have
$$|\langle f,u\rangle_{H^{-s,-k}_{\comp,\Gamma} \times H^{s,k}_{\loc,\Gamma}}|\leq  \| f\|_{H^{-s,-k}_{\comp , \Gamma, \eps}( \Sigma_T)} \| u \|_{H^{s,k}_{\loc, \Gamma,\eps}(\Sigma_T)} .$$
\end{lemma}
\bnp
Let $(f, u) \in C^\infty_c((0,T)\times \Int(\Sigma)) \times C^\infty((0,T)\times \Int(\Sigma))$.  Fix $\e>0$ so that $\varphi_\e f=f.$ We compute
\begin{align*}
|(f,u)_{L^2((0,T)\times \Sigma)}| &= |( f,\varphi _\e u)_{L^2((0,T)\times \Sigma)}| \\
& \leq \left|\left(  f,  B^{\Gamma}_\eps\varphi_\e  u\right)_{L^2((0,T)\times\Sigma)}  \right|  
+ \left|\left(  f,  (1-B^{\Gamma}_\eps) \varphi_\e u\right)_{L^2((0,T)\times\Sigma)}  \right|   \\
&\leq \| f\|_{\bar{H}^{-s}([0,T]\times \Sigma)} \|B^{\Gamma}_\eps \varphi_\e u \|_{\bar{H}^s([0,T]\times \Sigma)} 
 +\|(1-B^{\Gamma}_\eps)  f\|_{\bar{H}^{-k}([0,T]\times \Sigma)} \|\varphi_\e u \|_{\bar{H}^{k}([0,T]\times \Sigma)}\\
 &\leq \| f\|_{H^{-s,-k}_{\comp ,\Gamma, \eps}( \Sigma_T)} \| u \|_{H^{s,k}_{\loc,\Gamma, \eps}(\Sigma_T)} .
\end{align*}
Then, the density of $C^\infty_c((0,T)\times \Int(\Sigma))$ in $H^{-s,-k}_{\comp,\Gamma}(\Sigma_T)$ and that of $C^\infty((0,T)\times \Int(\Sigma))$ in $H^{s,k}_{\loc,\Gamma}( \Sigma_T)$ prove the statement.
\enp

\begin{lemma}
\label{l:dual}
For all $s,k\in \R$, $k\geq s$, we have $\big(H^{s,k}_{\comp,\Gamma}(\Sigma_T)\big)'=H^{-s,-k}_{\loc,\Gamma}(\Sigma_T)$.
\end{lemma}
\begin{proof}
Lemma~\ref{l:H-loc-comp} proves $H^{-s,-k}_{\loc,\Gamma}(\Sigma_T)\subset (H^{s,k}_{\comp,\Gamma}(\Sigma_T))'.$ Suppose $\mu\in (H^{s,k}_{\comp,\Gamma}(\Sigma_T))'$. Then, since $C_c^\infty((0,T)\times \Int(\Sigma))\subset H^{s,k}_{\comp,\Gamma}(\Sigma_T)$, $\mu\in \mc{D}'((0,T)\times \Int(\Sigma)).$ Fix $\e>0$. Then for $\chi\in C_c^\infty((0,T)\times \Int(\Sigma))$, 
$$|\langle \varphi_\e\mu, \chi\rangle|=|\langle \mu,\varphi_\e \chi\rangle|\leq C_\e (\|\varphi_\e\chi \|_{\bar{H}^s([0,T] \times \Sigma)} + \| (1-B^\Gamma_\eps) \varphi_\e\chi \|_{\bar{H}^k([0,T] \times \Sigma)}).$$
So, since $k\geq s$, we obtain in particular
$$|\langle \varphi_\e \mu, \chi\rangle|\leq C_\e \|\varphi_\e \chi \|_{\bar{H}^k([0,T] \times \Sigma)} \leq C_\e\|\chi\|_{\bar{H}^{k}([0,T]\times \Sigma)}$$
and hence $\varphi_\e \mu\in H^{-k}_{\comp}((0,T)\times \Int \Sigma)$ with 
\bnan
\label{e:H-kbdd}
\|\varphi_\e \mu\|_{\bar{H}^{-k}([0,T]\times \Sigma)}\leq C_\e.
\enan

Fix any $\e \in \mc{S}$ and $\chi\in C_c^\infty((0,T)\times \Int(\Sigma))$. Then there exists $\e_0>0$ depending only on $\e$ such that for $\e_0> \e' \in \mc{S}$, $B^\Gamma_\e\chi\in H^{s,k}_{\comp,\Gamma,\e'}(\Sigma_T)$. Choose $\e'<\e_0$ small enough so that $\WF[(1-B^{\Gamma}_{\eps'})B^{\Gamma}_\eps]=\emptyset$. Then, we have
\begin{align*}
|\langle B^{\Gamma}_{\e}\mu,\chi\rangle|&=|\langle \mu,B^{\Gamma}_\e\chi\rangle|\leq C_{\e'}(\|B^{\Gamma}_{\e}\chi \|_{\bar{H}^s([0,T] \times \Sigma)} + \| (1-B^{\Gamma}_{\eps'}) B^{\Gamma}_{\e}\chi \|_{\bar{H}^k([0,T] \times \Sigma)})\\
&\leq C_{\e'}(\|B^{\Gamma}_{\e}\chi \|_{\bar{H}^s([0,T] \times \Sigma)} + \| \chi \|_{\bar{H}^{-N}([0,T] \times \Sigma)})\\
&\leq C_{\e'} \|\chi\|_{\bar{H}^s([0,T]\times \Sigma)}.
\end{align*}
Therefore, $B^{\Gamma}_\e\mu\in H^{-s}_{\comp}((0,T)\times \Int(\Sigma))$ with 
$$\|B^{\Gamma}_\e\mu\|_{\bar{H}^{-s}([0,T]\times \Sigma)}\leq C_{\e'}.$$
This, together with~\eqref{e:H-kbdd} proves that $\mu \in H^{-s,-k}_{\loc,\Gamma}(\Sigma_T)$, and hence the lemma.
\end{proof}

%%%%%%%%%%%%%%%%%%%%%%%%%%%
\subsection{Definition of solutions and well-posedness}
\label{s:def-sol-WP}

Observe that $\GS$ and $\ES:=\GS\cup\mc{E}^\Sigma$ satisfy~\eqref{e:spaceSet} and for $k\geq s$, we therefore have Fr\'echet spaces $H^{s,k}_{\comp,\GS}(\Sigma_T)$, $H^{s,k}_{\comp,\ES}(\Sigma_T)$ with dual spaces $H^{-s,-k}_{\loc,\GS}(\Sigma_T)$, $H^{-s,-k}_{\loc,\ES}(\Sigma_T)$.

With these definitions in hand, we can reformulate Proposition~\ref{prop: regularity-waves} as follows: 
For any $T>0$, the map 
\begin{equation}
\label{e:regularityState}
\begin{aligned}
H^1_0(M) \times L^2(M) \times L^2(0,T;L^2(M))  &\to   H^{1,\frac12}_{\loc,\GS}(\Sigma_T )\times H^{0,-\frac12}_{\loc,\GS}(\Sigma_T) \\
(u_0,u_1, F) &\mapsto (u|_{\Sigma} , \d_\nu u|_{\Sigma})
\end{aligned}
\end{equation}
(where $u$ is solves \eqref{e:waveObs}) is continuous.

 We can now study the well-posedness for the control problem~\eqref{e:waveControl}.
We first recall that, given $f_0,\,f_1\in C^\infty_c (\R\times \Sigma)$, $f_0 \delta_\Sigma$ and $f_1 \delta_\Sigma'$ are usual distributions defined by~\eqref{e:defDist}.
\begin{lemma}
\label{l:IPP-smooth}
Given $T>0$, assume that the functions $v\in C^\infty([0,T]\times M\setminus \Sigma)\cap C^1((0,T); L^2(M))$ $u,F \in  C^\infty([0,T]\times M)$ and $f_0, f_1 \in C^\infty_c((0,T)\times \Int(\Sigma))$ solve
$$
\Box v = f_0 \delta_\Sigma +  f_1 \delta_\Sigma' \text{ in } \D'((0,T)\times \Int(M)), \quad \text{ and }\quad \Box u = F .
$$
Then, we have the identity
$$
\left[ (\d_t v, u)_{L^2(M)} - ( v,\d_t u)_{L^2(M)} \right]_0^T + (v,F)_{L^2((0,T)\times M)} 
= \int_{(0,T)\times\Sigma} \Big(f_0 u|_{\Sigma}  - f_1 \d_\nu u|_{\Sigma} \Big) dtd\sigma .
$$
\end{lemma}
The duality property of Lemma~\ref{l:H-loc-comp}, together with the formula of Lemma~\ref{l:IPP-smooth}, valid for smooth functions, and~\eqref{e:regularityState} suggest that taking $f_0 \in H^{-1, -\frac12}_{\comp,\GS}(\Sigma_T)$ and $f_1 \in H^{0,\frac12}_{\comp,\GS}(\Sigma_T)$ could be an appropriate set of spaces for control functions, as well as the following definition of transposition solutions for the control problem.

\begin{definition}
\label{d:transp-sol}
Given $T>0$, $(v_0,v_1) \in L^2(M) \times H^{-1}(M), f_0 \in H^{-1, -\frac12}_{\comp,\GS}(\Sigma_T), f_1 \in H^{0,\frac12}_{\comp,\GS}(\Sigma_T)$, we say that $v$ is a solution of \eqref{e:waveControl} if $v \in L^2((0,T); L^2(M))$ and for any $F \in L^2((0,T);L^2 (M))$, we have
\bna
\int_0^T  (v,F)_{L^2(M)} dt& = &
\langle v_1 , u(0) \rangle_{H^{-1}(M), H^1(M)} - ( v_0 ,\d_t u(0))_{L^2(M)}  \\
&& +  \langle f_0 , u|_{\Sigma} \rangle_{H^{-1, -\frac12}_{\comp,\GS}(\Sigma_T), H^{1,\frac12}_{\loc,\GS}(\Sigma_T)} 
 - \langle f_1 , \d_\nu u|_{\Sigma} \rangle_{H^{0,\frac12}_{\comp,\GS}(\Sigma_T) , H^{0,-\frac12}_{\loc,\GS}(\Sigma_T)} .
\ena
where $u$ is the unique solution to
\begin{equation}
\begin{cases}
\label{e:F-IC=0}
\Box u=F &\text{on }(0,T)\times \Int(M)\\
(u,\partial_tu)|_{t=T}=(0,0)&\text{in } \Int(M) .
\end{cases}
\end{equation}
\end{definition}
Note in particular that taking $F \in C^\infty_c((0,T)\times\Int( M))$ implies that such a solution is a solution of the first equation of \eqref{e:waveControl} in the sense of distributions.

\begin{theorem}
\label{t:well-posedness}
Let $T>0$. For all $(v_0,v_1)\in L^2(M) \times H^{-1}(M)$ and for all $f_0 \in H^{-1, -\frac12}_{\comp,\GS}(\Sigma_T)$ and $f_1 \in H^{0,\frac12}_{\comp,\GS}(\Sigma_T)$, there exists a unique $v \in L^2((0,T); L^2(M))$ solution of \eqref{e:waveControl} in the sense of Definition~\ref{d:transp-sol}. The linear map
\bna
L^2(M) \times H^{-1}(M) \times H^{-1, -\frac12}_{\comp,\GS}(\Sigma_T)\times H^{0,\frac12}_{\comp,\GS}(\Sigma_T) & \to &  L^2(0,T;L^2(M))\\
(v_0,v_1, f_0 , f_1) &\mapsto& v 
\ena
is continuous. 
\end{theorem}
\begin{remark}
Note that, given two different times $T <T'$, an initial data $(v_0,v_1)$ and control functions $f_0 , f_1$ compactly supported in $(0,T)\subset (0,T')$, the above definition/theorem yield two different solutions: one defined on $(0,T)$ and one defined on $(0,T')$. However, one can observe that these two solutions coincide by extending all test functions $F\in L^2((0,T); L^2(M))$ by zero on $(T,T')$ to obtain test functions in $L^2((0,T'); L^2(M))$. With this in mind, Theorem~\ref{t:well-posedness-simple} is a direct consequence (and a simplified version) of Theorem~\ref{t:well-posedness}.
\end{remark}

\bnp[Proof of Theorem~\ref{t:well-posedness}]
First, we define 
 \bna
 \ell (F) &: = &
\langle v_1 , u(0) \rangle_{H^{-1}(M), H^1(M)} - ( v_0 ,\d_t u(0))_{L^2(M)}  \\
&& +  \langle f_0 , u|_{\Sigma} \rangle_{H^{-1, -\frac12}_{\comp,\GS}(\Sigma_T), H^{1,\frac12}_{\loc,\GS}(\Sigma_T)} 
 - \langle f_1 , \d_\nu u|_{\Sigma} \rangle_{H^{0,\frac12}_{\comp,\GS}(\Sigma_T) , H^{0,-\frac12}_{\loc,\GS}(\Sigma_T)} ,
 \ena
 and prove that it is as a continuous linear form on $L^2(0,T;L^2(M))$, with appropriate norm. We have
\bna 
|\ell (F) | &\leq & \|v_1\|_{H^{-1}} \| u(0)\|_{H^1_0} + \| v_0\|_{L^2(M)} \| \d_t u(0)\|_{L^2(M)}+R \\
& \leq &  \|(v_0, v_1)\|_{L^2 \times H^{-1}} \| F\|_{L^2(0,T;L^2(M))} +R,
\ena
with 
$$
R = \left| \langle f_0 , u|_{\Sigma} \rangle_{H^{-1, -\frac12}_{\comp,\GS}(\Sigma_T), H^{1,\frac12}_{\loc,\GS}(\Sigma_T)} 
 - \langle f_1 , \d_\nu u|_{\Sigma} \rangle_{H^{0,\frac12}_{\comp,\GS}(\Sigma_T) , H^{0,-\frac12}_{\loc,\GS}(\Sigma_T)} \right| .
$$
From the definition of the spaces in Section~\ref{s:spaces-comp-loc}, there exists $\eps>0$ such that $(f_0 , f_1) \in H^{-1, -\frac12}_{\comp,\GS, \eps}(\Sigma_T)\times H^{0,\frac12}_{\comp,\GS, \eps}(\Sigma_T)$ and hence, we obtain from Lemma~\ref{l:H-loc-comp},
$$
R \leq  \|f_0\|_{H^{-1, -\frac12}_{\comp,\GS, \eps}(\Sigma_T)}    \|u|_{\Sigma}\|_{H^{1,\frac12}_{\loc,\GS,\eps}(\Sigma_T)} 
 + \| f_1\|_{H^{0,\frac12}_{\comp,\GS,\eps}(\Sigma_T)}\| \d_\nu u|_{\Sigma} \|_{H^{0,-\frac12}_{\loc,\GS, \eps}(\Sigma_T)}  .
$$
Proposition~\ref{prop: regularity-waves} with $A=1-B^{\mc{G}^\Sigma}_\e$ (satisfying the appropriate conditions) then yields
$$
R \leq C_\eps \left( \|f_0\|_{H^{-1, -\frac12}_{\comp,\GS, \eps}(\Sigma_T)}    
 + \| f_1\|_{H^{0,\frac12}_{\comp,\GS,\eps}(\Sigma_T)} \right)\| F\|_{L^2(0,T;L^2(M))} .
$$

Coming back to $\ell$, we have obtained the existence of $\eps \in \mc{S} , C_\eps>0$ such that 
\bna 
|\ell (F) | \leq  C_\eps \left(  \|(v_0, v_1)\|_{L^2 \times H^{-1}} + C_\eps \|(f_0 ,  f_1)\|_{H^{-1, -\frac12}_{\comp,\GS, \eps}(\Sigma_T) \times H^{0,\frac12}_{\comp,\GS,\eps}(\Sigma_T)}  \right) \| F\|_{L^2(0,T;L^2(M))} .
\ena
Hence, $\ell$ is a continuous linear form on $L^2(0,T;L^2(M))$. There is thus a unique $v\in L^2(0,T;L^2(M))$ such that $\ell(F) = \int_0^T (v(t) , F(t) )_{L^2(M)} dt$ for all $F\in L^2(0,T;L^2(M))$, that is precisely the definition of a solution of \eqref{e:waveControl}  in Definition~\ref{d:transp-sol}. This solution moreover satisfies, for $(f_0 ,  f_1) \in H^{-1, -\frac12}_{\comp,\GS, \eps}(\Sigma_T) \times H^{0,\frac12}_{\comp,\GS,\eps}(\Sigma_T)$, the estimate
$$
\|v\|_{L^2(0,T; L^2(M))} \leq  \|(v_0, v_1)\|_{L^2 \times H^{-1}} + C_\eps \|(f_0 ,  f_1)\|_{H^{-1, -\frac12}_{\comp,\GS, \eps}(\Sigma_T) \times H^{0,\frac12}_{\comp,\GS,\eps}(\Sigma_T)} ,
$$
which is the continuity statement.
This concludes the proof of the Theorem.
\enp

%%%%%%%%%%%%%%%%%%%%%%
\section{Observability and controllability for the wave equation}
\label{s:obs-cont-wave}
The aim of this section is to study the observability of~\eqref{e:waveObs} from $\Sigma$. In particular, we prove
\begin{theorem}[Observability]
\label{thm:observe}
Let $\chi\in C^\infty(\R)$ have $\chi \equiv 1$ on $(-\infty,-1]$ and $\supp \chi \subset (-\infty,-\frac{1}{2}]$.
Under Assumption \asref{GC}{T}, there exists $\delta_0>0$, so that for all $\delta\in (0, \delta_0)$, all $A_\delta \in \Psi^0_{\phg}((0,T) \times \Int(\Sigma))$ with principal symbol $\chi \big(\frac{r_0(x', \xi')-\tau^2}{\delta \tau^2} \big) \varphi_\delta$, where $\varphi_\delta\in C_c^\infty((0,T)\times \Int(\Sigma))$ with $\varphi_\delta\equiv 1$ on $[\delta,T-\delta]\times \Sigma_{\delta}$ where $\Sigma_{\delta}$ is as in~\eqref{e:sigs}, for all $N>0$, there exists $c_N>0$ so that for any solution $u$ to~\eqref{e:waveObs}, we have
\begin{equation}
\label{e:obs-waves}
\begin{aligned}
c_N\|(u_0,u_1)\|_{H^1\times L^2}^2& \leq \| \varphi_\delta \partial_\nu u|_{\Sigma_0}\|_{\bar{H}^{-N}(\R\times \Sigma)}^2+\|\varphi_\delta u|_{\Sigma_0}\|_{\bar{H}^{-N}(\R\times \Sigma)}^2\\
&\qquad +\|A_\delta (\partial_\nu u|_{\Sigma_0})\|_{L^2(\R\times \Sigma)}^2+\|A_\delta(u|_{\Sigma_0})\|_{\bar{H}^1(\R\times \Sigma)}^2 +\|F\|_{L^2((0,T)\times M)}.
\end{aligned}
\end{equation}
\end{theorem}
Let us briefly explain why the observability inequality of Theorem \ref{thm:observe} implies Theorem \ref{t:uniqueControl}.

\bnp[Proof of  Theorem \ref{t:uniqueControl}]
We apply Theorem \ref{thm:observe} to the function $u(t,x)=e^{it\lambda}v(x)$ with $v\in H^1_0(M)\cap H^2(M)$. First observe that $A_{\delta}$ is bounded on $L^2$ and hence 
$$\| A_\delta \partial_\nu u|_{\Sigma}\|_{L^2(\R\times \Sigma)}\leq C\|\partial_\nu u|_{\Sigma}\|_{L^2([0,T]\times \Sigma)}\leq C\|\partial_\nu v|_{\Sigma}\|_{L^2(\Sigma)}.$$
 Observe also that there exists $\delta_0>0$ so that 
$\varphi_{\delta_0}D_t$ is elliptic on $\WF(A_\delta)$ and therefore,  
$$ \| A_\delta u|_{\Sigma}\|_{\bar{H}^1(\R\times \Sigma)}\leq C(\| \varphi_{\delta_0}D_t u\|_{L^2(\R\times \Sigma)} +\|u|_{\Sigma}\|_{L^2([0,T]\times \Sigma)}\leq C\langle \lambda\rangle \|v\|_{L^2(\Sigma)}.$$
Note also that 
$$\Box u=e^{it\lambda}(-\Delta_g-\lambda^2)v$$
and hence the right hand side of \eqref{e:obs-waves} is bounded by
$$C\left(\|\partial_\nu v|_{\Sigma}\|_{L^2}+ \langle \lambda\rangle  \|v|_{\Sigma}\|_{L^2} +\|(-\Delta_g-\lambda^2)v\|_{L^2((0,T)\times M)}\right).$$
Finally, noticing that 
$$(u|_{t=0},\partial_t u|_{t=0})=(v,i\lambda v),$$
gives
$$\langle \lambda \rangle\|u\|_{L^2(M)}\leq \|(u|_{t=0},\partial_t u|_{t=0})\|_{H_0^1(M)\times L^2(M)},$$
finishing the proof of Theorem \ref{t:uniqueControl}.
\enp

%%%%%%%%%%%%%%%%%%%%%%
\subsection{The geometric assumption $\mc{T}$GCC}
\label{s:TGCC}

To prove Theorem \ref{thm:observe} we start with a dynamical lemma where we show that the {\em a priori} weaker assumption \asref{GC}{$T$} implies the stronger assumption
\begin{assume}{$\e$}{$T$}\label{GC2}
For all $p\in Z$, we have 
$$\bigcup\limits_{s\in \R}\{\varphi(s,p)\}\cap \mc{T}_\e\cap T_{(\e,T-\e)\times \Sigma_\e}^*(\R\times M)\neq \emptyset.$$
\end{assume}
\noindent Recall that $Z$ is as in~\eqref{e:defZ}.

\begin{lemma}
\label{l:GC}
Suppose that Assumption \asref{GC}{$T$} holds. Then there exists $\e>0$ so that Assumption \asref{GC2}{$T$} holds.
\end{lemma}
\begin{proof}
We define $Z^{\pm}_1:=Z\cap \{\tau=\pm 1,t=0\}$. We shall show that Assumption \asref{GC}{$T$} implies the existence of $\e>0$ such that 
\begin{equation}
\label{e:platypus}
\bigcup\limits_{s\in \R}\{\varphi(s,p)\}\cap \mc{T}_\e\cap T_{(\e,T-\e)\times \Sigma_\e}^*(\R\times M)\neq \emptyset,\quad \text{ for all } p\in Z_1^{\pm}.
\end{equation}
We first show that~\eqref{e:platypus} implies the lemma. 
With the identification $\TRM \simeq T^*\R \times \TM$, consider $p=(t_0,\tau,q)\in (T^*\R \times \TM)\cap Z$. Let $M_\lambda$ be multiplication in the fiber by $\lambda>0$. Then, 
$$p'=\varphi(  t_0\sgn \tau ,M_{|\tau|^{-1}}(t_0,\tau,q))\in Z^{+}_1\cup Z^-_1.$$
According to the homogeneity of $\varphi$, see~\eqref{e:flow-hom}, and the flow property~\eqref{e:phi-flowprop}, we have
 \bnan
 \label{e:UUU}
 \bigcup\limits_{s\in \R}\{M_{|\tau|^{-1}}\varphi(s,p)\}=\bigcup\limits_{s\in \R}\{\varphi(s|\tau|,M_{|\tau|^{-1}}p)\}=\bigcup\limits_{s\in \R}\{\varphi(s|\tau|,\varphi(t_0\sgn \tau,M_{|\tau|^{-1}}p))\}=\bigcup\limits_{s\in \R}\{\varphi(s,p')\}.
 \enan
 But, by~\eqref{e:platypus}, since $p'\in Z_1^+\cup Z_1^-$, we have
 $$\bigcup\limits_{s\in \R}\{\varphi(s,p')\}\cap \mc{T}_\e\cap T_{(\e,T-\e)\times \Sigma_\e}^*(\R\times M)\neq \emptyset,
 $$
 and hence homogeneity of $\varphi, \mc{T}_\e$ and $T_{(\e,T-\e)\times \Sigma_\e}^*(\R\times M)$ together with~\eqref{e:UUU} completes the proof of the lemma from~\eqref{e:platypus}.
 
\bigskip
We now prove~\eqref{e:platypus}, writing explicitly the argument for $Z_1^-$. The case of $Z_1^+$ is handled similarly. Notice first that since $\varphi$ is the generalized bicharacteristic flow for $\frac{1}{2}(-\tau^2+|\xi|_g^2)$, we have for $p\in Z^-_1$, $t(\varphi(s,p))=s$. This, together with Assumption \asref{GC}{$T$} implies that for each $p\in Z_1^-$, we have
$$\bigcup\limits_{s\in (0,T)}\{\varphi(s,p)\}\cap \mc{T}_0\neq \emptyset.$$
Therefore, for each $p\in Z_1^-$, there exists $\e_p>0$ and $s_p\in (\e_p,T-\e_p)$ such that 
$$\varphi(s_p,p)\in \mc{T}_{\e_p}\cap T_{(\e_p,T-\e_p)\times \Sigma_{\e_p}}^*(\R\times M).$$
  Let $\beta$ be a defining function for $\Sigma_0$ near $\varphi(s_p,p)$, and consider $g(s,q)=\beta\circ \pi_0 \circ \varphi(s,q)$ for $(s,q)$ in a neighborhood $N_p$ of $(s_p, p)$, where $\pi_0 : T^*(\R \times \Int(M)) \to \R \times \Int(M)$ is the canonical projection. By \cite[Theorem~3.34]{MS:78}, the Melrose--Sj\"ostrand generalized bicharacteristic flow $\varphi$ is continuous and so $g$ is continuous on $N_p$. 
  
  Moreover, since $\Sigma$ is an interior hypersurface, there exists $\delta_p$ so that 
  $$g(\cdot, q):(s_p-\delta_p,s_p+\delta_p)\to j(T^*(\R\times \Int(M))\cap \Char (\Box))\subset Z$$
   is $C^1$ for $q$ in a neighborhood of $p$ since $\varphi$ coincides with the usual bicharacteristic flow of $\Box$ near $\varphi(s_p,p)$.
   
Notice that $\varphi(s_p,p)\in \mc{T}_{\e_p}$ implies $$\partial_sg(s_p,p) = \langle d\beta (\pi_0 \circ \varphi(s_p,p)) d \pi_0 (\varphi(s_p,p)) , H_{\sigma(\Box)}(\varphi(s_p,p)) \rangle\neq 0 .$$
 according to Remark~\ref{rem:transverse}. Hence by the implicit function theorem \cite{Kum}, the equation $g(s,q)=0$ defines a continuous function $s=s(q)$ near $q=p$. In particular, set 
$$\delta_0=\min\Big( \frac{s_p}{2}, \frac{T-s_p}{2}\Big).$$
 Then there is a neighborhood, $U_p$ of $p$ and a continuous function $s:U_p\to \R$ with $s(p)=s_p$, such that $\varphi_{s(q)}(q)\in \mc{T}_{\e_p/2}\cap T_{(\e_p/2,T-\e_p/2)\times \Sigma_{\e_p/2}}^*(\R\times M)$ and $|s(q)-s_p|<\delta_0$ for all $q \in U_p$. 

Since 
$$Z_1^-=j(\Char(\Box)\cap \{\tau=-1,t=0\})$$
 is compact, we may extract from the cover $Z_1^- \subset \bigcup_{p \in Z_1^-}U_p$ a finite cover $\{U_{p_i}\}_{i=1}^n$. Then taking $\e=\min_{1\leq i\leq n}\e_{p_i}/2$, we have that for all $p\in Z_1^-$, 
$$\bigcup_{s\in (0,T)} \{\varphi(s,p)\}\cap \mc{T}_\e\cap T_{(\e,T-\e)\times \Sigma_\e}^*(\R\times M)\neq \emptyset.$$
In particular, \eqref{e:platypus} holds, which concludes the proof of the lemma.
\end{proof}

%%%%%%%%%%%%%%%%%%%%%%
\subsection{Observability at High Frequency}
\label{s:weakObs}

The aim of this section is to prove the following proposition, which is a high-frequency version of Theorem~\ref{thm:observe}.
 The estimates in these results differ in two respects: here in Proposition~\ref{e:obsTemp} there is $\|(u_0,u_1)\|^2_{L^2\times H^{-1}}$ in the right handside, so that this estimate does not care about low frequencies. The treatment of low frequencies in Theorem \ref{thm:observe} (see Section~\ref{s:obs} below) requires the addition of observation terms in an arbitrary weak norm $\| \varphi_\delta \partial_\nu u|_{\Sigma_0}\|_{\bar{H}^{-N}(\R\times \Sigma)}^2+\|\varphi_\delta u|_{\Sigma_0}\|_{\bar{H}^{-N}(\R\times \Sigma)}^2$, which is not needed here.

{
\begin{proposition}
\label{tempThm}
Let $\chi\in C^\infty(\R)$ have $\chi \equiv 1$ on $(-\infty,-1]$ and $\supp \chi \subset (-\infty,-\frac{1}{2}]$.
Under Assumption \asref{GC}{T}, there exists $\delta_0>0$, so that for all $\delta\in (0, \delta_0)$, all $A_\delta \in \Psi^0_{\phg}((0,T) \times \Int(\Sigma))$ with principal symbol $\chi \big(\frac{r_0(x', \xi')-\tau^2}{\delta \tau^2} \big) \varphi_\delta$, where $\varphi_\delta\in C_c^\infty((0,T)\times \Int(\Sigma))$ with $\varphi_\delta\equiv 1$ on $[\delta,T-\delta]\times \Sigma_{\delta}$ where $\Sigma_{\delta}$ is as in~\eqref{e:sigs}, 
there exists $c>0$ so that for any solution $u$ to \eqref{e:waveObs}, we have
\begin{equation}
\label{e:obsTemp}
c\|(u_0,u_1)\|^2_{H^1\times L^2} \leq \|A_\delta(u|_{\Sigma_0})\|^2_{\bar{H}^1(\R\times \Sigma)} +\|A_{\delta}(\partial_\nu u|_{\Sigma_0})\|^2_{L^2(\R \times \Sigma)}+\|F\|_{L^2((0,T)\times M)}^2+\|(u_0,u_1)\|^2_{L^2\times H^{-1}} .
\end{equation}
\end{proposition}
}
We begin with two preliminary lemmas.
We again work in fermi normal coordinates near $\Sigma$. A more general of version of the following Lemma is given in \cite[Lemma 23.2.8]{Hoermander:V3}, but we decided to include a short proof in this particular context for the sake of readability.
\begin{lemma}
\label{l:factor}
Denote $\Box = -D_t^2 +D_{x_1}^2+ r(x, D_{x'})+c(x)D_{x_1}$, where $r$ is defined in Section~\ref{s:fermi}.
For any $0<\nu< 1$, there exist $\epsilon>0$ and $\Lambda_\pm ,\tilde{\Lambda}_\pm\in C^\infty((-\epsilon,\epsilon);\Psi_{\phg}^1(\R\times \R^{n-1}))$  with 
$$\sigma(\Lambda_{\pm})=\sigma(\tilde{\Lambda}_{\pm})=\sqrt{\tau^2-r(x,\xi')}\text{ on } \{ \tau^2-r(x,\xi')\geq \nu \tau^2\}$$
such that for all $b\in C_c^\infty \big((-\epsilon,\epsilon);S^0_{\phg}(T^*(\R\times \R^{n-1}))\big)$ with
$\supp b\subset \{ \tau^2-r(x,\xi')\geq \nu \tau^2\},$
\begin{gather*} 
\Op(b)\Box=\Op(b)\big[(D_{x_1}-\Lambda_-)(D_{x_1}+\Lambda_+)+R\big]\\
\Op(b)\Box=\Op(b)\big[(D_{x_1}+\tilde{\Lambda}_+)(D_{x_1}-\tilde{\Lambda}_-)+\tilde{R}\big]
\end{gather*}
where $R,\tilde{R}\in C^\infty((-\epsilon,\epsilon);\Psi^{-\infty}_{\phg}(\R \times\R^{n-1}))$.
\end{lemma}
\begin{proof}
Throughout this proof, we will write $S^k_{\tan}$ for $C^\infty\big((-\epsilon,\epsilon);S^k_{\phg}(T^*(\R\times\R^{n-1}))\big)$ and $\Psi^k_{\tan}$ for the corresponding quantization $C^\infty((-\epsilon,\epsilon);\Psi^{k}_{\phg}(\R\times \R^{n-1}))$. {For an operator $A\in \Psi_{\tan}^\infty$, we will write 
$$
\WF(A)=\bigcup_{x_1}\WF(A_{x_1}), 
$$
where $A_{y_1}$ is the pseudodifferential operator acting on $\R^{n-1}$ at $x_1=y_1$. }

Given $0<\nu < 1$, we let $\check{\chi}(t,x,\tau,\xi')\in S^0_{\tan}$ with 
$$\check{\chi}\equiv 1\text{ on }\{\tau^2-r(x,\xi')\geq \nu\tau^2/3\},\qquad \supp \check{\chi}\subset \{\tau^2-r(x,\xi')\geq \nu\tau^2/4\}.$$
Then, for $(t,x,\tau,\xi')\in \supp \check{\chi}$, we have the following factorization
$$\sigma(\Box)=-\tau^2 + \xi_1^2 + r(x,\xi')=\Big[\xi_1+\sqrt{\tau^2-r(x,\xi')}\Big]\Big[\xi_1-\sqrt{\tau^2-r(x,\xi')}\Big].$$
We thus let $\lambda_0(t,x,\tau,\xi')=\sqrt{\tau^2-r(x,\xi')}$ and $\Lambda_0=\Op(\check{\chi}\lambda_0)$.

Now, write $D_{x_1}=D_{x_1}-\Lambda_0+\Lambda_0$ so that 
\begin{align*} 
\Box &= (D_{x_1}-\Lambda_0)D_{x_1} +\Lambda_0D_{x_1}-D_{t}^2+r(x,D_{x'})+(D_{x_1}-\Lambda_0)c(x)+[c(x),D_{x_1}]+\Lambda_0 c(x)\\
&=(D_{x_1}-\Lambda_0)(D_{x_1}+\Lambda_0+c(x))+\Lambda_0^2 + [\Lambda_0+c(x),D_{x_1}]-D_t^2+r(x,D_{x'})+\Lambda_0c(x)\\
&=(D_{x_1}-\Lambda_0)Q_0+\tilde{R}_0 ,  
\end{align*}
where 
\begin{align*}
Q_0& =D_{x_1}+\Lambda_0+c(x) \in \Psi_{\tan}^{0} D_{x_1} + \Psi_{\tan}^{1} , \\ 
\tilde{R}_0& =\Lambda_0^2 + [\Lambda_0+c(x),D_{x_1}]-D_t^2+r(x,D_{x'})+\Lambda_0c(x) \in \Psi_{\tan}^2.
\end{align*} 
First, we remark that $\sigma(Q_0)=\xi_1+\check\chi\lambda_0$. Second, noting that $\sigma(\tilde{R}_0)$ is independent of $\xi_1$, we take $\xi_1=\lambda_0$ on $\check{\chi}\equiv 1$ in
$$
\sigma(\Box)=\xi_1^2-\tau^2+r(x,\xi')=(\xi_1-\lambda_0\check{\chi})\sigma(Q_0)+\sigma(\tilde{R}_0)
$$ 
to obtain $\sigma(\tilde{R}_0)=0$ on that set. This yields $\tilde{R}_0=R_0+E_0$ with $R_0\in \Psi_{\tan}^{1}$ and $E_0\in \Psi^2_{\tan}$ with $\WF(E_0)\cap \{ \tau^2-r(x,\xi')\geq \nu \tau^2\} =\emptyset$. Indeed for $\chi_1\in S^0_{\tan}$ with $\supp \chi_1\subset \{\check\chi\equiv 1\}$ and $\chi_1\equiv 1$ in a neighborhood of $\tau^2-r(x,\xi')\geq \nu \tau^2\}$, $\sigma(\Op(\chi_1)\tilde{R}_0)=0$. Thus, 
\begin{gather*}
\tilde{R}_0=E_0+R_0,\qquad\qquad E_0=\Op(\chi_1)\tilde{R}\in \Psi^{1}_{\tan},\qquad R_0=\Op(1-\chi_1)\tilde{R}.
\end{gather*}
 This implies the first factorization formula with $R\in \Psi_{\tan}^1$. We now proceed with an induction to improve this remainder term.

Suppose we have for some $j\geq 0$
\begin{equation}
\label{e:intFactor}\Box=(D_{x_1}-\Lambda_{-,j})(D_{x_1}+\Lambda_{+,j})+R_j +E_j
\end{equation}
with $\Lambda_{\pm,j}\in \Psi_{\tan}^1$, with principal symbol $\lambda_0\check{\chi}$, $R_j \in \Psi_{\tan}^{1-j}$, and $E_j\in \Psi^2_{\tan}$ with $\WF(E_j)\cap \{ \tau^2-r(x,\xi')\geq \nu \tau^2\}  =\emptyset.$ 
Now,  we want to adjust $\Lambda_{\pm,j}$ to improve the error $R_j$. Let $\lambda_{j+1}\in S_{\tan}^{-j}$ have 
\begin{align}
\label{e:symbol-j}
\sigma(R_j)+\lambda_{j+1}\sigma(\Lambda_{+,j}+\Lambda_{-,j}) = 0 \quad \text{ in a neighborhood of } \check{\chi}\equiv 1 .
\end{align}
This is possible since $\sigma(\Lambda_{\pm,j})=\check{\chi}\lambda_0$ is elliptic on a neighborhood of $\check{\chi}\equiv 1$.

Now, observe that 
\begin{align*}
\Op(\lambda_{j+1})(D_{x_1}+\Lambda_{+,j})+R_j&=\Op(\lambda_{j+1})(\Lambda_{-,j}+\Lambda_{+,j})+R_j+\Op(\lambda_{j+1})(D_{x_1}-\Lambda_{-,j})\\
&=\Op(\lambda_{j+1})(\Lambda_{-,j}+\Lambda_{+,j})+R_j+(D_{x_1}-\Lambda_{-,j})\Op(\lambda_{j+1}) \\
& \quad +[\Op(\lambda_{j+1}),D_{x_1}-\Lambda_{-,j}]\\
&=R_{j+1,1}+E_{j+1,1}+(D_{x_1}-\Lambda_{-,j})\Op(\lambda_{j+1}) ,
\end{align*}
where, according to~\eqref{e:symbol-j}, we have $R_{j+1,1}\in \Psi^{-j}_{\tan}$ and $E_{j+1,1}\in \Psi^{1-j}_{\tan}$ with $\WF(E_{j+1,1})\cap \{\check{\chi}=1\}=\emptyset.$ So, coming back to~\eqref{e:intFactor}, we now obtain
\begin{align*}
\Box&=(D_{x_1}-\Lambda_{-,j})(D_{x_1}+\Lambda_{+,j})+R_j+E_{j}\\
&=(D_{x_1}-\Lambda_{-,j})(D_{x_1}+\Lambda_{+,j})-\Op(\lambda_{j+1})(D_{x_1}+\Lambda_{+,j})+E_{j+1,1}+R_{j+1,1}+(D_{x_1}-\Lambda_{-,j})\Op(\lambda_{j+1})\\
&=(D_{x_1}-\Lambda_{-,j}-\Op(\lambda_{j+1}))(D_{x_1}+\Lambda_{+,j})+E_{j+1,1}+R_{j+1,1} \\
& \quad +(D_{x_1}-\Lambda_{-,j}-\Op(\lambda_{j+1}))\Op(\lambda_{j+1})+\Op(\lambda_{j+1})^2\\
&=(D_{x_1}-\Lambda_{-,j}-\Op(\lambda_{j+1}))(D_{x_1}+\Lambda_{+,j}+\Op(\lambda_{j+1}))+E_{j+1}+R_{j+1}
\end{align*}
where $R_{j+1}\in \Psi^{-j}_{\tan}$ and $E_{j+1}\in \Psi^{1-j}_{\tan}$ with $\WF(E_{j+1})\cap \{\check{\chi}=1\}=\emptyset$. Putting $\Lambda_{-,j+1}=\Lambda_{-,j}+\Op(\lambda_{j+1})$ and $\Lambda_{+,j+1}=\Lambda_{+,j}+\Op(\lambda_{j+1})$, we have~\eqref{e:intFactor} with $j$ replaced by $j+1$. Since we modified $\Lambda_{-,j}$ and $\Lambda_{+,j}$ by terms in $\Psi^{-j}$, summing asymptotically and composing on the left with $\Op(b)$ gives the desired result. Repeating the argument but starting with $D_{x_1}+\Lambda_0$ on the left, we obtain the second factorization. 
\end{proof}

\begin{lemma}
\label{l:mushrooms}
Let $q_0\in \mc{T}_0\cap T^*((0,T)\times \R^n)$. Suppose $b_0\in S^0_{\phg}\big(T^*((0,T)\times \R^{n-1})\big)$ with $\supp b_0\subset \mc{T}^\Sigma_0$ and $b_0(\pi(q_0))=1$ (with $\pi$ as in~\eqref{e:projection}). Then there exists a neighborhood $U$ of $q_0$
 so that for all $\check{\chi}\in S^0_{\phg}((0,T)\times \R^{n})$ with $\supp \check\chi\subset U$ there exists $\tilde\varphi\in C_c^\infty((0,T)\times \R^n)$  such that
$$\|\Op(\check \chi)u\|_{H^1}\leq C(\|\Op(b_0) D_{x_1}u|_{x_1=0}\|_{L^2}+\|\Op(b_0) u|_{x_1=0}\|_{H^1}+\|\tilde{\varphi}\Box u\|_{L^2}+\|\tilde{\varphi}u\|_{L^2})$$

\end{lemma}
\begin{proof}
Write  $q_0=(t_0,0,x_0',\tau_0,(\xi_1)_0,\xi_0')$. We consider the case $(\xi_0)_1>0$ (the case $(\xi_0)_1<0$ is treated similarly) and denote by $\lambda \in C^\infty((-\epsilon,\epsilon);S^0_{\phg}(\T^*(\R\times \R^{n-1}))$ a smooth symbol such that $\lambda (t,x, \tau , \xi') = \sqrt{\tau^2-r(x,\xi')}$ on a neighborhood of $(-\epsilon,\epsilon)\times \supp(b_0)$. Let $b\in C^\infty((-\epsilon,\epsilon);S^0_{\phg}(T^*(\R\times \R^{n-1}))$ solve 
\begin{equation}
\label{e:invariance} \partial_{x_1}b-H_{\lambda}b=0,\qquad b|_{x_1=0}=b_0.
\end{equation}
Denote by $\Lambda_\pm$ the two operators given by Lemma~\ref{l:factor} associated to $\lambda$, so that 
$$
\Op(b)(D_{x_1}-\Lambda_-)(D_{x_1}+\Lambda_+)=\Op(b)\big(\Box +R\big) ,
$$
with $R\in C^\infty((-\epsilon,\epsilon)_{x_1};\Psi^{-\infty}_{\phg}(\R\times \R^{n-1}))$. 
Letting $\Omega_{\epsilon}=\{x_1\in (-\epsilon/2,\epsilon/2)\}$, by Lemma~\ref{l:energy}, we have
\begin{align} 
\label{e:technestim0}
\|\Op(b)(D_{x_1}+\Lambda_+)u\|_{L^2(\Omega_\epsilon)}&\leq C(\|\Op(b_0)[(D_{x_1}+\Lambda_+)u]|_{x_1=0}\|_{L^2}+\|(D_{x_1}-\Lambda_+)\Op(b)(D_{x_1}+\Lambda_+)u\|_{L^2(\Omega_\epsilon)} ) \nonumber\\
&\leq C(\|\Op(b_0)[(D_{x_1}+\Lambda_+)u]|_{x_1=0}\|_{L^2}+\|\Op(b)(D_{x_1}-\Lambda_+)(D_{x_1}+\Lambda_+)u\|_{L^2(\Omega_\epsilon)} \nonumber\\
&\qquad +C\|[(D_{x_1}-\Lambda_+),\Op(b)](D_{x_1}+\Lambda_+)u\|_{L^2(\Omega_\epsilon)}).
\end{align}
Let us now estimate each term in the right hand-side. 
First, taking $\check\varphi$ such that $\check\varphi=1$ in a neighborhood of the support of the kernel of $\Op(b)$ intersected with $\Omega_\epsilon$, and $\tilde\varphi\in C_c^\infty((0,T)\times \R^n)$ with $\tilde\varphi = 1$ in a neighborhood of $\supp\check\varphi$, we have
\begin{align}
\label{e:technestim1}
\|\Op(b)(D_{x_1}-\Lambda_+)(D_{x_1}+\Lambda_+)u\|_{L^2(\Omega_\epsilon)} & =\| \Op(b)\big(\Box +R\big)u\|_{L^2(\Omega_\epsilon)} \nonumber \\
& \leq C \| \check\varphi \big(\Box +R\big)u\|_{L^2(\Omega_\epsilon)}
\leq  C (\| \tilde\varphi \Box u \|_{L^2} +\|\tilde\varphi u\|_{L^2}) ,
\end{align}
where we used the $L^2$ boundedness of $R$ and $\Op(b)$ together with the fact that the quantization $\Op$ gives operators whose kernels are compactly supported in $((0,T)\times \R^n)^2$.
Second, with $\tilde{b}_0\in S^0_{\phg}\big(T^*((0,T)\times \R^{n-1})\big)$ with $\supp \tilde{b}_0\subset \mc{T}^\Sigma_0$ and $\tilde{b}_0=1$ in a neighborhood of $\supp(b_0)$, we obtain 
\begin{align}
\label{e:technestim2}
\|\Op(b_0)[(D_{x_1}+\Lambda_+)u]|_{x_1=0}\|_{L^2}&  \leq \|\Op(b_0)D_{x_1}u|_{x_1=0}\|_{L^2} + \|\Op(b_0) \Lambda_+ u|_{x_1=0}\|_{L^2} \nonumber \\
&  \leq \|\Op(b_0)D_{x_1}u|_{x_1=0}\|_{L^2} + \|\Lambda_+ \Op(b_0)  u|_{x_1=0}\|_{L^2} +\| \Op(\tilde{b}_0)  u|_{x_1=0}\|_{L^2} \nonumber\\
&  \leq \|\Op(b_0)D_{x_1}u|_{x_1=0}\|_{L^2} +\| \Op(\tilde{b}_0)  u|_{x_1=0}\|_{H^1} .
\end{align}
Third, according to~\eqref{e:invariance}, the tangential operator $[(D_{x_1}-\Lambda_+),\Op(b)] \in C^\infty((-\epsilon,\epsilon)_{x_1};\Psi^{0}_{\phg}((0,T)\times \R^{n-1}))$ has principal symbol $\frac{1}{i}\{\xi_1 - \lambda, b\} =\frac{1}{i}(\partial_{x_1}b-H_{\lambda}b )= 0$ and is hence in $C^\infty((-\epsilon,\epsilon)_{x_1};\Psi^{-1}_{\phg}((0,T)\times \R^{n-1}))$. This yields 
\begin{align}
\label{e:technestim3}
\|[(D_{x_1}-\Lambda_+),\Op(b)](D_{x_1}+\Lambda_+)u\|_{L^2(\Omega_\epsilon)} \leq \| [(D_{x_1}-\Lambda_+),\Op(b)] D_{x_1} u\|_{L^2(\Omega_\epsilon)} +C \|\tilde{\varphi}u\|_{L^2} .
\end{align}
To eliminate the first term in the right handside, we let $\varphi\in S^0_{\phg}(\R\times \R^n)$ with $\varphi = 0$ in a conic neighborhood of $|(\tau ,\xi')| \leq \eps |\xi_1|$ and $\varphi = 1$ in a conic neighborhood of $|(\tau ,\xi')| \geq 2\eps |\xi_1|$, with $\eps>0$ small enough so that  $\varphi = 1$ on $\Char(\Box)$. 
We write
$$
[(D_{x_1}-\Lambda_+),\Op(b)] D_{x_1} u= [(D_{x_1}-\Lambda_+),\Op(b)] D_{x_1}\Op(\varphi) u + [(D_{x_1}-\Lambda_+),\Op(b)] D_{x_1}(1-\Op(\varphi)) u
$$
and remark first that $[(D_{x_1}-\Lambda_+),\Op(b)] D_{x_1}\Op(\varphi) \in \Psi^0_{\phg}(\R\times \R^n)$ due to pseudodifferential calculus (see~\cite[Theorem 18.1.35]{Hoermander:V3}), and hence
\begin{align}
\label{e:technestim4}
\| [(D_{x_1}-\Lambda_+),\Op(b)] D_{x_1}\Op(\varphi) u \|_{L^2(\Omega_\epsilon)} \leq C\|\tilde\varphi u\|_{L^2} .
\end{align}
Now, $1-\varphi$ vanishes in a conic neighborhood of $|(\tau ,\xi')| \geq 2\eps |\xi_1|$, and hence we have $[(D_{x_1}-\Lambda_+),\Op(b)] D_{x_1}(1-\Op(\varphi)) \in \Psi^1_{\phg}(\R\times \R^n)$ with principal symbol vanishing in a neighborhood of $\Char(\Box)$. The ellipticity of $\Box$ there yields 
$$
[(D_{x_1}-\Lambda_+),\Op(b)] D_{x_1}(1-\Op(\varphi))  = E \Box+R_1, \text{ with } E\in \Psi^{-1}_{\phg}(\R\times \R^n), R_1\in \Psi^{-\infty}_{\phg}(\R\times \R^n) .
$$
In particular, 
$$
\|[(D_{x_1}-\Lambda_+),\Op(b)] D_{x_1}(1-\Op(\varphi))  u\|_{L^2(\Omega_\epsilon)} \leq C (\| \tilde\varphi \Box u \|_{L^2} +\|\tilde\varphi u\|_{L^2}) ,
$$
which, combined with~\eqref{e:technestim1}-\eqref{e:technestim2}-\eqref{e:technestim3}-\eqref{e:technestim4} in~\eqref{e:technestim0} implies
\begin{align*}
\|\Op(b)(D_{x_1}+\Lambda_+)u\|_{L^2(\Omega_\epsilon)} \leq C \big(\|\Op(b_0)D_{x_1}u|_{x_1=0}\|_{L^2} +\| \Op(\tilde{b}_0)  u|_{x_1=0}\|_{H^1} + \| \tilde\varphi \Box u \|_{L^2} +\|\tilde\varphi u\|_{L^2} \big) 
\end{align*}
For $\psi\in S^0_{\phg}(\R\times \R^n)$ vanishing near $|(\tau ,\xi')| \leq \eps |\xi_1|$ and such that $\psi=1$ at $q_0$, $\Op(\psi)\Op(b)(D_{x_1}+\Lambda)\in \Psi^1_{\phg}(\R\times \R^n)$. Moreover, since $(\xi_0)_1>0$ and $b(t_0,0,x_0',\tau_0,\xi_0')=1$, the operator $\Op(\psi)\Op(b)(D_{x_1}+\Lambda)$ is elliptic at $q_0$. Therefore, for $\check \chi$ supported near enough to $q_0$, adjusting $\tilde{\varphi}$ if necessary, we finally obtain
\begin{align*}
\|\Op(\check \chi)u\|_{H^1}& \leq C \big(\| \Op(\psi)\Op(b)(D_{x_1}+\Lambda)u\|_{H^1} + \|\tilde{\varphi}u\|_{L^2} \big) \\
& \leq C \big(\|\Op(\tilde{b}_0) D_{x_1}u|_{x_1=0}\|_{L^2}+\|\Op( \tilde{b}_0) u|_{x_1=0}\|_{H^1}+\|\tilde{\varphi}\Box u\|_{L^2}+\|\tilde{\varphi}u\|_{L^2}\big),
\end{align*}
which concludes the proof of the lemma  (up to changing $b_0$ into $\tilde{b}_0$ in the statement).
\end{proof}

We now turn to the proof of Proposition~\ref{tempThm}. We follow the general structure of proof introduced by Lebeau in~\cite{Leb:96}, using the microlocal defect measures of G\'erard~\cite{Gerard:91} and Tartar~\cite{Tartar:90}.
{Note that  from the quantitative estimate of Lemma~\ref{l:mushrooms}, and in case $\d M=\emptyset$, ``constructive proofs'' (i.e. using no contradiction argument, and hence no defect measures) of Proposition \ref{tempThm} are possible, see~\cite{LL:14} or~\cite{LL:16}.}

\begin{proof}[Proof of Proposition \ref{tempThm}]
We prove estimate~\eqref{e:obsTemp} by contradiction. Assuming that estimate~\eqref{e:obsTemp} is false, there exist a sequence of data $F_k \in L^2((0,T)\times M)$ and $(u_{0,k},u_{1,k}) \in H^1_0(M) \times L^2(M)$ with 
\bnan
\label{e:normalized-seq}
\|(u_{0,k},u_{1,k})\|_{H^1\times L^2}=1
\enan such that the associated solution $(u_k)$ to \eqref{e:waveObs} satisfies 
\bnan
\label{e:seq-to-zero}
\|A_\delta(u_k|_{\Sigma_0})\|^2_{\bar{H}^1(\R\times \Sigma)} +\|A_{\delta}(\partial_\nu u_k|_{\Sigma_0})\|^2_{L^2(\R \times \Sigma)}+\|F_k\|_{L^2((0,T)\times M)}^2+\|(u_{0,k},u_{1,k})\|^2_{L^2\times H^{-1}} \to 0 .
\enan
Classical energy estimates for then yield $\|u_k\|_{H^1([0,T]\times M)}\leq C$ together with $\|u_k\|_{L^2([0,T]\times M)}\to 0$. Hence $u_k\rightharpoonup 0$ in $H^1$ and, possibly after taking a subsequence, we may assume (see~\cite{Gerard:91,Tartar:90} in the case without boundary and~\cite{Leb:96} or~\cite{BL:01} in the general case) there exists a nonnegative measure $\mu$ on $S\hat{Z}$ (see~\eqref{e:defSZ} for a defintion) so that,
\bnan
\label{e:converge-mu}
( Au_k,u_k)_{H^{1}(\R\times M)} \to \int (j^{-1})^*\sigma(A)d\mu ,
\enan
for all $A\in \Psi_{\phg}^0((0,T)\times \Int(M))$. Moreover, letting $(x_1,x')$ be Fermi normal coordinates near $\partial M$, then the convergence~\eqref{e:converge-mu} also holds for $A\in C^\infty([0,\e);\Psi^2(\R\times \partial M_{x'}))$, for $\e>0$.
Note that in both cases $(j^{-1})^*\sigma(A)$ lies in $C^0(S\hat{Z})$ since $\sigma(A)$ is independent of $\xi_1$ for $x_1$ small enough.

Let us first show that $\mu\equiv 0$. Notice that Lemma~\ref{l:GC} implies there exists $\e>0$ so that Assumption~\asref{GC2}{T} holds. We first prove that $\mu=0$ on a neighborhood of $\overline{\mc{T}_\e}\cap T^*_{(\e,T-\e)\times \Sigma_\e}(\R\times M)$. {Then, since $\mu$ is invariant under the generalized bicharacteristic flow $\varphi(s,\cdot)$ defined in~\eqref{e:def-flow} (which passes to the quotient space $S\hat{Z}$ according to homogeneity~\eqref{e:flow-hom}), see \cite{Leb:96,BL:01}, Assumption \asref{GC2}{T} implies $\mu\equiv 0$ (note that it is sufficient that $\supp(\mu)$ is invariant). }

Suppose $q_0\in\overline{\mc{T}_\e}\cap T^*_{[\e,T-\e]\times \Sigma_\e}(\R\times M)$. Then for $\delta<\e$, we have $\sigma(A_\delta)(\pi(q_0))=1$. Therefore, Lemma~\ref{l:mushrooms} applies with $\Op(b_0)=A_{\delta}$ and hence for $\check{\chi}$ supported close enough to $q_0$, 
$$\|\Op(\check \chi)u_k\|_{H^1}\leq C(\|A_\delta D_{x_1}u_k|_{x_1=0}\|_{L^2}+\|A_\delta u_k|_{x_1=0}\|_{H^1}+\|\tilde{\varphi}\Box u_k\|_{L^2(\Omega_\e)}+\|\tilde{\varphi}u_k\|_{L^2(\Omega_\e)})).$$
Now, the right hand side tends to 0 by assumption. Thus, pseudodifferential calculus together with~\eqref{e:converge-mu}, imply the existence of a conic neighborhood $U$ of $q_0$ so that $\mu(U/\R_+^*)=0$. Since this is true for any $q_0\in\overline{\mc{T}_\e}\cap T^*_{[\e,T-\e]\times \Sigma_\e}(\R\times M)$, there is a conic neighborhood $U_1$ of $\overline{\mc{T}_\e}\cap T^*_{(\e,T-\e)\times \Sigma_\e}(\R\times M)$ so that $\mu(U_1/\R_+^*)=0$. 
Invariance of $\mu$ and Assumption \asref{GC2}{T} imply that $\mu$ vanishes identically, which precisely means 
\bnan
\label{e:cv-0-strong}
u_k\to 0 \text{ in } H^1((0,T)\times M). 
\enan
Now, we denote
 $$ 
 E_k(t):=\|\nabla u_k(t,\cdot)\|_{L^2(M)}^2+\|\partial_t u_k(t,\cdot)\|_{L^2(M)}^2 ,
 $$  
 and observe from~\eqref{e:normalized-seq}-\eqref{e:seq-to-zero} that $E_k(0)\to 1$.
Moreover, for all $s_1,s_2 \in [0,T]$, we have 
$$
|E_k(s_2)-E_k(s_1)|\leq \left|\frac{1}{2}\int_{s_1}^{s_2}\partial_t E_k(t)dt\right|\leq \|F_k\|_{L^2}\|u_k\|_{H^1}\to 0.
$$
In particular, since this convergence is uniform in $s_1,s_2$,
$$\int_0^TE_k(t)dt=\int_0^TE_k(t)-E_k(0)dt+TE_k(0)\to  T.$$
Together with~\eqref{e:cv-0-strong}, this yields
$$0< T \leftarrow \left|\int_0^TE_k(t)dt\right|\leq \|u_k\|_{H^1}^2\to 0 ,$$
and hence the sought contradiction.
\end{proof}

%%%%%%%%%%%%%%%%%%%%%%%
\subsection{Observability: the Low Frequencies. From Proposition~\ref{tempThm} to Theorem~\ref{thm:observe}}
\label{s:obs}
There are different ways of writing the compactness-uniqueness argument of~\cite{BLR:92} (both reducing the problem to a unique continuation property for Laplace eigenfunctions). The first one is the precise argument of~\cite{BLR:92}: it uses again the geometric condition together with the propagation of wavefront sets (see also~\cite{LRLTT:16}). A second form seems to be due to~\cite{BG:02}: it is a bit longer but uses only that the observation region is time invariant. We write this version of the proof in the present context.

We first need a weak unique continuation property from a hypersurface. This is a weak version of Theorem~\ref{t:expo-bound-eigenfunctions}, but we chose to give a proof since it is much less involved. Note that no compactness is assumed and no boundary conditions are prescribed here.
\begin{lemma}[Unique continuation]
\label{l:uniq-cont}
Let $\Sigma$ be a nonempty interior hypersurface of a connected manifold $M$ and assume
$$
(-\Delta_g -\lambda^2) u =0 \quad \text{in } M , \quad u|_{\Sigma} = \d_{\nu}u |_{\Sigma} = 0,
$$
then $u$ vanishes identically.
\end{lemma}
\begin{proof}
Let $\Omega$ be a nonempty connected open set of $M$ such that $\Omega \cap \Sigma \neq \emptyset$ and $\Omega = \Omega^+ \cup (\Omega \cap \Sigma) \cup \Omega^-$ where the union is disjoint. Then, setting
$$
v (x)= u(x) \quad \text{for }x \in  \Omega^+, \quad v (x)= 0 \quad \text{for }x \in  \Omega^- ,
$$
we have $v \in L^2(\Omega)$, with, moreover ($\d_\nu$ pointing towards $\Omega^+$) 
$$
(-\Delta_g -\lambda^2) v =0 - [v]_\Sigma \delta_\Sigma' +(c(x)[v]_\Sigma- [\d_\nu v]_\Sigma) \delta_\Sigma = - u|_\Sigma \delta_\Sigma' +(c(x)u|_{\Sigma} -  \d_\nu u|_\Sigma) \delta_\Sigma  = 0 .
$$
This follows from the jump formula written in Fermi coordinate charts $(x_1, x')$ with $\Omega^+=\{x_1>0\}\cap \Omega$ and $-\Delta=- \d_{x_1}^2+R(x_1,x',D')+c(x)D_{x_1}$ with $R$ tangential (see Section~\ref{s:fermi}).

 A classical unique continuation result for elliptic operators (see e.g.~\cite[Theorem~4.2]{LeLe:09}) then implies that $v=0$ in all $\Omega$. From the definition of $v$, this yields $u|_{\Omega^+}=0$, and, using again the elliptic unique continuation result and the connectedness of $M$, implies that $u$ vanishes identically in $M$.
\end{proof}

\medskip
We next define for any $T>0$ and $\e>0$ the set of invisible solutions from $[\e,T-\e]\times  \Sigma_\e$ where $\Sigma_\e$ is as in~\eqref{e:sigs}:
\begin{multline*}
    \mathcal{N} (\e,T) = \big\{ (u_0, u_1) \in H^1_0(M) \times L^2(M) 
     \text{ such that the associated solution of~\eqref{e:waveObs} with $F=0$ }\\
       \text{satisfies }\d_\nu u|_{\Sigma} = u|_{\Sigma} = 0 \text{ in } \mc{D}' \big( (\e,T-\e)\times \Sigma_\e \big) \big\},
\end{multline*}
We have the following lemma, which is a consequence of Proposition~\ref{tempThm}.
\begin{lemma}
  \label{lemma: N(T) = 0}
  Suppose \asref{GC}{T} holds. Then there exists $\e_0>0$ such that for all $0<\e<\e_0$, we have $\mathcal{N}(\e,T) = \{0\}$.
\end{lemma}

We denote by $\calA$ the generator of the wave group, namely 
\begin{equation}
  \label{eq: def-A}
  \calA = \left(
  \begin{array}{cc}
      0 & - \id \\
    - \Delta_g  & 0
  \end{array}
  \right), \quad
  D(\calA) =  (H^2 \cap H^1_0(M)) \times H^1_0(M) ,
\end{equation}
so that the wave equation~\eqref{e:waveObs} with $F=0$ may be rewritten as 
\begin{equation}
\label{e:wave-abstract}
  \d_t U + \calA U  = 0, \quad U|_{t=0} = U_0 = (u_0,u_1) .
\end{equation}

\bnp
\textbf{Step 1: $\calN(\e,T)$ is finite dimensional.}
First, Proposition~\ref{prop: regularity-waves} implies that $\calN(\e,T)$ 
is a closed linear subspace of $H^1_0(M) \times L^2(M)$ for all $\e>0$. Since Assumption~\asref{GC}{T} holds, we may apply Proposition~\ref{tempThm}. The kernel of the operator $A_\delta$ in~\eqref{e:obsTemp} is compactly supported in $(0,T)\times \Int(\Sigma)$, and hence in $(\e_0,T-\e_0)\times \Sigma_{\e_0}$ for some $\e_0>0$.
Thus, for all $0<\e<\e_0$, the relaxed observability inequality~\eqref{e:obsTemp} applied to elements of $\calN(\e,T)$ gives
\begin{equation}
\label{eq:compact-balls}
c \|(u_0,u_1)\|_{H_0^1\times L^2}^2   \leq  \|(u_0,u_1)\|^2_{L^2\times H^{-1}}, \quad \text{for all }(u_0,u_1) \in \calN(\e,T) ,
\end{equation}
since the kernel of the operator $A_\delta$ is compactly supported in $(\e_0,T-\e_0)\times \Sigma_{\e_0}$, and $u|_{\Sigma}$, $\d_\nu u|_{\Sigma}$ vanish on this set.

Using the compact imbedding $H^1_0\times L^2 \subset L^2\times H^{-1}$, this implies that the unit ball of $\calN(\e,T)$ for the $H^1_0\times L^2$-norm is compact, that is, $\calN(\e,T)$  has finite dimension. Note also that it is thus complete for any norm. 

\medskip
\textbf{Step 2: $\calN(\e,T)\subset C^\infty(\overline{M})$ and $\calA \calN(\e,T) \subset \calN(\e,T)$.}

Taking $\eta >0$, sufficiently small (namely $\eta < \e_0-\e$), we remark that~\eqref{eq:compact-balls}
is also satisfied by all $U_0 = (u_0,u_1) \in \calN(\e+\eta,T)$. Taking $U_0 \in
\calN(\e,T)$ implies that, for all $\epsilon \in (0,\eta)$, we have $e^{-\epsilon
  \calA}U_0\in \calN(\e+\eta,T)$.  We also have, for $\lambda_0$ sufficiently large,
$$
(\lambda_0 + \calA)^{-1} \frac{1}{\epsilon}(\id - e^{-\epsilon \calA})U_0
= \frac{1}{\epsilon}(\id - e^{-\epsilon \calA})(\lambda_0 + \calA)^{-1} U_0
\mathop{\to}_{\epsilon \to 0^+} \calA(\lambda_0 + \calA)^{-1} U_0 \quad \text{in } H^1_0\times L^2 ,
$$ as $(\lambda_0 + \calA)^{-1} U_0 \in D(\calA)$.  As a
consequence, the sequence $\big(\frac{1}{\epsilon}(\id - e^{-\epsilon
  \calA})U_0 \big)_{\epsilon >0}$ is a Cauchy sequence in
$\calN(T'-\eta)$, endowed with the norm $\| (\lambda_0 +
\calA)^{-1} \cdot \|_{H^1_0\times L^2}$.  As all norms are equivalent in $\calN(\e+\eta,T)$,
 the sequence
$\big(\frac{1}{\epsilon}(\id - e^{-\epsilon \calA})U_0 \big)_{\epsilon >0}$ is
thus also a Cauchy sequence in this space, endowed with
the norm $\| \cdot \|_{H^1_0 \times L^2}$, which yields $\calA U_0 \in
H^1_0 \times L^2$. Hence, we have $\calN(\e,T)\subset D(\calA)$.
This argument may be inductively repeated to prove that $\calN(\e,T)\subset D(\calA^k)$ for all $k \in \N$, and yields in particular, that functions in $\calN(\e,T)$ are $C^\infty(\overline{M})$.

Take now $U_0 \in \calN(\e,T)$, and denote by $U(t)$ the associated solution of~\eqref{e:waveObs}, or equivalently~\eqref{e:wave-abstract}. Then, $u \in C^\infty( \R \times \overline{M})$, and using the fact that $\d_t$ is tangential to the manifold $\R \times \Sigma$ (thus commuting with $\d_\nu$), we obtain that $\d_tu|_{\Sigma}(t,x) = 0$ and $\d_\nu(\d_t u)|_{\Sigma}(t,x) = 0$ for all $(t,x) \in [\e,T-\e]\times \Sigma_\e$ (since this $U_0 \in \calN(\e,T)$ implies that this is satisfied by $u$). This is $\d_t U |_{t=0} \in \calN(\e,T)$. Remarking then that we have $\calA U_0 = - \d_t U |_{t=0} \in \calN(\e,T)$, this implies $\calA \calN(\e,T) \subset \calN(\e,T)$. 

\medskip

\textbf{Step 3: reduction to unique continuation for Laplace eigenfunctions: end of the proof.}
Since $\calN(\e,T)$ is a finite dimensional subspace of $D(\calA)$,
stable by the action of the operator $\calA$, it contains an
eigenfunction of $\calA$. There exist $\mu \in \C$ and $U=(u_0, u_1) \in
\calN(\e,T)$ such that $\calA U = \mu U$, that is, given the definition of $\calA$ in~\eqref{eq: def-A}, $-\Delta_D u_0 = -\mu^2 u_0$ and $u_1=-\mu u_0$. Hence $u_0$ is an eigenfunction of the Laplace-Dirichlet operator on $M$, associated to $-\mu^2 \in \R^+$, i.e. $\mu =i \lambda, \lambda \in \R$. The associated solution to~\eqref{e:waveObs} is $u(t,x) = e^{i\lambda t}u_0$, and $U_0 \in \calN(\e,T)$ implies $\d_\nu u_0|_{\Sigma} = u_0|_{\Sigma}= 0$. This, together with the fact that $u_0$ is a Laplace eigenfunction and Lemma~\ref{l:uniq-cont} proves that $u_0=0$ and then $U=0$. This proves that $\calN(\e,T)=\{0\}$.
\enp

From Lemma~\ref{lemma: N(T) = 0}, we can now conclude the proof of Theorem~\ref{thm:observe}.
\bnp[Proof of Theorem~\ref{thm:observe}]
We proceed by contradiction and suppose that the observability inequality~\eqref{e:obs-waves} does not hold for any $\delta>0$. Thus, for any $\delta>0$, there exists a
sequence $(u_0^k , u_1^k, F^k)_{k \in \N}$ of $H^1_0(M)\times L^2(M)\times L^2((0,T)\times M)$ such that, with $u^k$ the associated solution to~\eqref{e:waveObs}, we have
\begin{align}
  & \|(u_0^k , u_1^k)\|_{H^1_0\times L^2} =1 , \label{energy = 1}\\
  &  \|\varphi_\delta\partial_\nu u^k|_{\Sigma}\|_{\bar{H}^{-N}(\R\times \Sigma)}^2+\|\varphi_\delta u^k|_{\Sigma}\|_{\bar{H}^{-N}(\R\times \Sigma)}^2+\|F^k\|_{L^2((0,T)\times M)} \to 0  , \label{obs go to zero} \\
  & \|A_\delta (\partial_\nu u^k|_{\Sigma})\|_{L^2(\R\times \Sigma)}+\|A_\delta(u^k|_{\Sigma})\|_{\bar{H}^1(\R\times \Sigma)}^2 \to 0. \label{highFreq}
\end{align}
From~\eqref{energy = 1}, we may extract a subsequence of $(u_0^k , u_1^k)$ converging weakly in $H^1_0 \times L^2$ to some $(u_0, u_1)$. Denote by $u$ the associated solution to~\eqref{e:waveObs}, with $F=0$. Since $F^k\to 0$ in $L^2$ we may further extract from $u^k$ a subsequence converging to $u$ weakly in $H^1((0,T)\times M)$. According to Proposition~\ref{prop: regularity-waves}, we have $\partial_\nu u^k|_{\Sigma} \rightharpoonup \partial_\nu u|_{\Sigma}$ and $u^k|_{\Sigma} \rightharpoonup u|_{\Sigma}$ weakly in $H^{-1}((0,T')\times \Sigma)$. But according to~\eqref{obs go to zero}, this yields
$$
\varphi_\delta \partial_\nu u|_{\Sigma} =\varphi_\delta u|_{\Sigma}  = 0 , 
$$
and in particular, taking $\delta<\e$, 
$$
 \partial_\nu u|_{\Sigma} =u|_{\Sigma}  = 0 , \quad \text{ on } [\e,T-\e]\times \Sigma_\e.
$$
Thus,
$$(u_0,u_1)=(u(0),\partial_t u(0)) \in \calN(\e,T).$$
So, from Lemma~\ref{lemma: N(T) = 0}, we obtain $(u_0, u_1) =0$. The imbedding $H^1_0 \times L^2 \hookrightarrow L^2 \times H^{-1}$ being compact, this implies 
\begin{equation}\label{e:lowFreqConclude}\|(u_0^k , u_1^k)\|_{L^2 \times H^{-1}} \to \|(u_0 , u_1)\|_{L^2 \times H^{-1}} = 0.\end{equation}

Finally, Proposition~\ref{tempThm} implies that~\eqref{e:obsTemp} holds for any $\delta<\delta_0$. Therefore, taking $\delta<\min(\e,\delta_0)$ and using~\eqref{energy = 1}, \eqref{obs go to zero}, \eqref{highFreq}, \eqref{e:lowFreqConclude} in the relaxed observability inequality~\eqref{e:obsTemp}, we obtain at the limit $0<c \leq 0$, which is a contradiction.
\enp

\subsection{Controllability of the Wave Equation}
\label{s:control}

Theorem \ref{thm:control} is a straightforward corollary of the following theorem. Recall that $\ES = \mc{E}^\Sigma \cup \mc{G}^\Sigma$.
\begin{theorem}
\label{thm:control2}
Assume $(\Sigma,T)$ satisfies Assumption~\asref{GC}{$T$}.
Then there exists a continuous map
$$L^2(M)\times H^{-1}(M)\ni (v_0,v_1)\mapsto (f_0,f_1)\in \bigcap_{N\in \N} H^{-1,N}_{\comp,\ES}(\Sigma_T)\times H^{0,N}_{\comp,\ES}(\Sigma_T)$$
(the latter space being a Fr\'echet space when endowed with the seminorms of all $H^{-1,N}_{\comp,\ES}(\Sigma_T)\times H^{0,N}_{\comp,\ES}(\Sigma_T)$) so that the associated solution to~\eqref{e:waveControl} has $v\equiv 0$ for $t\geq T$. 
 \end{theorem}

\begin{proof}
Fix $0<T<T_1$. Then define
$$L^2([T,T_1]\times M)=\{F\in L^2([0,T_1]\times M),\,\supp F\subset [T,T_1]\}$$
and for $N\geq \frac{1}{2}$ the map
$$K: L^2([T,T_1]\times M)\to H^{1,-N}_{\loc,\ES}(\Sigma_{T})\times H^{0,-N}_{\loc,\ES}(\Sigma_{T})$$
 given by 
$$F\mapsto (u|_{(0,T)\times \Sigma},-\partial_\nu u|_{(0,T)\times \Sigma})$$
where $u$ solves
$$\begin{cases} \Box u=F&\text{on }(0,T_1)\times \Int(M) , \\
u=0&\text{on }(0,T_1)\times \partial M , \\
(u|_{t=T_1},\partial_tu|_{t=T_1})=(0,0)&\text{in } \Int(M).
\end{cases}$$
This map is well defined by~\eqref{e:regularityState}.
Define also the operator $S:L^2([T,T_1]\times M)\to H_0^1(M)\times L^2(M)$ by 
\begin{equation}
\label{e:data}S(F):=(u|_{t=0},\partial_t u|_{t=0}).
\end{equation}

Now, suppose that Assumption~\asref{GC}{$T$} holds and let $A_\delta$ as in Theorem~\ref{thm:observe}. For $\e>0$ small $B^{\ES}_\e$ is elliptic on $\WF(A_\delta)$ and hence using the elliptic parametrix construction we write 
$$A_\delta =GB^{\ES}_\e+R$$
with $R\in \Psi^{-\infty}_{\phg}((0,T)\times \Int(\Sigma))$, and $G\in \Psi^0_{\phg} ((0,T)\times \Int(\Sigma))$. Therefore   
 Theorem~\ref{thm:observe} implies that there exists $\e>0$ small enough depending only on $(\Sigma,T)$ and for all $N\in \N$, there exists $C_N>0$ so that 
\begin{equation}
\label{e:observe2}\|S(F)\|_{H_0^1(M)\times L^2(M)}\leq C_N\|K (F)\|_{H^{1,-N}_{\loc,\ES,\e}(\Sigma_{T})\times H^{0,-N}_{\loc,\ES,\e}(\Sigma_{T})}.\end{equation}

Let $(v_0,v_1)\in H^{-1}(M)\times L^2(M)$ and define the linear functional $\ell_N:
\operatorname{ran}(K)\to \mathbb{C}$
by 
$$\ell_N(K(F))=\langle S(F),(-v_1,v_0)\rangle_{H_0^1(M)\times L^2(M),H^{-1}(M)\times L^2(M)}$$
where $S$ is defined in~\eqref{e:data}.
Then, $\ell_N$ is well defined and continuous by \eqref{e:observe2}. In particular, 
$$|\ell_N(K(F))|\leq C_N\|(v_0,v_1)\|_{H^{-1}(M)\times L^2(M)}\|K(F)\|_{H^{1,-N}_{\loc,\ES,\e}(\Sigma_{T})\times H^{0,-N}_{\loc,\ES,\e}(\Sigma_{T})}.$$
Since $\ell_N$ is a continuous linear functional defined on a subspace of $H^{1,-N}_{\loc,\ES}(\Sigma_{T})\times H^{0,-N}_{\loc,\ES}(\Sigma_{T})$ by the Hahn-Banach theorem $\ell_N$ extends to a continuous linear functional on the whole space (still denoted $\ell_N$) with 
$$|\ell_N(w_1,w_2)|\leq C_N\|(v_0,v_1)\|_{H^{-1}(M)\times L^2(M)}\|(w_1,w_2)\|_{H^{1,-N}_{\loc,\ES,\e}(\Sigma_{T})\times H^{0,-N}_{\loc,\ES,\e}(\Sigma_{T})}.$$

Thus, by Lemma~\ref{l:dual}, there exists $(f_{0,N},f_{1,N})\in H^{-1,N}_{\comp,\ES}(\Sigma_T)\times H_{\comp,\ES}^{0,N}(\Sigma_{T})$ so that for all $(w_1,w_2)\in H^{1,-N}_{\loc,\ES}(\Sigma_{T})\times H^{0,-N}_{\loc,\ES}(\Sigma_{T})$, we have
$$
\ell_N(w_1,w_2)=\langle (w_1,w_2),(f_{0,N},f_{1,N})\rangle_{H^{1,-N}_{\loc,\ES}\times H^{0,-N}_{\loc,\ES},H^{-1,N}_{\comp,\ES}\times H_{\comp,\ES}^{0,N}},
$$
and hence for some $\e'>0$, 
$$\|(f_{0,N},f_{1,N})\|_{H^{1,N}_{\loc,\ES,\e'}(\Sigma_{T})\times H^{0,N}_{\loc,\ES,\e'}(\Sigma_{T})}\leq C_{N,\e,\e'}\|(v_0,v_1)\|_{H^{-1}(M)\times L^2(M)}.$$

Let $v$ be the unique solution to 
$$\begin{cases}\Box v=f_{0,N}\delta_\Sigma+f_{1,N}\delta'_\Sigma&\text{ on }(0,T_1)\times \Int(M), \\
v=0&\text{ on }(0,T_1)\times \partial M, \\
(v,\partial_t v)|_{t=0}=(v_0,v_1)&\text{ in } \Int(M) ,
\end{cases}
$$
given by Definition~\ref{d:transp-sol} and Theorem~\ref{t:well-posedness}. 
Then for any $F\in L^2([T,T_1]\times M)$ we have
\begin{align*}
\langle v,  F\rangle_{L^2((0,T_1)\times M)} &=\langle v_1 , u(0) \rangle_{H^{-1}(M), H^1(M)} - ( v_0 ,\d_t u(0))_{L^2(M)}  \\
&\qquad +  \langle f_{0,N} , u|_{\Sigma} \rangle_{H^{-1, N}_{\comp,\ES}(\Sigma_T), H^{1,-N}_{\loc,\ES}(\Sigma_T)} 
 - \langle f_{1,N} , \d_\nu u|_{\Sigma} \rangle_{H^{0,N}_{\comp,\ES}(\Sigma_T) , H^{0,-N}_{\loc,\ES}(\Sigma_T)} \\
 &=\langle (v_1,-v_0),S(F)\rangle_{H^{-1}(M)\times L^2(M), H_0^1(M)\times L^2(M)}   \\
 &\qquad+ \langle (f_{0,N},f_{1,N}),K(F)\rangle _{H^{-1,N}_{\comp,\ES}\times H^{1,N}_{\comp,\ES},H^{1,-N}_{\loc,\ES}\times H_{\loc,\ES}^{0,-N}}\\
  &=\langle (v_1,-v_0),S(F)\rangle_{H^{-1}(M)\times L^2(M), H_0^1(M)\times L^2(M)}   + \overline{\ell_N(K(F))}\\
&=\langle (v_1,-v_0),S(F)\rangle_{H^{-1}(M)\times L^2(M), H_0^1(M)\times L^2(M)} \\
&\qquad+\langle (-v_1,v_0),S(F)\rangle_{ H^{-1}(M)\times L^2(M), H_0^1(M)\times L^2(M)}=0 .
\end{align*}
Since this is true for all $F\in L^2([T,T_1]\times M)$, we obtain $v\equiv 0$ on $[T,T_1]\times M$. 
 
Now, for $k>N$, the inclusion $H^{-1,k}_{\comp,\ES}(\Sigma_{T})\times H^{0,k}_{\comp,\ES}(\Sigma_{T})\subset H^{-1,N}_{\comp,\ES}(\Sigma_{T})\times H^{0,N}_{\comp,\ES}(\Sigma_{T})$ is dense and $H^{1,-N}_{\loc,\ES}(\Sigma_{T})\times H^{0,-N}_{\loc,\ES}(\Sigma_{T})\subset H^{1,-k}_{\loc,\ES}(\Sigma_{T})\times H^{0,-k}_{\loc,\ES}(\Sigma_{T})$ is dense. So, in particular, $\ell_N$ extends to a linear functional on $H^{1,-k}_{\loc,\ES}(\Sigma_{T})\times H^{0,-k}_{\loc,\ES}(\Sigma_{T})$ by density. This yields
\bna
\langle (w_1,w_2),(f_{0,k},f_{0,k})\rangle_{H^{1,-N}_{\loc,\ES}\times H^{0,-N}_{\loc,\ES},H^{-1,N}_{\comp,\ES}\times H_{\comp,\ES}^{0,N}}=
\langle(w_1,w_2),(f_{0,N},f_{1,N})\rangle_{H^{1,-N}_{\loc,\ES}\times H^{0,-N}_{\loc,\ES},H^{-1,N}_{\comp,\ES}\times H_{\comp,\ES}^{0,N}}
\ena
for all $(w_1,w_2)\in H^{1,-N}_{\loc,\ES}(\Sigma_{T})\times H^{0,-N}_{\loc,\ES}(\Sigma_{T})$. This implies that $f_{0,k}=f_{0,N}$ and $f_{1,k}=f_{1,N}$ and hence that 
$$f_{0,N}\equiv f_0\in \bigcap_N H^{-1,N}_{\comp,\ES}(\Sigma_{T}),\qquad f_{1,N}\equiv f_1\in \bigcap _NH^{0,N}_{\comp,\ES}(\Sigma_{T}),$$
which concludes the proof of the theorem.
\end{proof}

%%%%%%%%%%%%%%%%%%%%%%%%%%%%%%%%%%
\section{Controllability of the heat Equation}
\label{s:controlHeat}

%%%%%%%%%%%%%%%%%%%%%%%%%%%%%%%%%%
\subsection{Well-posedness for the heat equation controlled from a hypersurface}
\label{s:WP:heat}
The well-posedness theory is easier than that of the wave equation since the regularity theory for the heat equation directly implies that the traces of the solution on $\Sigma$ are ``admissible'' observations, in the usual sense, see~\cite[Chapter~2.3]{Cor:book} and \cite[Chapter~4.3]{TW:09}.

\begin{lemma}
\label{l:IPP-smooth2}
Given $T>0$, assume that the functions $v\in C^\infty([0,T]\times M\setminus \Sigma)\cap C^0((0,T); L^2(M))$ $u,F \in  C^\infty([0,T]\times M)$ and $f_0, f_1 \in C^\infty_c((0,T)\times \Int(\Sigma))$ solve
$$
(\d_t -\Delta)v = f_0 \delta_\Sigma +  f_1 \delta_\Sigma' \text{ in } \D'((0,T)\times \Int(M)), \quad \text{ and }\quad (- \d_t -\Delta) u = F .
$$
Then, we have the identity
$$
\left[ ( v, u)_{L^2(M)} \right]_0^T + (v,F)_{L^2((0,T)\times M)} 
= \int_{(0,T)\times\Sigma} \Big(f_0 u|_{\Sigma}  - f_1 \d_\nu u|_{\Sigma} \Big)d\sigma dt .
$$
\end{lemma}
Also, we have the following ``admisibility result'' (regularity of traces).
\begin{lemma}
\label{l:admissibility}
Given $T>0$, there is $C>0$ such that for all $F\in L^2((0,T)\times M)$, $\tilde{u} \in H^1_0(M)$ and $u$ associated solution of
\begin{equation}
\begin{cases}
\label{e:F-IC-heat}
(-\d_t - \Delta) u=F &\text{on }(0,T)\times \Int(M), \\
u= 0 &\text{on } (0,T)\times \d M , \\
u|_{t=T}= \tilde{u}&\text{in } \Int(M) ,
\end{cases}
\end{equation}
we have 
$$
 \|\partial_\nu u|_{\Sigma}\|_{L^2(0,T; \ovl{H}^{\frac12}(\Sigma))}^2+\|u|_{\Sigma}\|_{L^2(0,T; \ovl{H}^{\frac32}(\Sigma))}^2  
 \leq  C \|F\|_{L^2((0,T)\times M)}^2 + C  \|\tilde{u}\|_{H^1(M)}^2 .
$$
\end{lemma}
\begin{proof}
This is a direct consequence of the regularity theory for the heat equation~\eqref{e:F-IC-heat}, namely $u \in C^0([0,T]; H^1_0(M)) \cap L^2(0,T;H^2(M))\cap H^1(0,T; H^1_0(M))$, with 
$$
\|u\|_{L^\infty(0,T;H^1(M))}^2+ \|u\|_{L^2(0,T;H^2(M))}^2 +  \|u\|_{H^1(0,T;H^1(M))}^2\leq C \|F\|_{L^2((0,T)\times M)}^2 + C \|\tilde{u}\|_{H^1(M)}^2 ,
$$
see e.g.~\cite[Chapter~7.1.3, Theorem~5]{Evans:98}. The standard trace estimates then yield
$$
 \|\partial_\nu u|_{\Sigma}\|_{L^2(0,T; \ovl{H}^{\frac12}(\Sigma))}^2+\|u|_{\Sigma}\|_{L^2(0,T; \ovl{H}^{\frac32}(\Sigma))}^2   \leq C \|u\|_{L^2(0,T;H^2(M))}^2,
$$
which concludes the proof of the lemma.
\end{proof}
This suggests the following definition (see~\cite[Chapter~2.3]{Cor:book}) of solutions of the controlled heat equation~\eqref{e:heat-control}. 

\begin{definition}
\label{d:transp-sol-heat}
Given $T>0$, $v_0 \in H^{-1}(M), f_0 \in L^2(0,T; H_{\comp}^{-\frac32}(\Int(\Sigma))), f_1 \in L^2(0,T; H^{-\frac12}_{\comp}(\Int(\Sigma)))$, we say that $v$ is a solution of \eqref{e:waveControl} if $v \in C^0([0,T]; H^{-1}(M))$ and for any $t \in [0,T]$, for any $\tilde{u} \in H^1_0(M)$, we have
\bna
\langle v(t) , \tilde{u} \rangle_{H^{-1}, H^1_0}& = &\langle v_0 , u(0) \rangle_{H^{-1}, H^1_0} \\
&& +\int_0^t  \langle f_0 (s), u|_{\Sigma}(s) \rangle_{H_{\comp}^{-\frac32}(\Sigma), H_{\loc}^{\frac32}(\Sigma)} 
 - \langle f_1(s) , \d_\nu u|_{\Sigma}(s) \rangle_{H_{\comp}^{-\frac12}(\Sigma), H_{\loc}^{\frac12}(\Sigma)} ds .
\ena
where $u$ is the unique solution to
\begin{equation}
\begin{cases}
\label{e:F-IC=0}
(-\d_s - \Delta) u=0 &\text{on }(0,t)\times\Int(M)\\
u=0 &\text{on }(0,t)\times \d M\\
u|_{s=t}=\tilde{u}&\text{in } \Int(M) ,
\end{cases}
\end{equation}
i.e. $u(s) = e^{(t-s)\Delta}\tilde{u}$.
\end{definition}
The following result is a direct consequence of (a slight variation on) \cite[Theorem~2.37]{Cor:book} and the admissibility estimate of Lemma~\ref{l:admissibility}.
\begin{theorem}[Well-posedness of the controlled heat equation]
\label{t:well-posedness-heat}
Let $T>0$. There exist $C>0$ such that for all $v_0 \in H^{-1}(M),  f_0 \in L^2(0,T; H_{\comp}^{-\frac32}(\Int(\Sigma))), f_1 \in L^2(0,T; H^{-\frac12}_{\comp}(\Int(\Sigma)))$, there exists a unique solution $v$ of~\eqref{e:heat-control} in the sense of Definition~\ref{d:transp-sol-heat} and we have:
$$
\|v\|_{L^\infty(0,T; H^{-1}(M))}\leq C \left( \|v_0\|_{H^{-1}(M)}+ \| f_0 \|_{L^2(0,T; \ovl{H}^{-\frac32}(\Sigma))} + \|f_1\|_{L^2(0,T; \ovl{H}^{-\frac12}(\Sigma))} \right).
$$
\end{theorem}

%%%%%%%%%%%%%%%%%%

\subsection{Global interpolation inequality and universal lower bound for traces of eigenfunctions}
\label{s:unique}
We follow the general Lebeau-Robbiano method~\cite{LR:95} and use moreover a Carleman estimate of~\cite{LR:97}. We refer to~\cite{LeLe:09} for an exposition of these works.

The global strategy~\cite{LR:95} is the following: 
\begin{enumerate}
\item Local Carleman estimates 
\item $\implies$ local interpolation estimates
\item $\implies$ a global interpolation estimate
\item $\implies$ finite dimensional observability/controllability for an elliptic evolution equation
\item $\implies$ finite dimensional observability/controllability for the heat equation
\item $\implies$ observability/controllability for the heat equation.
\end{enumerate}

Also, the unique continuation estimate for eigenfunctions of Theorem~\ref{t:expo-bound-eigenfunctions} can be deduced from the global interpolation estimate.
The present section proves steps $1, 2, 3$. The next section is devoted to that of steps $4, 5, 6$.

In the following, for $\alpha>0$, we set $Y_\alpha= (-\alpha , \alpha) \times M$, $\Sigma_\alpha = (- \alpha, \alpha) \times \Sigma$, and denote $Q = -\d_{s}^2 - \Delta_g$.
\begin{theorem}[Global interpolation] 
\label{t:global-interp}
Let $S>\beta >0$. For all $\psi \in C^\infty_c(\Sigma_\beta)$ not identically vanishing, there exist $C, \delta>0$ such that
\begin{equation}
\label{e:loc-interp}
\|v\|_{H^1(Y_\beta)} \leq C \left(\|Qv\|_{L^2(Y_S)} + \|\psi v|_{\Sigma_\beta}\|_{L^2(\Sigma_\beta)} +  \|\psi \d_\nu v|_{\Sigma_\beta}\|_{L^2(\Sigma_\beta)}  \right)^\delta \|v\|_{H^1(Y_S)}^{1-\delta}
\end{equation}
for all $v \in H^2(Y_S)$ such that $v|_{(-S,S)\times \d M}=0$.
\end{theorem}

If we were considering a second order elliptic operator $Q$ on a manifold $Y_S$ with {\em smooth} boundary, and with Dirichlet condition on the {\em whole} $\d Y_S$, this estimate would simply read
\begin{equation*}
\|v\|_{H^1(Y_S)} \leq C \left(\|Qv\|_{L^2(Y_S)} + \|\psi v|_{\Sigma_0}\|_{L^2(\Sigma_0)} +  \|\psi \d_\nu v|_{\Sigma_0}\|_{L^2(\Sigma_0)}  \right) .
\end{equation*}
However, here $Y_S = (-S , S) \times \M$ is not smooth at $(-S , S) \times \d \M$ and it is crucial for the next arguments that no boundary condition is prescribed at the boundary $(\{-S \} \cup \{S\})\times M$.

\bigskip
The proof of Theorem~\ref{t:global-interp} follows from arguments of Lebeau and Robbiano~\cite{LR:95,LR:97}. The idea is that such interpolation inequalities follow locally from Carleman estimates, and then propagate well. Hence, our task is only 
\begin{itemize}
\item[(i)] to deduce from a local Carleman estimate near $\Sigma_\beta$ that the traces at the boundary ``control'' a small nonempty open set near $\Sigma_\beta$ (i.e. that \eqref{e:loc-interp} holds with, in the l.h.s. the local $H^1$ norm in this set)
\item[(ii)]  to use a global interpolation inequality implying that such a small set ``controls'' the $H^1(Y_\beta)$ norm, and then put the two inequalities together.
\end{itemize}

For the second point (ii), we can start from the following result of ~\cite[Section~3, Estimate (1)]{LR:95}.
\begin{theorem}
\label{t:global-interp-LR}
Let $U \subset Y_S$ be any nonempty open set, then there is $C>0$ and $\delta_0 \in (0,1)$ such that we have
\begin{equation}
\label{e:loc-interp-LR}
\|v\|_{H^1(Y_\beta)} \leq C \left(\|Qv\|_{L^2(Y_S)} +  \| v \|_{H^1(U)}  \right)^{\delta_0} \|v\|_{H^1(Y_S)}^{1-\delta_0}
\end{equation}
for all $v \in H^2(Y_S)$ such that $v|_{(-S,S)\times \d M}=0$.
\end{theorem}
As a consequence, it suffices to prove the first point (i), namely, that there exists such an $U$ such that, for some $C, \delta_1 >0$ we have
\begin{equation}
\label{e:loc-interp-wish}
\|v\|_{H^1(U)} \leq C \left(\|Qv\|_{L^2(Y_S)} + \|\psi v|_{\Sigma_\beta}\|_{L^2(\Sigma_\beta)} +  \|\psi \d_\nu v|_{\Sigma_\beta}\|_{L^2(\Sigma_\beta)}  \right)^{\delta_1} \|v\|_{H^1(Y_S)}^{1-\delta_1},
\end{equation}
which is now a local estimate. Indeed, \eqref{e:loc-interp-LR} together with~\eqref{e:loc-interp-wish} directly yield~\eqref{e:loc-interp} for $\delta = \delta_0 \delta_1$ (see e.g.~\cite[Lemme~4]{LR:95}).

\bigskip
To prove~\eqref{e:loc-interp-wish}, we shall take $m\in \Sigma$ a point for which $\psi(m) \neq 0$, and assume that the set $U$ is a small neighborhood of $m$ intersected with a single side of $\Sigma$. Also, we shall say that $\d_\nu$ is pointing towards $U$. 
We now work in the local Fermi normal coordinates near $m \in \Sigma$, described in Section~\ref{s:fermi}.  
 The operator $Q = -\d_{s}^2 -\Delta_g$, still denoted by $Q$ in these coordinates, is given, modulo conjugation by a harmless exponential factor, by
\begin{equation*}
  Q = - \d_{x_1}^2  - \d_s^2 + r(x_1, x', \frac{\d_{x'}}{i} ) ,  \quad \text{with principal symbol} \quad  q = \xi_1^2 + \xi_s^2+ r(x_1, x', \xi') ,
\end{equation*}
 where 
 \begin{itemize}
 \item $(s,x')$ are the variables in $(-S,S)\times \Sigma$, $\xi_s\in \R$ is the cotangent variable associated to $s$;
 \item variables are in a neighborhood of zero in the half space $\R^{n+1}_+ = \R_{s} \times \R_{+, x_1} \times \R_{x'}^{n-1}$ (we only estimate things on $\{x_1 >0\}$, where $U$ is);
 \item $\d_\nu$ is given by $\d_{x_1}$ in these coordinates.
 \end{itemize}

Now, the proof of~\eqref{e:loc-interp-wish} relies on the following Proposition~\cite[Proposition 1]{LR:97}. Here, the variable $s$ does not play a particular role: hence, in what follows, we only write (with a slight abuse of notation) $x\in \R^{n+1}$ for the overall variable, and accordingly $q = q(x, \xi) = q(s , x_1,  x', \xi_s ,\xi_1, \xi')$. 
We also use the notation 
$$
q_\varphi (x, \xi) = q(x, \xi + i d\varphi(x)) .
$$

\begin{proposition}
Let $R>0$ and $\varphi \in C^\infty$ in a neighborhood of $K:= \R^{n+1}_+ \cap \ovl{B}(0,R)$ and such that 
\begin{itemize}
\item $\d_{x_1} \varphi \neq 0$  on $K$,
\item (H\"ormander subellipticity condition) $ \forall (x ,  \xi) \in K \times \R^{n+1}, \quad  q_\varphi(x, \xi) = 0 \implies
    \{ \Re(q_\varphi), \Im (q_\varphi)\} (x,\xi) >0$.
\end{itemize}
Then, we have 
\begin{multline}
\label{e:Carleman}
 h\| e^{\varphi/h}u\|_{L^2(\R^{n+1}_+)}^2 +h^3 \| e^{\varphi/h} \nabla u\|_{L^2(\R^{n+1}_+)}^2  \\
\lesssim  h^4\| e^{\varphi/h} Q u\|_{L^2(\R^{n+1}_+)}^2 + h\| e^{\varphi/h}u|_{x_1=0}\|_{L^2(\R^{n})}^2 +h^3 \| e^{\varphi/h} \d_{x_1} u|_{x_1=0}\|_{L^2(\R^{n})}^2 
\end{multline}
for all $u \in C^\infty(\overline{\R^{n+1}_+})$ such that $\supp(u) \subset B(0,R)$ and $h \in (0,h_0)$. 
\end{proposition}
The end of proof of Theorem~\ref{t:global-interp-LR} is then similar to~\cite{LR:95} or~\cite[Appendix]{LZ:98}.
\begin{proof}[End of the proof of Theorem~\ref{t:global-interp-LR}]
We first fix $R>0$ small enough such that $B(0,R)$ is contained in the coordinate chart and that the set $\overline{B}(0,R) \cap \{x_1 = 0\}$ (where the observation shall take place) is contained in the set $\{\psi >0\}$ (where $\psi$ is the cutoff function appearing in~\eqref{e:loc-interp}). Second, we define the weight function $\varphi (x) = e^{-\mu |x-x^a|}- e^{-\mu |x^a|}$, where $\mu>0$ (large, to be chosen) and, for $a \in (0, R)$, we have $x^a = (0, \cdots , 0 , -a) \notin \overline{\R^{n+1}_+}$. Hence, $\varphi$ is smooth and satisfies $\d_{x_1} \varphi \neq 0$ on $K= \R^{n+1}_+ \cap \ovl{B}(0,R)$. 

 According to classical computations (see e.g.~\cite[Lemma~A.1]{LeLe:09}), $\varphi$ satisfies the H\"ormander subellipticity condition on $K$ for $\mu$ large enough (depending on $R$ and $a$, and fixed from now on).

 Note that levelsets of $\varphi$ are balls.
Moreover, we have $\varphi (0)=0$ and $\varphi(x) <0$ if $|x - x^a|>|x^a|$, and in particular on $\{ x_1>0 \}$. 

For $\eps>0$ sufficiently small (depending on $R, a$ and $\mu$), the set $\{\varphi \geq - 4\eps\} \cap \R^{n+1}_+$ is contained in $B(0,R)\cap \R^{n+1}_+$, where~\eqref{e:Carleman} holds. Also, the set $\{\varphi \geq - 4\eps\} \cap \{x_1 = 0\} \subset B(0,R) \cap \{x_1 = 0\}$ is contained in the set $\{\psi >0\}$.

Finally, setting $$U:= \{\varphi > - \frac{\eps}{2}\} \cap \{x_1 >0\} ,$$
we have $U \neq \emptyset$ since $\varphi(0)=0$ and $\varphi<0$ on $\{ x_1>0 \}$.

We let $\chi \in C^\infty(\R^{n+1})$ such that $\chi=1$ on $\{\varphi \geq - 2\eps\} $ and  $\chi=0$ on $\{\varphi \leq - 3\eps\}$,  and apply~\eqref{e:Carleman} to $u = \chi v$. We have $\varphi \leq 0$ on the support of $u$ so that 
\bna
\| e^{\varphi/h}u|_{x_1=0}\|_{L^2(\R^{n})}^2 & \leq &\| \chi|_{x_1=0} v|_{x_1=0}\|_{L^2(\R^{n})}^2 \leq C \| \psi v|_{x_1=0}\|_{L^2(\R^{n})}^2 \\
 \| e^{\varphi/h} \d_{x_1} u|_{x_1=0}\|_{L^2(\R^{n})}^2 & \leq & \|\d_{x_1} \chi|_{x_1=0} v|_{x_1=0}\|_{L^2(\R^{n})}^2  +\| \chi|_{x_1=0} \d_{x_1} v|_{x_1=0}\|_{L^2(\R^{n})}^2 \\
& \leq & C \| \psi v|_{x_1=0}\|_{L^2(\R^{n})}^2  + C \|  \psi \d_{x_1} v|_{x_1=0}\|_{L^2(\R^{n})}^2  .
\ena
Using that $\chi=1$ on $\{\varphi \geq - \frac{\eps}{2}\} \subset\{\varphi \geq - \eps\}$ and $U= \{\varphi \geq - \frac{\eps}{2}\} \cap \R^{n+1}_+$, we have that
\bna
 h\| e^{\varphi/h}u\|_{L^2(\R^{n+1}_+)}^2 +h^3 \| e^{\varphi/h} \nabla u\|_{L^2(\R^{n+1}_+)}^2
& \geq & h\| e^{\varphi/h}u\|_{L^2(U)}^2 +h^3 \| e^{\varphi/h} \nabla u\|_{L^2(U)}^2  \\
 & \geq&  h^3 e^{-\eps/h} \| v \|_{H^1(U)}^2\geq e^{- \frac32 \eps/h} \| v \|_{H^1(U)}^2  .
\ena
Finally, we have $Q\chi v = \chi Qv + [Q, \chi]v$, where $[Q, \chi]$ is a first order differential operator with coefficients supported in $\{-\eps \geq \varphi \geq - 2\eps\}\cap \R^{n+1}_+$. Thus, we have
\bna
h^4\| e^{\varphi/h} Q u\|_{L^2(\R^{n+1}_+)}^2 & \lesssim & \| e^{\varphi/h} \chi Qv \|_{L^2(\R^{n+1}_+)}^2 +  \| e^{\varphi/h}  [Q, \chi] v \|_{L^2(\R^{n+1}_+)}^2 \\
 & \lesssim & \|Qv \|_{L^2(K)}^2 +  e^{-2\eps/h} \| v \|_{H^1(K)}^2  .
\ena
Combining the last three estimates with~\eqref{e:Carleman}, we find that there is $C,h_0>0$ such that for any $v \in C^\infty(\R^{n+1})$, for all $h \in (0,h_0)$, we have 
\bna
e^{-\frac32 \eps/h} \| v \|_{H^1(U)}^2 \leq
 C \| \psi v|_{x_1=0}\|_{L^2(\R^{n})}^2  + C \|  \psi \d_{x_1} v|_{x_1=0}\|_{L^2(\R^{n})}^2  +C  \|Qv \|_{L^2(K)}^2 +  Ce^{-2\eps/h} \| v \|_{H^1(K)}^2
\ena
and hence, for all $h \in (0,h_0)$, 
\bna
 \| v \|_{H^1(U)}^2 \leq
 C e^{\frac32 \eps/h}\left( \| \psi v|_{x_1=0}\|_{L^2(\R^{n})}^2  +  \|  \psi \d_{x_1} v|_{x_1=0}\|_{L^2(\R^{n})}^2  + \|Qv \|_{L^2(K)}^2 \right)+  Ce^{-\frac12 \eps/h} \| v \|_{H^1(K)}^2
\ena
After an optimization in the parameter $h$ (see~\cite{Robbiano:95}), this yields the existence of $C>0$ and $\delta_1\in (0,1)$ such that 
\bna
 \| v \|_{H^1(U)}^2 \leq
 C \left( \| \psi v|_{x_1=0}\|_{L^2(\R^{n})}^2  +  \|  \psi \d_{x_1} v|_{x_1=0}\|_{L^2(\R^{n})}^2  + \|Qv \|_{L^2(K)}^2 \right)^{\delta_1} \| v \|_{H^1(K)}^{2(1-\delta_1)} ,
\ena
which, coming back to the original variables, implies~\eqref{e:loc-interp-wish}, and then according to Theorem~\ref{t:global-interp-LR} and \cite[Lemme~4]{LR:95}), concludes the proof of Theorem~\ref{t:global-interp} (see the above discussion).
\end{proof}

From Theorem~\ref{t:global-interp}, we deduce a proof of Theorem~\ref{t:expo-bound-eigenfunctions}. 
\begin{proof}[Proof of Theorem~\ref{t:expo-bound-eigenfunctions}]
For a non identically vanishing function $\psi$ such that $\supp(\psi) \subset \Sigma_\beta$, we apply Theorem~\ref{t:global-interp} to $v(s, x)= e^{\lambda s} u(x) \in C^\infty((-S,S);H^2(M)\cap H^1_0(M) )$, which satisfies $$Q v=e^{\lambda s} (-\Delta_g-\lambda^2)u, \quad \text{in } \Int(Y_S),
$$ as well as $v|_{(-S,S)\times \d M}=0$. Hence, Equation~\eqref{e:loc-interp} gives
\begin{equation}
\label{e:loc-interp-ter}
  \|v\|_{L^2 (Y_\beta)}^2 \leq C \left( e^{ 2S\lambda}\|(-\Delta_g-\lambda^2) u\|_{L^2(M)}^2 + \|\psi v|_{\Sigma_\beta}\|_{L^2(\Sigma_\beta)}^2 +  \|\psi \d_\nu v|_{\Sigma_\beta}\|_{L^2(\Sigma_\beta)}^2  \right)^\delta \|v\|_{H^1(Y_S)}^{2(1-\delta)}
\end{equation}
and we estimate each remaining term. First, we have
$$
  \|v\|_{L^2 (Y_\beta)}^2 \geq C\frac{e^{2\beta\lambda}}{\lambda} \|u \|_{L^2(M)}^2 .
$$
Second, we write
\bna
 \|v\|_{H^1(Y_S)}^2 & = & \|\d_{s} v\|_{L^2(Y_S)}^2 + \|\nabla_gv\|_{L^2(Y_S)}^2+ \|v\|_{L^2(Y_S)}^2 \\
&=& \|u\|_{L^2(M)}^2\int_{-S}^S  2\lambda^2 e^{2\lambda s} +e^{2\lambda s} ds +( (-\Delta_g-\lambda ^2)u,u )_{L^2(M)}  \int_{-S}^S e^{2\lambda s}ds\\
&\leq& Ce^{c\lambda}(\|u\|_{L^2(M)}^2+\|(-\Delta_g -\lambda^2)u\|_{L^2(M)}\|u\|_{L^2(M)}).
\ena
We may assume that $\|(-\Delta_g-\lambda^2)u\|_{L^2(M)}\leq \|u\|_{L^2(M)}$ since otherwise the inequality~\eqref{e:unique} holds trivially, and therefore obtain
$$
\|v\|_{H^1(Y_S)}^2\leq Ce^{c\lambda}\|u\|_{L^2(M)}^2.
$$
Third, we have 
\bna
 \|\psi v|_{\Sigma_\beta}\|_{L^2(\Sigma_\beta)}^2  +  \|\psi \d_\nu v|_{\Sigma_\beta}\|_{L^2(\Sigma_\beta)}^2
 & \leq &  \int_{-S}^{S}  e^{2\lambda s} \left( \| u |_{\Sigma}\|_{L^2(\Sigma)}^2 + \|\d_\nu u |_{\Sigma}\|_{L^2(\Sigma)}^2 \right) ds \\
 & \leq &2S  e^{2\lambda S} \left( \| u |_{\Sigma}\|_{L^2(\Sigma)}^2 + \|\d_\nu u |_{\Sigma}\|_{L^2(\Sigma)}^2 \right) .
\ena
Plugging the above three inequalities in~\eqref{e:loc-interp-ter} and dividing by $\|u \|_{L^2(M)}^{2(1-\delta)}$ (if non zero) yields the sought result.
\end{proof}

%%%%%%%%%%%%%%%%%%%%%%%%%%%%%%%%%%%%%%%
\subsection{From interpolation inequality to observability in an abstract setting: the original Lebeau-Robbiano method revisited}
\label{s:LR95-revisited}
In this section, we explain how to deduce the observability estimate for the heat equation from the interpolation inequality of Theorem~\ref{t:global-interp}. This follows the Lebeau-Robbiano method introduced in~\cite{LR:95} in its original form (used also in~\cite{Lea:10}), as opposed to the simplified version (see e.g.~\cite{LZ:98,LeLe:09}) which uses the stronger spectral inequality~\cite{JL:99,LZ:98} (which we do not prove in the present context). We explain how this method can be simplified using~\cite{Miller:10,EZ:11s,EZ:11}.

We consider an abstract setting containing the above particular situation of the heat equation. Most results presented here still hold in the much more general abstract setting of~\cite{Miller:10}. 
In Section~\ref{s:heat-application} below, we explain how the problem of the heat equation controlled by a hypersurface is put in this general framework.

\medskip
We denote by $H$ (with norm $\|\cdot \|$) and $K$ (with norm $\|\cdot \|_K$) two Hilbert spaces, namely the state space and the observation space. We denote by $A : D(A) \subset H \to H$ a non-positive selfadjoint operator on $H$, with compact resolvent.
We denote by $(\phi_j)$ an orthonormal basis of eigenfunctions associated to the eigenvalues $\lambda_j^2 \geq 0$ of $-A$ (we keep the notation used for the Laplace operator) and set
\bnan
\label{e:Elambda}
E_\lambda:= \vect\{\phi_j , \lambda_j \leq \lambda\},  \quad \lambda>0 .
\enan

The operator $A$ generates a contraction semigroup $(e^{tA})$ on $H$. 
We denote by $B \in \L (D(A) ; K)$ the observation operator. We say that $B$ is an admissible observation for $(e^{tA})$ if there is $T>0$ and $C_{adm,T}>0$ such that
\bnan
\label{e:admissibility}
\|B e^{tA}y\|_{L^2((0,T), K)} \leq C_{adm,T}\|y\| , \quad \text{ for all }y \in D(A) .
\enan
On account of the semigroup property,~\eqref{e:admissibility} holds for all $T>0$ if it holds for some $T$ (see~\cite[Section~2.3]{Cor:book}). Hence, under the above admissibility assumption, for any $T>0$, the map $u_0 \mapsto (t \mapsto B e^{tA}u_0 )$ extends uniquely as a continuous linear map $H \to L^2(0,T ; K )$,
which we shall still denote $B e^{tA}$.

\bigskip
In our next lemma, we use the notation, for $s \in \N$ and $ \tau >0$,
$$
\mc{H}^s_\tau = \bigcap_{n=0}^s H^{s-n}\left( - \tau , \tau ; \mathcal{D}((-A)^{n/2})\right)  ,
$$
normed by 
$$  \|v\|_{\mathcal{H}^s_\tau} = \bigg( \sum_{n+m \leq s} \| \partial_t^{m} (I-A)^{n/2} v\|_{L^2\left( - \tau, \tau ; H \right)}^2 \bigg)^{1/2} .
$$

\begin{lemma}
\label{l:interpolation-elliptic}
Let $S>\beta>0$ and $\varphi \in C^\infty_c(-S,S)$. Assume there is $C>0$ and $\delta \in (0,1)$ such that for all $v\in \mc{H}^2_S$, we have
\begin{equation}
\label{e:interp-abstract}
\|v\|_{\mc{H}^1_{\beta}} \leq C \left(\|(-\d_s^2-A)v\|_{\mc{H}^0_S} + \|\varphi(s)Bv \|_{L^2(-S,S; K)}  \right)^\delta \|v\|_{\mc{H}^1_S}^{1-\delta}.
\end{equation}
Then, there exists $S,C,c >0$ such that
\bna
\|v_0\|^2 +\|v_1\|^2 \leq Ce^{c \lambda} \left\|\varphi(s)Bv(s) \right\|_{L^2(-S,S; K)}^2 , \quad \text{for all } \lambda>0 , (v_0 , v_1)\in E_\lambda \times E_\lambda ,
\ena
with 
\bnan
\label{e:v-cosh-sinh}
v(s) = \cosh(s \sqrt{-A})v_0 +  \frac{\sinh(s \sqrt{-A})}{\sqrt{-A}} v_1 .
\enan
\end{lemma}
Note that in the formula~\eqref{e:v-cosh-sinh}, we extend $\cosh(s \sqrt{-A})$ (resp. $\frac{\sinh(s \sqrt{-A})}{\sqrt{-A}}$) by continuity by $\id$ (resp. by $s\id$) on $\ker(A)$. Thus, denoting by $\Pi_0$ the orthogonal projector on $\ker(A)$ and $\Pi_+ = \id-\Pi_0$,~\eqref{e:v-cosh-sinh} can be rewritten more explicitely by
$$
v(s) = \cosh(s \sqrt{-A})\Pi_+ v_0  + \Pi_0 v_0 +  \frac{\sinh(s \sqrt{-A})}{\sqrt{-A}} \Pi_+ v_1 + s \Pi_0 v_1 .
$$
Hence $v(s)$ in~\eqref{e:v-cosh-sinh} is the unique solution to 
$$
(-\d_s^2-A)v = 0 , \quad (v, \d_s v)|_{s=0} = (v_0 , v_1) .
$$

\begin{proof}[Proof of Lemma~\ref{l:interpolation-elliptic}]
Note first that with $v$ in~\eqref{e:v-cosh-sinh}, we have $(-\d_s^2-A)v(s)=0$ so that, in \eqref{e:interp-abstract}, it suffices to estimate $\|v\|_{\mc{H}^1_S}$ from above and $\|v\|_{\mc{H}^1_{\beta}}$ from below.
For $(v_0,v_1)\in E_\lambda \times E_\lambda$, we denote by $w_k = \Pi_0 v_k$, $k=0,1$, and $w^\pm = \frac12 (\Pi_+ v_0 \pm (-A)^{-1/2}\Pi_+ v_1)$. This is 
$$
\Pi_+ v_0 = w^+ + w^- , \quad \Pi_+ v_1 =\sqrt{-A} w^+ - \sqrt{-A}  w^-,  
$$
and the parallelogram law yields
$$
\|(-A)^{\frac{k}{2}} \Pi_+ v_0\|^2 + \|(- A)^{\frac{k-1}{2}} \Pi_+ v_1\|^2 = 2( \|(-A)^{\frac{k}{2}} w^+\|^2 + \|(- A)^{\frac{k}{2}} w^- \|^2). 
$$
We also have, with $w^\pm = \sum_{0<\lambda_j \leq \lambda}w_j^\pm \phi_j$,  
\begin{align*}
v(s) & = \cosh(s \sqrt{-A})v_0 + \frac{\sinh(s \sqrt{-A})}{\sqrt{-A}} v_1 = e^{s\sqrt{-A}}w^+ + e^{-s\sqrt{-A}}w^- + w_0+ sw_1 \\
& = \sum_{0<\lambda_j \leq \lambda} (e^{s \lambda_j} w_j^+ + e^{-s \lambda_j} w_j^-) \phi_j + w_0+ sw_1 .
\end{align*}
 Now, we estimate $\|v\|_{\mc{H}^1_S}$ and $\|v\|_{\mc{H}^1_{\beta}}$ in terms of $\lambda$. Firstly, we have
\begin{align*}
  \|v\|_{\mc{H}^1_{\beta}}^2 \geq \|v\|_{\mc{H}^0_{\beta}}^2&  = \sum_{0<\lambda_j \leq \lambda}  \int_{-\beta}^{\beta} \left| e^{s \lambda_j} w_j^+ + e^{-s \lambda_j} w_j^- \right|^2 ds + \int_{-\beta}^{\beta} \left\| w_0+ sw_1\right\|^2 ds \\
   &  = \sum_{0<\lambda_j \leq \lambda}  \frac{e^{2 \beta \lambda_j} -e^{-2 \beta \lambda_j} }{2\lambda_j} \left(|w_j^+|^2 +  |w_j^-|^2 \right) + 4 \beta \Re(w_j^+ \overline{w_j^-}) 
   + 2\beta \|w_0\|^2 + \frac23 \beta^3 \|w_1\|^2 \\
&   =  2 \beta \sum_{0<\lambda_j \leq \lambda} \mathcal{Q}_j \big((w_j^+, w_j^-),(\overline{w_j^+}, \overline{w_j^-}) \big)+  2\beta \|w_0\|^2 + \frac23 \beta^3 \|w_1\|^2,
  \end{align*}
where  $\mathcal{Q}_j$ is the matrix
$$
\mathcal{Q}_j =
\left( 
\begin{array}{cc}
\frac{\sinh(X_j)}{X_j}& 1\\
1 & \frac{\sinh(X_j)}{X_j}
\end{array}
\right), 
\quad X_j= 2 \beta \lambda_j.
$$
The eigenvalues of $\mathcal{Q}_j$ are $\frac{\sinh(X_j)}{X_j} \pm1 \geq \eps e^{X_j/2} $ on the set $[2 \beta \tilde\lambda_0 , +\infty[$, where $\tilde\lambda_0$ is the first non-zero eigenvalue of $-A$, and $\eps$ only depends on $2 \beta \tilde\lambda_0$. As a consequence, we obtain
\begin{align*}
 \|v\|_{\mc{H}^1_{\beta}}^2\geq C \left(
\|v_0\|^2 + \|v_1\|^2   \right) .
\end{align*}
Secondly, we also have
\bna
  \|v\|_{\mc{H}^1_{S}}^2 & = & \int_{-S}^S \|\d_s v\|^2 + \|(I-A)^{\frac12} v\|^2+ \|v\|^2  ds\\
& \leq  &\sum_{0< \lambda_j\leq \lambda}(|w_j^+|^2 + |w_j^-|^2) \int_{-S}^{S} {(2\lambda_j^2 + 4 )}e^{2s\lambda_j}  ds
  \\
&& + 2\int_{-S}^{S} \left\| w_0+ sw_1\right\|^2 ds +\int_{-S}^{S} \left\| w_1\right\|^2 ds  \\
& \leq  & C e^{3S\lambda}  \left( \|v_0\|^2 + \|v_1\|^2  \right).
\ena
Combining the last two estimates together with~\eqref{e:interp-abstract} yields 
\begin{align*}
 \|v_0\|^2 + \|v_1\|^2 \leq C \|\varphi(s)Bv \|_{L^2(-S,S; K)}^{2\delta}  \left(Ce^{3S\lambda} \big( \|v_0\|^2 + \|v_1\|^2 \big)\right)^{1-\delta} ,
\end{align*}
and hence the sought result when dividing by $\big(\|v_0\|^2 + \|v_1\|^2 \big)^{1-\delta}$ .
\end{proof}

The next step of the Lebeau-Robbiano method relies on a so-called ``transmutation argument'' to deduce from the observability of the elliptic system on $E_\lambda$ the observability of the heat equation on $E_\lambda$, with a precise estimate on the cost in terms of $\lambda$ and $T$ (observation time). Here, we use an idea of Ervedoza and Zuazua~\cite{EZ:11s,EZ:11} to simplify the original argument of Lebeau and Robbiano~\cite{LR:95} (who used the moment method of Russell to pass from the elliptic system to the parabolic system, and was quite technically involved, see~\cite{Lea:10} for a review of the method).

\begin{lemma}
\label{l:elliptic-parabolic}
Assume that there exists $S,C,c >0$ such that
\bna
\|v_0\| \leq Ce^{c \lambda} \left\|B \frac{\sinh(s \sqrt{-A})}{\sqrt{-A}} v_0 \right\|_{L^2(-S,S; K)} , \quad \text{for all } \lambda>0 , v_0 \in E_\lambda .
\ena
Then, there exist $C,c>0$ such that
\bna
\|e^{TA}u_0\|  \leq C e^{c\lambda+ \frac{c}{T}} \|Be^{tA}u_0\|_{L^2(0,T; K)} ,  \quad \text{for all } T>0,  \lambda>0 , u_0 \in E_\lambda .
\ena
\end{lemma}
Note that in the assumption of Lemma~\ref{l:elliptic-parabolic}, $\frac{\sinh(s \sqrt{-A})}{\sqrt{-A}}$ can equivalently be replaced by $\cosh(s \sqrt{-A})$. 

We need the following lemma, which is a slight variant on~\cite{EZ:11s,EZ:11}.
 \begin{lemma}
\label{l:EZ-elliptic}
Given $S,T>0$, $\delta \in (0,1)$, and $\alpha > S^2 \left( 1+ \frac{1}{\delta} \right)$,
there exists a function $k_T \in C^{\infty}([0,T]\times [-S,S])$ satisfying
 \bnan
 \label{e:heat-1-D}
 (\d_t -\d_{s}^2) k_T = 0 , \quad  \text{for }(t,s) \in (0,T) \times (-S,S), 
 \enan
\bnan
\label{e:kT-bord}
\begin{cases}
k_T|_{t=0} = 0, \quad  k_T|_{t=T} = 0, & \text{for }s \in  (-S,S), \\
k_T|_{s=0} =0 , \quad \d_s k_T|_{s=0} =e^{- \alpha \left(\frac{1}{t} + \frac{1}{T-t} \right)} , & \text{for }t \in  (0,T) ,
\end{cases}
\enan
\bnan
\label{e:estimEZ}
|k_T(t,s)| \leq |s| e^{\frac{1}{\tau} \left(\frac{s^2}{\delta} -\frac{\alpha}{1+\delta}\right)}, \quad \tau=\min(t, T-t) ,  \quad  \text{for }(t,s) \in (0,T) \times (-S,S).
\enan
 \end{lemma}
For the proof of Lemma~\ref{l:EZ-elliptic}, we follow \cite[Section~3.1]{EZ:11s}, where the authors go from the wave equation to the heat equation. Here, we use the method to go from an elliptic equation to heat equation. The only difference is that we take $g_{2k+1} = g_1^{(k)}$ where Ervedoza and Zuazua~\cite{EZ:11s,EZ:11} take $g_{2k+1} = (-1)^k g_1^{(k)}$ in the proof below.
 \bnp[Sketch of proof of Lemma~\ref{l:EZ-elliptic}]
 
 The starting point is that, if it converges, then the function
 \bnan
\label{e:def-kT}
 k_T(t,s) = \sum_{n \in \N} \frac{s^n}{n!} g_n(t) , \quad g_{2k} = g_0^{(k)} , \quad g_{2k+1} = g_1^{(k)}, \quad k \in \N ,
 \enan
 solves~\eqref{e:heat-1-D}.
 Choose $g_0=0$ and, for $\alpha>0$, choose $g_1$ to be the Gevrey function 
 $$
 g_1(t) = 
 \begin{cases}
e^{- \alpha \left(\frac{1}{t} + \frac{1}{T-t} \right)} & \text{if } t\in (0,T), \\
0 & \text{otherwise}.  
 \end{cases}
 $$
Then, \cite[Lemma~3.1]{EZ:11} yields for all $\delta \in(0,1)$, $|g_{2k+1}(t)| = |g_1^{(k)}(t)|\leq \frac{k!}{(\delta\tau)^k}e^{-\frac{\alpha}{(1+\delta)\tau}}$ with $\tau=\min(t, T-t)$. This implies (see~\cite[Equation~(3.8)]{EZ:11s}) that for all $\delta \in (0,1)$, $S>0$ and $\alpha > S^2 \left( 1+ \frac{1}{\delta} \right)$, the series~\eqref{e:def-kT} converges towards $k_T \in C^{\infty}([0,T]\times [-S,S])$ with~\eqref{e:estimEZ}-\eqref{e:kT-bord}.
\enp

With this lemma, the proof of Lemma~\ref{l:elliptic-parabolic} follows~\cite[Section~3.1]{EZ:11s}.
 \begin{proof}[Proof of Lemma~\ref{l:elliptic-parabolic}]
We first pick $\delta \in (0,1)$, and $\alpha > S^2 \left( 1+ \frac{1}{\delta} \right)$, and denote by $k_T$ the kernel then furnished by Lemma~\ref{e:def-kT}.
Given $u_0 \in E_\lambda$, we define 
$$
v(s):= \int_0^T k_T(t,s) e^{tA}u_0 dt .
$$
From the above properties of $k_T$, the function $v(s)$ satisfies 
$$
(v, \d_s v)|_{s=0} = \left( 0, \int_0^Tg_1(t) e^{tA}u_0 dt \right) \in E_\lambda \times E_\lambda ,
$$
where $g_1(t) = e^{- \alpha \left(\frac{1}{t} + \frac{1}{T-t} \right)}$, together with 
\begin{align*}
\d_s^2 v(s) &=   \int_0^T \d_s^2 k_T(t,s) u(t) dt =  \int_0^T \d_t k_T(t,s) e^{tA}u_0 dt = - \int_0^T k_T(t,s) \d_t e^{tA}u_0 dt \\
& = - \int_0^T k_T(t,s) A e^{tA}u_0 dt = - A \left(\int_0^T k_T(t,s) e^{tA}u_0 dt \right) = - A v(s) .
\end{align*}
Hence, $v(s) = \frac{\sinh(s \sqrt{-A})}{\sqrt{-A}} \left( \int_0^Tg_1(t) e^{tA}u_0 dt \right)$, so that Lemma~\ref{l:elliptic-parabolic} yields the estimate
$$
\left\| \int_0^Tg_1(t) u(t) dt \right\| \leq C e^{c\lambda} \|B v(s)\|_{L^2(-S,S;K)} .
$$
Now, writing $u_0 = \sum_j \alpha_j \phi_j$, we have
\begin{align*}
\left\| \int_0^Tg_1(t) e^{tA} u_0 dt \right\|^2 & = \sum_j \left| \int_0^Tg_1(t) e^{-t\lambda_j^2} \alpha_j dt \right|^2 \\
&  \geq  \sum_j \left( \int_0^Tg_1(t) dt \right)^2 e^{-2T\lambda_j^2} |\alpha_j |^2
= \left( \int_0^Tg_1(t) dt \right)^2 \left\| e^{TA} u_0\right\|^2
 \geq \frac{T^2}{9}e^{-\frac{9\alpha}{T}}\left\| e^{TA} u_0\right\|^2.
\end{align*}
Also, we have from~\eqref{e:estimEZ} the estimate
\bna
 \|B v(s)\|_{L^2(-S,S;K)}^2
  & = &\int_{-S}^S \left\| \int_0^T k_T(t,s) B e^{tA} u_0 dt \right\|_{K}^2 ds \\
& \leq & \left( \int_{]0,T[\times ]-S,S[} k_T(t,s)^2 dtds \right) \int_{0}^T \|  B e^{tA} u_0 dt \|_{K}^2dt \\
& \leq & C_S T  \int_{0}^T \|  B e^{tA} u_0 dt \|_{K}^2dt .
\ena
Combining the last three estimates concludes the proof of Lemma~\ref{l:elliptic-parabolic}.
 \end{proof}

From the low frequency observability estimate with precise cost, we may now deduce the full observability estimate. The original Lebeau-Robbiano strategy~\cite{LR:95} does not provide with an optimal dependance on the blow-up of the constant as $T\to 0^+$. The modified and simplified argument of~\cite{Miller:10} does so, and we follow it here.

\begin{lemma}
\label{e:BF-HF-abstract}
Assume $B : D(A) \subset H \to K$ is an admissible observation for $(e^{tA})$. Assume for some $a_0,a,b>0$ we have 
\bnan
\label{e:obs-LF}
\|e^{TA} y\| \leq a_0 e^{a \lambda + \frac{b}{T}}\|B e^{tA}y\|_{L^2(0,T; K)}, \quad  \text{for all } y \in E_\lambda , \lambda>0 , T >0 .
\enan
Then there is $C,c>0$ such that we have 
\bna
\|e^{TA} y\| \leq C e^{\frac{c}{T}}\|B e^{tA}y\|_{L^2(0,T; K)}, \quad   \text{for all }  y \in H , T>0.
\ena
\end{lemma}
A proof of this lemma (in much more generality) is included in the proof of~\cite[Theorem~2.1]{Miller:10}, but we give it for the sake of readability. The key feature of the semigroup $(e^{tA})$ we shall use is that 
\bnan
\label{e:decay-HF}
\|e^{tA} y\|_H \leq  e^{-\lambda^2 t}\|y\|_H, \quad \text{for all } y \in E_\lambda^\perp , \lambda>0 , t>0 .
\enan
We also make use of the following particular case of \cite[Lemma~2.1]{Miller:10}. 
\begin{lemma}
\label{l:Miller-2.1}
Let $T_* >0$, $q \in (0,1)$ and $f : (0,T_*] \to \R_+^*$ increasing, such that $\lim_{t\to 0^+} f(t) = 0$. Assume that $B$ is an admissible observation for $(e^{tA})$ and that 
\bna
f(T)\|e^{TA} y \|^2 - f(q T)\| y \|^2 \leq \|B e^{tA}y\|_{L^2(0,T;K)}^2 ,\quad \text{for all $T\in (0,T_*)$ and $y \in H$}.
\ena
Then we have
\bna
f((1-q)T)\|e^{TA} y \|^2  \leq \|B e^{tA}y\|_{L^2(0,T;K)}^2 ,\quad \text{for all $T\in (0,T_*)$ and $y \in H$}.
\ena
\end{lemma}

\bnp[Proof of Lemma~\ref{e:BF-HF-abstract}]
For $y \in H$, we decompose $y = y_\lambda + r_\lambda$ with $y_\lambda \in E_\lambda$ and $r_\lambda \in E_\lambda^\perp$. Then, we estimate
\bnan
\label{e:decomp-orth}
\|e^{TA} y\| \leq \|e^{TA} y_\lambda \| + \|e^{TA} r_\lambda \|.
\enan
Concerning the second term in~\eqref{e:decomp-orth}, we only use~\eqref{e:decay-HF} to write
\bna
 \|e^{TA} r_\lambda \| \leq  e^{-\lambda^2 T} \| r_\lambda \| \leq  e^{-\lambda^2 T} \| y \| . 
\ena
Concerning the first term in \eqref{e:decomp-orth}, we write $e^{TA} = e^{\eps TA} e^{(1-\eps)TA}$ and apply~\eqref{e:obs-LF} to $e^{(1-\eps)TA}y_\lambda \in E_\lambda$ to obtain
\bna
 \|e^{TA} y_\lambda \| & \leq & a_0 e^{a \lambda + \frac{b}{\eps T}}\|B e^{tA} e^{(1-\eps)TA}y_\lambda \|_{L^2(0,\eps T ; K)} \\
 &  \leq & a_0 e^{a \lambda + \frac{b}{\eps T}} \big( \|B e^{tA} e^{(1-\eps)TA} y \|_{L^2(0,\eps T ; K)} +\|B e^{tA}  e^{(1-\eps)TA} r_\lambda \|_{L^2(0,\eps T; K)} \big) .
\ena
We remark  that $\|B e^{tA} e^{(1-\eps)TA} y \|_{L^2(0,\eps T ; K)} = \|B e^{tA}y \|_{L^2((1-\eps) T , T; K)} \leq  \|B e^{tA}y \|_{L^2(0, T; K)}$
and estimate the last term using~\eqref{e:admissibility}, and then~\eqref{e:decay-HF} as
\bna
\|B e^{tA}  e^{(1-\eps)TA} r_\lambda \|_{L^2(0,\eps T; K)}&  \leq & C_{adm, \eps T} \|  e^{(1-\eps)TA}r_\lambda \|
 \leq C_{adm, \eps T}  e^{-\lambda^2(1-\eps)T} \|r_\lambda \|  \\
  & \leq & C_{adm, \eps T}  e^{-\lambda^2(1-\eps)T} \|y \| .
\ena
Combining the above three estimates in~\eqref{e:decomp-orth} implies for all $y \in H$, $T>0$ and $\lambda>0$,
\bna
 \|e^{TA} y \| \leq a_0 e^{a \lambda + \frac{b}{\eps T}}  \|B e^{tA}y \|_{L^2(0 , T; K)} +  e^{-\lambda^2(1-\eps)T} \left( a_0 e^{a \lambda + \frac{b}{\eps T}} C_{adm, \eps T}  + e^{-\eps\lambda^2 T} \right) \| y \| . 
\ena
We notice that $C_{adm, \eps T} \leq C_{adm, T_*}$ for $T\leq T_*$ and $\eps \in (0,1)$, and denote $m_1 := C_{adm, T_*} + \frac{1}{a_0}$. 
We then rewrite this estimate for $\lambda= \frac{1}{rT}$, with $r>0$ to be chosen, as 
\bna
\frac{1}{a_0} e^{-\frac{1}{T}\left( \frac{a}{r}+ \frac{b}{\eps}\right)}  \|e^{TA} y \| \leq  \|B e^{tA}y \|_{L^2(0 , T; K)} + m_1 e^{-\frac{1-\eps}{r^2}\frac{1}{T}}  \| y \| , \quad T \leq T_* . 
\ena
Writing $f(T)= \frac{1}{2 a_0^2} e^{-\frac{2}{T}\left( \frac{a}{r}+ \frac{b}{\eps}\right)}$, and assuming the parameters $\eps\in (0,1), r>0, q\in (0,1)$ are such that
$$
\frac{1}{q}\left( \frac{a}{r}+ \frac{b}{\eps}\right) \leq  \frac{1-\eps}{r^2} ,
$$
(which we may, taking e.g. $\eps=q=1/2$ and $r$ sufficiently small)
we have $\left( m_1 e^{-\frac{1-\eps}{r^2}\frac{1}{T}} \right)^2 \leq f(qT)$ for $T \in (0,T']$ for some $T'\in (0,T_*]$, and we obtain 
\bna
f(T)  \|e^{TA} y \|^2  \leq  \|B e^{tA}y \|_{L^2(0 , T; K)}^2 +f(qT) \|y \|^2. 
\ena
Lemma~\ref{l:Miller-2.1} implies 
\bna
f((1-q)T)  \|e^{TA} y \|^2  \leq  \|B e^{tA}y \|_{L^2(0 , T; K)}^2 ,   \quad T \in (0,T'], y \in H, 
\ena
which is the sought result for $t\in (0,T']$. The case $T>T'$ follows from the boundedness of the semigroup and the case $T<T'$.
\enp

%%%%%%%%%%%%%%%%%%%%%%%%%%%%%%%%%%
\subsection{From interpolation inequality to the observability estimate for the heat equation}
\label{s:heat-application}
Let us now put the above context of the heat equation in the present abstract framework, and state the consequences of the above abstract setting. We have $H=H^1_0(M)$, $A = \Delta_D$ (the Dirichlet Laplacian) with $D(A)=\{ u \in H^3(M), u|_{\d M} =0,  \Delta_g u|_{\d M} = 0\}$. We also have $K = L^2(\Sigma) \times L^2(\Sigma)$ as well as 
$$
\begin{array}{rcl}
B : D(A) \subset H^3(M) &\to & L^2(\Sigma) \times L^2(\Sigma)\\
 u &\mapsto &(u |_{\Sigma} , \d_\nu u |_{\Sigma}) .
 \end{array}
$$
Lemma~\ref{l:admissibility} implies that $B$ is an admissible observation for $(e^{tA})$ in the sense of~\eqref{e:admissibility}.

The first lemma is a consequence of the interpolation inequality of Theorem~\ref{t:global-interp} and Lemma~\ref{l:interpolation-elliptic}. Here, $E_\lambda$ is defined by~\eqref{e:Elambda} where $\phi_j, \lambda_j$ are an orthonormal basis of solutions to 
$$
(-\Delta_g-\lambda_j^2)\phi_j=0.
$$

\begin{lemma}[observability of finite dimensional elliptic equation]
\label{l:obs-ell-lambda}
Assume $M$ is connected and $\Sigma$ is nonempy. Then, for all $S>0$, there exists $C,c>0$ such that for all $\lambda>0$, all $(v_0,v_1)\in E_\lambda \times E_\lambda$ and associated solution $v$ of
\begin{equation}
\label{e:elliptic-evolution}
\begin{cases}(-\d_s^2 -\Delta)v=0 &\text{on }(-S,S)\times \Int(M), \\
v=0 &\text{on }(-S,S)\times \d M, \\
(v,\partial_s v)|_{s=0}=(v_0,v_1)& \text{in } \Int(M) ,
\end{cases}
\end{equation}
we have 
$$
\|(v_0,v_1)\|_{H^2\times H^1} \leq C e^{c\lambda} \left(\| v|_{\Sigma}\|_{L^2((-S,S)\times\Sigma)} +\| \d_\nu v|_{\Sigma}\|_{L^2((-S,S)\times\Sigma)} \right) .
$$
\end{lemma}
This together with Lemma~\ref{l:elliptic-parabolic} this implies the following result.

\begin{lemma}[observability of finite dimensional heat equation with precise cost]
\label{l:obs-para-lambda}
Assume $M$ is connected and $\Sigma$ is nonempy. Then, there exists $C,c>0$ such that for all $\lambda, T>0$, all $u_0 \in E_\lambda$ and associated solution $u$ of
\begin{equation}
\label{e:heat-evolution}
\begin{cases}(\d_t -\Delta)u=0 &\text{on }(0,T)\times \Int(M) ,\\
u=0 &\text{on }(0,T)\times \d M ,\\
 u|_{t=0}= u_0& \text{in } \Int(M) ,
\end{cases}
\end{equation}
we have 
$$
\| u(T)\|_{H^1} \leq C e^{c\lambda+ \frac{c}{T}} \left(\| u|_{\Sigma}\|_{L^2((0,T)\times\Sigma)} +\| \d_\nu u|_{\Sigma}\|_{L^2((0,T)\times\Sigma)} \right) .
$$
\end{lemma}
Lemma~\ref{e:BF-HF-abstract} finally yields the following observability result.
\begin{theorem}[observability for heat equation]
\label{t:obs-heat}
Assume $M$ is connected and $\Sigma$ is nonempy. Then, there exist $C,c>0$ such that for all $T>0$, all $u_0 \in H^1(M)$ and associated solution $u$ of~\eqref{e:heat-evolution}, we have 
$$
\| u(T)\|_{H^1} \leq C e^{\frac{c}{T}} \left(\| u|_{\Sigma}\|_{L^2((0,T)\times\Sigma)} +\| \d_\nu u|_{\Sigma}\|_{L^2((0,T)\times\Sigma)} \right) .
$$
\end{theorem}
From this observability estimate and the duality with the control problem~\eqref{e:heat-control}, given by Definition~\ref{d:transp-sol-heat}, we deduce the null-controllability of the heat equation Theorem \ref{t:control-heat}. The proof is classical and we omit it (see e.g. \cite[Chapter~2.3]{Cor:book}).

%%%%%%%%
\appendix
\section{Facts and notations of pseudodifferential calculus}
\label{s:pseudo}
Here, we follow \cite[Section 1.1]{Burq:97b} or~\cite[Section~2.1]{DLRL:14} for the notation. 
We denote by $X$ an open set of a $d$ dimensional manifold, which, in the main part of the article, is, with $d = n-1, n ,n+1$, one of the following:
\bnan
\label{e:exX}
X = \R^d,
\quad X = \Int(M), 
\quad  X=(0,T)\times \Int(M), 
\quad X = \Int(\Sigma),
\quad X=(0,T)\times \Int(\Sigma),
\quad  X = \Int(\Sigma) .
\enan
We also denote by $x$ the variable in $X$ (whereas, in case $X=(0,T)\times \Int(M)$ the variable in denoted $(t,x)$ in the main part of the article). We denote by $\pi_0 : T^*X \to X$ the canonical projection.

{
We write $S_{\hom}^m(T^*X)$ for the set of positively homogeneous degree $m$ functions on $T^*X$ with {\em compact support} in $X$. That is, $a\in S_{\hom}^m(T^*X)$ if and only if $a\in C^\infty(T^*X)$, $\pi_0(\supp(a))$ is a compact of $X$, and there is $R>0$ (depending on $a$) such that for $(x,\xi) \in T^*X$, with $|\xi| \geq R$ $\lambda \geq 1$, we have $a(x,\lambda \xi)=\lambda^m a(x,\xi).$
For any $m$, the restriction to the sphere $S^*X = T^*M/\R^+_*$
\begin{equation}
  \label{eq: identification symbols}
  S_{\phg}^{m}(T^*X) \to C_c^\infty(S^*X), \quad a(x, \xi) \to \lim_{\lambda \to \infty} \lambda^{-m} a(x, \lambda \xi ),
\end{equation}
is onto, which identifies, for $m$ fixed, smooth functions on the sphere with {\em homogeneous} symbols of degree $m$.

We also write $S_{\phg}^{m}(T^* X )$ for the set of polyhomogeneous symbols of order $m$
on $X$ with {\em compact support} in $X$. That is, $a \in S_{\phg}^{m}(T^* X)$ if and only if $a\in C^\infty(T^*X)$, $\pi_0(\supp(a))$ is a compact of $X$, and there exist $a_j \in S^{m-j}_{\hom}(T^* X)$, such that for all $N\in \N$, $a- \sum_{j=0}^N a_j \in S^{m-N-1}(T^*X)$.
We recall that symbols
in the class $S_{\phg}^{m}(T^* \R^d)$ behave well
with respect to changes of variables, up to symbols in
$S_{\phg}^{m-1}(T^* \R^d)$ (see \cite[Theorem~18.1.17 and Lemma~18.1.18]{Hoermander:V3}). }

We denote by $\Psi_{\phg}^{m}(X)$ the space of polyhomogeneous pseudodifferential operators of order
$m$ on $X$, with a {\em compactly supported kernel} in $X \times X$: one says that $A \in \Psi_{\phg}^{m}(X)$ if 
\begin{enumerate}
  \item its kernel $K(x,y) \in \D'(X\times X)$ is such that 
    $\supp (K)$ is compact in $X\times X$;
  \item $K(x,y)$ is smooth away from the diagonal $\Delta_{X} = \{
    (x,x);\ x \in X\}$;
  \item for every coordinate patch $X_{\kappa} \subset X$ with
    coordinates $X_{\kappa} \ni x \mapsto \kappa(x) \in
    \tilde{X}_{\kappa} \subset \R^{d}$ and all $\phi_0$, $\phi_1
    \in C_c^\infty(X_{\kappa})$ the map 
    $$
    u \mapsto \phi_1 \big(\kappa^{-1}\big)^\ast A \kappa^\ast (\phi_0 u)
    $$
    is in $\Op(S_{\phg}^{m}(\R^d \times\R^d))$. Note that for $a\in S_{\phg}^{m}(\R^d\times \R^d)$ we write $\Op(a)$ for the standard quantization of $a$.
\end{enumerate}

In case $X$ is not compact (which happens in most examples of~\eqref{e:exX}), we also define a non-canonical quantization procedure $\Op:S^m_{\phg}(T^*X)\to \Psi_{\phg}^m(X)$. For this, fix $\chi_n\in C_c^\infty(X;[0,1])$ so that $\chi_n(x)\uparrow 1$ for all $x\in X$ uniformly on compact sets. Then fix $(X_i,\kappa_i)$ a coordinate atlas for $X$. Let $\psi_i\in C_c^\infty(X)$ be a partition of unity subordinate to $X_i$ and $\tilde{\psi}_i\in C_c^\infty(X_i)$ with $\supp \psi_i\subset \{\tilde{\psi}_i\equiv 1\}$. For $a\in S^m_{\phg}(X)$, notice that $a_i(x,\xi):=\psi_i (\kappa^{-1}_i(x))a(\kappa_i(x),([\partial \kappa_i^{-1}(x)]^{-1})^t\xi)$ has $a_i\in S^m_{\phg}(\R^d\times \R^d)$. We then define 
$$
\Op(a) =\sum_i \chi_N \kappa_i^*\big[\big((\kappa_i^{-1})^*\tilde{\psi}_i\big)\Op_i(a_i)(\kappa_i^{-1})^*(\tilde{\psi}_i\chi_Nu)\big],\qquad N:=\inf\{ n\mid \supp a\cap \supp (1-\chi_n)=\emptyset\}.
$$
Note that for all $A\in \Psi^m_{\phg}(X)$, there exists $a\in S^m_{\phg}(T^*X)$ so that
$$
\Op(a)- A = R\in \Psi^{-\infty}_{\phg}(X)
$$
(see e.g. \cite[Appendix E]{ZwScat}).

For $A \in \Psi_{\phg}^{m}(X)$, we denote by $\sigma_m(A) \in S_{\hom}^{m}(T^*X)$ the
principal symbol of $A$ (see~\cite[Chapter 18.1]{Hoermander:V3}). Note
that the principal symbol is uniquely defined in $S^m_{\hom}(T^*X)$ because of the polyhomogeneous structure (see the
remark following Definition~18.1.20 in~\cite{Hoermander:V3}). When it will not lead to confusion, we abuse notation slightly and write $\sigma(A)$ for the principal symbol of a pseudodifferential operator without reference to the order. The applications $\sigma_m$ and $\Op$ enjoy the following properties:
\begin{itemize}
\item The sequence
$$
0\to \Psi^{m-1}_{\phg}(T^*X)\to\Psi^m_{\phg}(X)\overset{\sigma_m}{\longrightarrow}S^m_{\hom}(T^*X)\to 0
$$
is exact.
\item $\sigma_m\circ \Op:S^m_{\phg}(T^*X)\to S_{\hom}^m(T^*X)$ is the natural projection map.
\item For all $A \in \Psi_{\phg}^{m}(X)$,
  $\sigma_m(A^*) = \ovl{\sigma_m(A)}$.
\item For all $A_1 \in \Psi_{\phg}^{m_1}(X)$ and
  $A_2 \in \Psi_{\phg}^{m_2}(X)$, we have $A_1 A_2
  \in \Psi_{\phg}^{m_1+m_2}(X)$ with
  $$
  \sigma_{m_1 +m_2}(A_1 A_2) = \sigma_{m_1}(A_1)\sigma_{m_2}(A_2).
  $$
\item For all $A_1 \in \Psi_{\phg}^{m_1}(X)$ and $A_2 \in
  \Psi_{\phg}^{m_2}(X)$, we have $[A_1 ,A_2] = A_1 A_2 - A_2
  A_1 \in \Psi_{\phg}^{m_1+m_2 - 1}(X)$ with
  $$
  \sigma_{m_1 +m_2 - 1}([A_1 ,A_2] ) = \frac{1}{i} \{ \sigma_{m_1}(A_1) , \sigma_{m_2}(A_2) \}.
  $$
  Here, $\{ a_1 , a_2 \}$ denotes the Poisson bracket, given in local charts by
  $$
  \{ a_1 , a_2 \} = 
  \sum_l ( \d_{\xi_l}a_1 \d_{x_l}a_2 - \d_{x_l}a_1 \d_{\xi_l}a_2 ) .
  $$
\item If $A \in \Psi_{\phg}^{m}(X)$, then $A$ maps
  continuously $H^k_{\loc}(X)$ into $H^{k-m}_{\comp}(X)$. In
  particular, for $m < 0$, $A$ is compact on $H^k(X)$.
\end{itemize}

\noindent Given an operator $A \in \Psi^m_{\phg}(X)$, we define $\Char(A)
= \{\rho \in T^*X \setminus 0 , \sigma_m(A)(\rho) = 0\}$ its characteristic set and 
$$
\Ell (A) = (T^*X\setminus 0) \setminus \Char(A) 
$$
its elliptic set.

\medskip
We define the wavefront set of an operator $A\in \Psi^m_{\phg}(X)$, denoted by $\WF(A)$ as follows (see~\cite[Proposition~18.1.26 p88]{Hoermander:V3}).  We say $(x_0,\xi_0)\in T^*X$ is not in $\WF(A)$ if there exists $B\in \Psi^0_{\phg}(X)$ with $\sigma_0(B)(x_0,\xi_0)=1$ and 
$$BA:\mc{D}'(X)\to C_c^\infty(X).$$
Note that in local coordinates, the wavefront set this is given by the support of the \emph{full} symbol of $A$ (seen as a subset of $S^*\R^d$).

\medskip
{
Also, in the main part of the article, we use so-called ``tangential'' symbols, pseudodifferential operators and pseudodifferential calculus. We write $a \in C^\infty((-\e,\e);S^m_{\phg}(T^*\R^d))$ if $a = a (x_1 , x', \xi')$ is a smooth $x_1$-dependent family of symbols in the $(x', \xi')$ variables. We write $A \in C^\infty((-\e,\e);\Psi^m_{\phg}(\R^d))$ for the associated operators. The rules of pseudodifferential calculus are then as above.
}

\medskip
Finally, in the main part of the article, we use estimates for the hyperbolic Cauchy problem.
We state the following Lemma from H\"ormander \cite[Lemma 23.1.1]{Hoermander:V3}.
\begin{lemma}
\label{l:energy}
Let $\eps>0$, suppose that $\lambda=\lambda(x_1,x',\xi')\in C^\infty((-\e,\e);S^1_{\phg}(T^*\R^d))$ is real valued and write $\Lambda=\Op(\lambda)$.
Then for all $s\in \R$, there exists $C>0$ so that for $x_1,y_1\in (-\e,\e)$ and all $u,f$ solutions of 
\begin{align*} 
(D_{x_1}-\Lambda)&u=f ,
 \end{align*}
 we have
$$\|u(x_1,\cdot)\|_{H^s(\R^{d})}\leq C(\|u(y_1,\cdot)\|_{H^s( \R^{d})}+\|f\|_{L^2((-\e,\e);H^s( \R^{d}))})$$
and moreover
$$\|u(x_1,\cdot)\|_{H^s(\R^{d})}\leq C(\|u\|_{L^2((-\e,\e);H^s(  \R^{d}))}+\|f\|_{L^2((-\e,\e);H^s( \R^{d}))}).$$
\end{lemma}
Note that the second estimate is obtained from the first one by integrating in $y_1$.

%%%%%%%%%%%%%%%%%%%%%%
\section{Sharpness of Theorem \ref{t:expo-bound-eigenfunctions}: Proof of Proposition~\ref{prop:counterexample-revolution}}
\label{s:sharpUnique}
We start with an abstract simple lemma linking the symmetries of the manifold with that of solutions to related elliptic problems.
\begin{lemma}
\label{l:flip}
Let $(M,g)$ be a compact Riemannian manifold possibly with boundary and suppose that there is an isometric involution $j:M\to M$ (i.e. a diffeomorphism such that $j^*g = g$ and $j^2=\Id$) and a compact hypersurface $\Sigma \subset M$ such that 
$$M=M_+\sqcup M_- \sqcup \Sigma,\qquad \Fix(j)=:\Sigma$$
where 
$$\Fix(j):=\{x\in M\mid j(x)=x\}$$ 
and  $j(M_+)=M_-.$ 
Let $V\in C^\infty(M)$ such that $V\circ j=V$. Suppose that $u,\,v$ solve
 \begin{align*} 
 (-\Delta_g+V)u&=0\,\text{ in }\Int M_+,&
 u|_{\Sigma}&=0,&u|_{\partial M}&=0,\\
  (-\Delta_g+V)v&=0\,\text{ in }\Int M_+,&
 \partial_\nu v|_{\Sigma}&=0,\,&v|_{\partial M}&=0.
 \end{align*}
 Then, 
 $$u_o:=\begin{cases}u(x)&x\in M_+\cup \Sigma ,\\
 -u(j(x))&x\in M_-,
 \end{cases}\qquad u_e:=\begin{cases}v(x)&x\in M_+\cup\Sigma ,\\
 v(j(x))&x\in M_- ,
 \end{cases}$$
 satisfy $u_o, u_e \in C^\infty(\overline{M})$ and solve
\begin{align*} 
(-\Delta_g+V)u_o&=0\,\text{ in }\Int M,&u_o|_{\partial M}&=0,& u_o|_{\Sigma}&=0 ,\\
(-\Delta_g+V)u_e&=0\,\text{ in }\Int M,& u_e|_{\partial M}&=0,& \partial_\nu u_e|_{\Sigma}&=0.
\end{align*}
\end{lemma}
\bnp
Notice first that $\d M_+ = \Sigma \sqcup (\d M \cap M_+)$ and by elliptic regularity, we have $u, v\in  C^\infty(\overline{M}_+)$. Moreover, if $w_\pm \in C^2(\overline{M}_\pm)$, then, in the distribution sense (with $\d_\nu$ pointing towards $M_+$)
$$(-\Delta_g+V)w(x)=\mathds{1}_{M_+}(-\Delta_g+V)w_+ +\mathds{1}_{M_-}(-\Delta_g+V)w_- - (w_+|_{\Sigma}-w_-|_{\Sigma})\delta_\Sigma' -(\partial_\nu w_+|_{\Sigma}-\partial_\nu w_-|_{\Sigma})\delta_\Sigma.$$
Hence, 
$$(-\Delta_g+V)u_e=(-\Delta_g+V)u_o=0$$
as distributions and by elliptic regularity, $u_e,u_o\in C^\infty$ and hence have the desired properties. 
\enp

We may now proceed to the proof of Proposition~\ref{prop:counterexample-revolution}.

\bnp[Proof of Proposition~\ref{prop:counterexample-revolution}]
The Riemannian volume element is $R(z) dz d\theta$ and the Laplace Beltrami operator is given by
$$
\Delta_g = \frac{1}{R(z)} \d_z R(z) \d_z  + \frac{1}{R(z)^2} \d_\theta^2 = \partial_z^2 + \frac{R'}{R}\partial_z + \frac{1}{R^2}\partial_\theta^2 . 
$$
The map
\begin{equation*}
\begin{array}{rcl}
T: L^2(M, R(z)dzd\theta) & \to & L^2(M, dzd\theta) \\
u  & \mapsto & Tu, \quad (Tu)(z,\theta) = R(z)^{\frac12} u(z,\theta)
\end{array}
\end{equation*}
is an isometry and the conjugated operator of $\Delta_g$ is
\bna
\tilde{\Delta} & = &T \Delta_g T^{-1} = R^{1/2}\Delta_gR^{-1/2} = \partial_z^2 + \frac{1}{4}\left(\frac{R'}{R}\right)^2 - \frac{1}{2}\frac{R''}{R} + \frac{1}{R^2}\partial_\theta^2 \\
& = &\partial_z^2 + \frac{1}{R(z)^2}\partial_\theta^2 - V_1(z) ,
\ena
where 
$$V_1(z)=-\frac{1}{4}\frac{R'(z)^2}{R(z)^2}+\frac{1}{2}\frac{R''(z)}{R(z)^2}$$
is a smooth $\theta$-independent potential on $M$.

We now construct eigenfunctions of $\tilde{\Delta}$ under the form $\tilde{\phi}_k (z,\theta)=e^{ik\theta}\psi_k(z)$. Setting $h= k^{-1}$, this amounts to find eigenfunctions of the operator 
$$
P_h :=- h^2 \partial_z^2 + \frac{1}{R(z)^2} + h^2 V_1(z)
$$
with Dirichlet boundary conditions on $\pm \pi$. We shall rather consider this operator on $(0,\pi)$, and then complete the construction by symmetry with Lemma~\ref{l:flip}.
Recall that the potential $V(z) = \frac{1}{R(z)^2}$ satisfies $V(0)=1, V(\pi/2)=1/5 , V(\pi)=2$. 
Denoting by $E^{\eo}_h$ the eigenvalues of $P_h$ on $(0,\pi)$ associated to Dirichlet on $\pi$ and Neumann on $0$ for $E_h^e$, resp. Dirichlet on $0$ for $E_h^o$. The Weyl law (see e.g. \cite[Corollary~9.7]{DS:book}, \cite[Theorem~6.8]{EZ:10} or \cite[Theorem 1.2.1]{SafVa}) implies 
$$
\sharp \{E_h^{\eo} , \frac{1}{5} \leq E_h^{\eo} \leq \frac12 \} \sim_{h \to 0^+} (2\pi h)^{-1} \left| \left\{ (z,\xi) \in (0,\pi) \times \R , \frac15 \leq \xi^2 + V(z) \leq \frac12 \right\}\right| ,
$$
so that, recalling the form of $V$, there is $h_0>0$ such that for $h \in (0, h_0)$, the set $\{E_h^{\eo} , \frac{1}{5} \leq E_h^{\eo} \leq \frac12 \}$ is nonempty. We pick such an eigenvalue $E_h^{\eo} \in [\frac15, \frac12]$, and denote $\psi_h^{\eo}$ an associated eigenfunction, i.e., which satisfies
$$
(P_h-E_h^{\eo} )\psi_h^{\eo} = 0
, \quad \| \psi_h^{\eo}\|_{L^2(0,\pi)}\neq 0 
, \quad \psi_h^{\eo}(\pi) = 0 
, \quad \d_z \psi_h^e (0) = 0 
, \quad \psi_h^o (0) = 0 .
$$
Applying now Lemma \ref{l:flip} on $M=[-\pi, \pi]$ with $j(z)=-z$, so that $\Fix(j)=\{z=0\}$ allows to extend $\psi_h^{\eo}$ by parity/imparity as solutions of 
$$
(P_h-E_h^{\eo} )\psi_h^{\eo} = 0 \text{ on }(-\pi, \pi)
, \quad \| \psi_h^{\eo}\|_{L^2(-\pi ,\pi)}\neq 0 
, \quad \psi_h^{\eo}(\pm \pi) = 0 
, \quad \d_z \psi_h^e (0) = 0 
, \quad \psi_h^o (0) = 0 .
$$
Now, since $E_h^{\eo} \in [\frac15, \frac12]$ and $V(0) = \frac{1}{R(0)^2} =1>\frac12$, $0$ is in the classically forbidden region, so that classical Agmon estimates (see e.g. \cite[Chapter~6]{DS:book} or \cite[Chapter 6]{EZ:10}),
yield for $\e, h_0>0$ small enough, for any $N>0$, the existence of $C,c>0$ such that for all $h \in (0,h_0)$
$$\|\psi_h^{\eo}\|_{H^N(-\e,\e)}\leq Ce^{-c/h}\|\psi_h^{\eo}\|_{L^2(-\pi,\pi)}.$$

Coming back to the variables $(z,\theta)$, setting $(\lambda_k^{\eo})^2 =k^2 E_{k^{-1}} \in k^2[\frac15, \frac12]$ and $\tilde{\phi}_k^{\eo} (z,\theta)=e^{ik\theta}\psi_{k^{-1}}^{\eo}(z)$, we have obtained for $k \geq k_0$ solutions to 
$$
\big(-\tilde\Delta - (\lambda_k^{\eo})^2\big)\tilde{\phi}_k^{\eo} = 0 , \quad \tilde{\phi}_k^{\eo} (\pm \pi)= 0 , \quad \|\tilde{\phi}_k^{\eo}\|_{L^2} \neq 0 ,
$$
together with 
$$\tilde{\phi}_k^{o}|_{\Sigma}=0,\qquad \|\partial_\nu \tilde{\phi}_k^{o}|_{\Sigma}\|_{L^2(\Sigma)}\leq Ce^{-c k }\|\tilde{\phi}_k^{o}\|_{L^2(M)}\leq Ce^{-c' \lambda_k^{o}}\|\tilde{\phi}_k^{o}\|_{L^2(M)}, $$
$$\partial_\nu \tilde{\phi}_k^{e}|_{\Sigma}=0,\qquad \| \tilde{\phi}_k^{e}|_{\Sigma}\|_{L^2(\Sigma)}\leq  Ce^{-c k }\|\tilde{\phi}_k^{e}\|_{L^2(M)}\leq Ce^{-c' \lambda_k^{e}}\|\tilde{\phi}_k^{e}\|_{L^2(M)}.$$

Setting $\phi_k^{\eo} (z, \theta) = R(z)^{-1/2}\tilde{\phi}_k^{\eo} (z, \theta) \|\tilde{\phi}_k^{\eo}\|_{L^2}^{-1}$ concludes the proof of the lemma.
\enp

\section{About $\mc{T}$GCC: Proof of Proposition~\ref{e:h1/4}}
\label{s:appendix-h1/4}
\begin{proof}[Proof of Proposition~\ref{e:h1/4}]
Here, $M=S^2$ and $\Sigma$ is an equator. We take the following coordinates on $S^2$:
$$[0,2\pi)\times [0,\pi]\ni (\theta,\phi)\mapsto (\cos \theta\sin\phi,\sin \theta\sin\phi,\cos \phi)\in S^2 ,$$
and let $\Sigma:=\{\phi=\frac{\pi}{2}\}$.

 Then an orthonormal basis of Laplace eigenfunctions is given by 
$$Y^m_l(\theta,\phi)=\left(\frac{(l-m)!(2l+1)}{4\pi(l+m)!}\right)^{1/2}e^{im\theta}P^m_l(\cos \phi),\quad\quad -l\leq m\leq l,$$
where $P^m_l$ is an associated Legendre function (see for example \cite[Section 14.30]{NIST}). For the definition of $P^m_l$ see \cite[Section 14.2]{NIST}. Note that 
$$(-\Delta_{S^2}-\lambda_l^2)Y^m_l=0,\quad\quad \lambda_l:=\sqrt{l(l+1)} \sim_{l\to \infty} l.$$

We take $\phi_l = Y_l^{l-1}$, and recall that $\Sigma=\{\phi=\frac{\pi}{2}\}$.
By \cite[Section 14.30ii and Section 14.5(i)]{NIST}, we have
$$\phi_l|_{\Sigma} = Y^{l-1}_l\left(\theta,\frac{\pi}{2}\right)=0 , $$
since $P^{l-1}_l(0) = 0$. Moreover, using~\cite[Equation~(14.5.2)]{NIST} together with the definition of $Y^{l-1}_l$, we have
\bnan
\label{e:calculus}
\left| \d_\nu \phi_l |_{\Sigma} \right| = \left|\partial_\phi Y^{l-1}_l\left(\theta,\frac{\pi}{2}\right)\right| =\left|\left(\frac{(2l+1)}{4\pi(2l-1)!}\right)^{1/2}\frac{2^{l}\pi^{1/2}}{\Gamma(-l+\frac{1}{2})}\right|\sim cl^{3/4}.
\enan
Indeed, note that for $l\geq 1$, 
$$\Gamma(\frac{1}{2}-l)= \frac{(-1)^l \pi}{\Gamma(l+ \frac12)}  =  \frac{(-1)^l \sqrt{\pi} 2^{2l}l! }{(2l)!} 
=\frac{2^{l}(-1)^l\sqrt{\pi}}{\prod_{j=1}^l(2j-1)}$$
and 
$$\frac{\prod_{j=1}^l(2j-1)^2}{(2 l-1)!}= \prod_{j=2}^l \frac{2j-1}{2j-2} =e^{\sum_{j=2}^l \log (1+\frac{1}{2j-2})}.$$
Then, note that 
$$\sum_{j=2}^l\log \Big(1+\frac{1}{2j-2}\Big)= \frac{1}{2}\log l +O(1)$$ 
and in particular, 
$$
\frac{\prod_{j=1}^l(2j-1)}{\sqrt{(2l-1)!}}\sim c l^{\frac{1}{4}}.
$$
The above four lines finally prove~\eqref{e:calculus}. Therefore, for $l$ large enough, we obtain
$$\lambda_l^{-1/4} \sim l^{-1/4}=l^{-1/4}\|Y^{l-1}_l\|_{L^2(S^2)}\geq c \|l^{-1}\partial_{\phi}Y^{l-1}_l\|_{L^2(\Sigma)},$$
which concludes the proof of the lemma.
\end{proof}

%%%%%%%%%%%%%%%%%%%%%%

\small
\bibliographystyle{alpha}     
\bibliography{bibli}

\end{document}

%% file: mymacros.tex
\newtheorem{lemma}{Lemma}[section]
\newtheorem{theorem}[lemma]{Theorem}
\newtheorem{proposition}[lemma]{Proposition}

\theoremstyle{definition}

\newtheorem{definition}[lemma]{Definition}
\newtheorem{remark}[lemma]{Remark}

\makeatletter
\def\keywords{
    \vspace{1ex}
    \noindent
    \if@twocolumn
      \small{\bf  Keywords}\/---$\!$    \else
      \begin{center}\small\ {\bf Keywords}\end{center}\quotation\small
    \fi}
\def\endkeywords{\vspace{0.6em}\par\if@twocolumn\else\endquotation\fi
    \normalsize\rm}
\makeatother

\newcommand{\G}{\ensuremath{\mathcal G}}

\newcommand{\calA}{\ensuremath{\mathcal A}}

\renewcommand{\L}{\ensuremath{\mathcal L}}
\newcommand{\calN}{\ensuremath{\mathcal N}}

\DeclareMathOperator{\phg}{phg}

\DeclareMathOperator{\loc}{loc}

\DeclareMathOperator{\comp}{comp}

\DeclareMathOperator{\Int}{Int}

\newcommand{\mb}[1]{\ensuremath{\mathbb{#1}}}
\newcommand{\N}{{\mb{N}}}

\newcommand{\R}{{\mb{R}}}
\newcommand{\C}{{\mb{C}}}

\newcommand{\eps}{\varepsilon}

\newcommand{\M}{\ensuremath{\mathcal M}}

\newcommand{\T}{\ensuremath{\mathbb T}}
\newcommand{\D}{\ensuremath{\mathscr D}}

\let \Re \relax
\DeclareMathOperator{\Re}{Re}
\let \Im \relax
\DeclareMathOperator{\Im}{Im}

\newcommand{\ovl}[1]{\overline{#1}}

\DeclareMathOperator{\supp}{supp}
\DeclareMathOperator{\WF}{WF}
\DeclareMathOperator{\Char}{Char}

\DeclareMathOperator{\vect}{span}

\DeclareMathOperator{\Op}{Op}

\DeclareMathOperator{\id}{Id}

\renewcommand{\d}{\ensuremath{\partial}}

\newcommand{\eo}{{\mbox{\scriptsize $^e\!\!\slash_{\!o}$}}}

\newcommand{\E}{\mathscr E}